\documentclass[10pt]{amsart}
\usepackage{setspace}
\usepackage{amsmath, amsfonts, amsthm, amstext, xspace}
\usepackage{amssymb,latexsym}
\usepackage{hyperref}
\usepackage{xcolor}
\definecolor{seagreen}{RGB}{46,139,87}
\definecolor{maroon}{RGB}{128,0,0}
\definecolor{darkviolet}{RGB}{148,0,211}
\definecolor{twelve}{RGB}{100,100,170}
\definecolor{thirteen}{RGB}{100,150,50}
\definecolor{fourteen}{RGB}{200,0,0}
\definecolor{fifteen}{RGB}{0,200,0}
\definecolor{sixteen}{RGB}{0,0,200}
\definecolor{seventeen}{RGB}{200,0,200}
\definecolor{eighteen}{RGB}{0,200,200}

\usepackage{mathrsfs}

\usepackage{lscape}

\usepackage[all]{xy}
\newtheorem{thmx}{Theorem}

\newtheorem{thm}{Theorem}[section]
\newtheorem*{thm*}{Theorem}
\newtheorem{lem}[thm]{Lemma}
\newtheorem{cor}[thm]{Corollary}

\newtheorem{prop}[thm]{Proposition}
\newtheorem*{prop*}{Proposition}

\newtheorem*{lem*}{Lemma}

\theoremstyle{definition}
\newtheorem{defin}[thm]{Definition}

\newtheorem{exm}[thm]{Example}
\newtheorem{conv}[thm]{Convention}
\theoremstyle{remark}
\newtheorem{rem2}[thm]{Remark}

\def\a{\mathbb{A}}
\def\c{\mathbb{C}}
\def\f{\mathbb{F}}

\def\l{\mathbb{L}}
\def\m{\mathbb{M}}
\def\n{\mathbb{N}}
\def\p{\mathbb{P}}

\def\r{\mathbb{R}}

\def\z{\mathbb{Z}}

\def\cc{\mathcal{C}}
\def\cd{\mathcal{D}}

\def\Spt{\mathcal{S}pt}
\def\Motc{\mathcal{M}ot_\c}
\def\Motr{\mathcal{M}ot_\r}
\def\Sptc{\mathcal{S}pt^{C_2}}

\def\mft{\underline{\mathbb{F}_2}}

\def\Spec{\operatorname{Spec}}

\def\Tor{\operatorname{Tor}}
\def\Ext{\operatorname{Ext}}

\def\colim{\operatorname{colim}}
\def\lim{\operatorname{lim}}

\usepackage{amssymb}
\usepackage{tikz-cd}

\author{J.D. Quigley}\address{Cornell University}\email{jdq27@cornell.edu}
\title[Motivic and equivariant Mahowald invariants]{Real motivic and $C_2$-equivariant Mahowald invariants}

\begin{document}
\maketitle

\section*{Abstract}
We generalize the Mahowald invariant to the $\r$-motivic and $C_2$-equivariant settings. For all $i>0$ with $i \equiv 2,3 \mod 4$, we show that the $\r$-motivic Mahowald invariant of $(2+\rho \eta)^i \in \pi_{0,0}^{\r}(S^{0,0})$ contains a lift of a certain element in Adams' classical $v_1$-periodic families, and for all $i > 0$, we show that the $\r$-motivic Mahowald invariant of $\eta^i \in \pi_{i,i}^{\r}(S^{0,0})$ contains a lift of a certain element in Andrews' $\c$-motivic $w_1$-periodic families. We prove analogous results about the $C_2$-equivariant Mahowald invariants of $(2+\rho \eta)^i \in \pi_{0,0}^{C_2}(S^{0,0})$ and $\eta^i \in \pi_{i,i}^{C_2}(S^{0,0})$ by leveraging connections between the classical, motivic, and equivariant stable homotopy categories. The infinite families we construct are some of the first periodic families of their kind studied in the $\r$-motivic and $C_2$-equivariant settings. 

\tableofcontents

\section{Introduction} 

The Mahowald invariant \cite{Mah67,MR93} is a method for producing new infinite families in the stable homotopy groups of spheres. Let $\gamma \to RP^\infty$ denote the tautological line bundle over infinite real projective space, let $P^\infty_k = Th(k\gamma \to RP^\infty)$ denote the Thom spectrum of its $k$-fold Whitney sum, and let $P^\infty_{-\infty}$ be the homotopy limit taken over the collapse maps $P^\infty_{k-1} \to P^\infty_k$. Lin's Theorem \cite{LDMA80} implies that the inclusion of the $0$-cell
\begin{equation}\label{Eqn:Lin}
S^0 \xrightarrow{\simeq} \Sigma P^\infty_{-\infty}
\end{equation} 
is an equivalence of spectra after $2$-completion. From now on, let us implicitly complete all spectra at the prime two. Given a class $\alpha \in \pi_t(S^0)$ in the $2$-complete stable stems, the equivalence \eqref{Eqn:Lin} implies that there is a minimal $N>0$ such that the composite
\begin{equation}\label{Eqn:CompositeClassical}
S^t \xrightarrow{\alpha} S^0 \xrightarrow{\simeq} \Sigma P^\infty_{-\infty} \xrightarrow{c} P^\infty_{-N}
\end{equation}
is nontrivial. Since $S^t \to \Sigma P^\infty_{-N+1}$ is trivial, the fiber sequence
$$S^{-N+1} \to \Sigma P^\infty_{-N} \to \Sigma P^\infty_{-N+1}$$
implies that there exists a nontrivial lift $S^t \to S^{-N+1}$ of \eqref{Eqn:CompositeClassical}. The coset of these lifts is the Mahowald invariant of $\alpha$, denoted $M(\alpha)$.

Let $2^i \in \pi_0(S^0)$ be the degree $2^i$ self-map of $S^0$. Mahowald--Ravenel showed \cite{MR93} that $M(2^i)$ contains the first $v_1$-periodic class detected in Adams filtration $i$. More generally, empirical evidence \cite{Beh06,Beh07,MR87,MR93,Qui19d,Sad92} suggests that the Mahowald invariant of a $v_n$-periodic element is $v_{n}$-torsion, and thus $v_{n+k}$-periodic for some $k \geq 1$ (with some low-dimensional exceptions). 

Although the types of periodic families are somewhat well-understood in classical stable homotopy theory \cite{DHS88,HS98,Rav92}, their counterparts in the motivic and equivariant settings are much more complicated. In the motivic setting, a full classification of the types of periodicity is not known over any base scheme. In the equivariant setting, it is known that there are exotic forms of periodicity, but to our knowledge, no infinite families of exotic periodic elements have been constructed before. 

In previous work \cite{Qui17}, we defined the $\c$-motivic Mahowald invariant and used it to reconstruct ordinary and exotic periodic families in the $\c$-motivic stable stems. The purpose of this paper is to define $\r$-motivic and $C_2$-equivariant Mahowald invariants and apply them towards the construction of new periodic families in the $\r$-motivic and $C_2$-equivariant stable stems. Along the way, we will explore the connections between classical, motivic, and equivariant stable homotopy theory.

\subsection{Statement of main results}

Our first set of main theorems takes place in the motivic setting. Fix a field $k$. The $k$-motivic stable stems $\pi_{**}^k(S^{0,0})$ are the bigraded homotopy groups of the motivic sphere spectrum $S^{0,0}$.\footnote{Recall we are implicitly working in the $2$-complete setting.} We say that $\alpha \in \pi_{m,n}^k(S^{0,0})$ has \emph{topological dimension}, or \emph{stem}, $m$ and \emph{motivic weight} $n$. We are interested in two non-nilpotent elements in $\pi_{**}^k(S^{0,0})$. The first element is the algebraic Hopf map $\eta \in \pi_{1,1}^k(S^{0,0})$, which is the stable homotopy class of the composite
$$S^{3,2} \simeq \a^2_k \setminus \{0\} \xrightarrow{\pi} \p^1_k \simeq S^{2,1}.$$
The other map is $2+\rho \eta \in \pi_{0,0}^k(S^{0,0})$, where $2$ is the degree $2$-self map of the sphere spectrum and $\rho$ is a certain element in $\pi_{-1,-1}^k(S^{0,0})$ arising from the element $[-1] \in K_1^M(k)$ in Milnor K-theory.\footnote{The map $2+\rho \eta$ plays the role of $2$ in the $\r$-motivic stable stems. For example, the element $h_0$ which detects $2$ in the classical and $\c$-motivic Adams spectral sequences detects $2 + \rho \eta$ in the $\r$-motivic Adams spectral sequence. Another interpretation in terms of Grothendieck--Witt groups  can be found in \cite[Rmk. 2.2]{BI20}.}

In \cite{Qui17}, we defined the $\c$-motivic Mahowald invariant $M^\c(-)$ and computed $M^\c(2^i)$ and $M^\c(\eta^i)$ for all $i \geq 1$. We briefly recall its definition now. Let $\mu_2$ denote the cyclic group of order two, let $B_{gm}\mu_2$ denote the geometric classifying space of $\mu_2$, let $\gamma \to B_{gm}\mu_2$ denote its tautological bundle, and let $\underline{L}^\infty_{k} = Th(k\gamma \to B_{gm}\mu_2)$ be the motivic Thom spectrum of the $k$-fold sum of $\gamma$. In analogy with \eqref{Eqn:Lin}, Gregersen proved \cite{Gre12} an equivalence after $2$-completion
\begin{equation}\label{Eqn:Gregersen}
S^{0,0} \xrightarrow{\simeq} \Sigma^{1,0} \underline{L}^\infty_{-\infty}.
\end{equation}
We define the \emph{$\r$-motivic Mahowald invariant} $M^\r(-)$ in Section \ref{Sec:MIDef} by replacing the classical spectrum $P^\infty_{-N}$ by the $\r$-motivic spectrum $\underline{L}^\infty_{-\infty}$.

\begin{thmx}[Real motivic $v_1$-periodic families, Theorem \ref{Thm:rm2i}]\label{MT:Rv1}
Let $ i \geq 0$ and let $j \in \{2,3\}$. The $\r$-motivic Mahowald invariant of $(2+\rho\eta)^{4i+j}$ satisfies
\[
M^\r((2+\rho\eta)^{4i+j}) \ni \begin{cases}
v^{4i}_1 \tau \eta^2 \quad & j=2, \\
v^{4i}_1 \tau \eta^3 \quad & j=3.
\end{cases}
\]
The classes $v_1^{4i} \tau \eta^2$ and $v_1^{4i}\tau \eta^3$, defined in Section \ref{Sec:2i}, base-change to the classes with the same name \cite{Qui17} in $\pi_{**}^\c(S^{0,0})$. 
\end{thmx}

\begin{thmx}[Real motivic $w_1$-periodic families, Theorem \ref{rmieta}]\label{MT:Rw1}
Let $i \geq 0$ and $0 \leq j \leq 3$ with $(i,j) \neq (0,0)$. The $\r$-motivic Mahowald invariant of $\eta^{4i+j}$ satisfies
$$M^\r(\eta^{4i+j}) \ni \begin{cases}
w_1^{4(i-1)} \eta^2 \eta_4 \quad&j = 0, \\
w_1^{4i} \nu \quad & j=1, \\
w_1^{4i}\nu^2 \quad & j=2, \\
w_1^{4i}\nu^3 \quad & j=3.
\end{cases}$$
The classes $w_1^{4i} \eta^2 \eta_4$ and $w_1^{4i} \nu^\ell$, $1 \leq \ell \leq 3$, defined in Section \ref{SectionEta}, base-change to the classes with the same name \cite{And14} in $\pi_{**}^\c(S^{0,0})$. 
\end{thmx}

Recall that the \emph{coweight} or \emph{Milnor--Witt stem} of an element $\alpha \in \pi_{m,n}^k(S^{0,0})$ is defined by $MW(\alpha) = m-n$. Dugger and Isaksen \cite{DI16a} calculated $\pi_{**}^\r(S^{0,0})$ up to coweight four and Belmont--Isaksen \cite{BI20} calculated up to coweight eleven.  Theorem \ref{MT:Rv1} provides a new construction of infinite $v_1$-periodic families in $\pi_{**}^\r(S^{0,0})$ and Theorem \ref{MT:Rw1} provides the first construction of infinite $w_1$-periodic families in $\pi_{**}^\r(S^{0,0})$. As $v_1$ has coweight $1$ and $w_1$ has coweight $2$, Theorems \ref{MT:Rv1} and \ref{MT:Rw1} identify several $\r$-motivic periodic families in unbounded coweight; we believe our $w_1$-periodic families are the first exotic families of their kind identified in the $\r$-motivic setting. 

Our second set of main theorems takes place in the $C_2$-equivariant setting. The $C_2$-equivariant stable stems $\pi_{**}^{C_2}(S^{0,0})$ are the bigraded homotopy groups of the $C_2$-equivariant sphere spectrum. If $\alpha \in \pi_{m,n}^{C_2}(S^{0,0})$, then $\alpha$ is the stable homotopy class of a map $S^{(m-n)} \wedge S^{n\sigma} \to S^0$ where $S^{(m-n)}$ is the $(m-n)$-sphere with trivial $C_2$-action and $S^{n\sigma}$ is the one-point compactification of the $n$-fold direct sum of the real sign representation of $C_2$. 

We extend the Mahowald invariant to the $C_2$-equivariant setting as follows. As discussed above, the classical Mahowald invariant was defined using Lin's Theorem \eqref{Eqn:Lin} and the $k$-motivic Mahowald invariant was defined using Gregersen's Theorem \eqref{Eqn:Gregersen}. In Appendix \ref{App:Lin}, we prove a $C_2$-equivariant version of Lin's Theorem which we briefly outline now. Let $B_{C_2}\mu_2$ denote the $C_2$-equivariant classifying space for principal $\mu_2$-bundles\footnote{The cyclic groups of order two $C_2$ and $\mu_2$ are not related here: $C_2$ is the group in ``genuine $C_2$-spectra" or ``$C_2$-spaces" and $\mu_2$ is the group which acts on a $C_2$-spectrum or $C_2$-space through $C_2$-equivariant maps.} and let $\underline{R}^\infty_k$ denote the $C_2$-equivariant Thom spectrum of  $k\gamma \to B_{C_2}\mu_2$ where $\gamma \to B_{C_2}\mu_2$ is the tautological line bundle over $B_{C_2}\mu_2$. We show in Theorem \ref{Thm:Qui} that the inclusion of the $(0,0)$-cell
\begin{equation}\label{Eqn:C2Lin}
S^{0,0} \xrightarrow{\simeq} \Sigma^{1,0} \underline{R}^\infty_{-\infty}
\end{equation}
is an equivalence of genuine $C_2$-spectra after $2$-completion. 

The $C_2$-equivariant equivalence \eqref{Eqn:C2Lin} allows us to define the $C_2$-equivariant Mahowald invariant. Our second set of results concerns the $C_2$-equivariant Mahowald invariants of the $C_2$-equivariant analogs of the motivic non-nilpotent self maps discussed above. The analog of the first map is the $C_2$-equivariant Hopf map $\eta \in \pi_{1,1}^{C_2}(S^{0,0})$, which is the stable homotopy class of the composite
$$S^{3,2} \simeq S(\c^2) \xrightarrow{\pi} \c\p^1 \simeq S^{2,1}$$
where $S(\c^2)$ is the sphere in $\c^2$ and $C_2$ acts on $S(\c^2)$ and $\c\p^1$ by complex conjugation. The second map is the element $2 + \rho \eta \in \pi_{0,0}^{C_2}(S^{0,0})$, where $2$ is the degree $2$ self map of the sphere and $\rho \in \pi_{-1,-1}^{C_2}(S^{0,0})$ is the Euler class defined by including fixed points $S^0 \hookrightarrow S^\sigma$. 

\begin{thmx}[$C_2$-equivariant $v_1$-periodic families, Theorem \ref{Thm:em2i}]\label{MT:Ev1}
Let $ i \geq 0$ and let $j \in \{2,3\}$. The $C_2$-equivariant Mahowald invariant of $(2+a\eta)^{4i+j}$ satisfies
\[
M^{C_2}((2+a\eta)^{4i+j}) \ni \begin{cases}
v^{4i}_1 \tau \eta^2 \quad & j=2, \\
v^{4i}_1 \tau \eta^3 \quad & j=3.
\end{cases}
\]
The classes $v_1^{4i} \tau \eta^2$ and $v_1^{4i}\tau  \eta^3$, defined in Section \ref{Sec:2i}, are $C_2$-equivariant lifts of the classes $v_1^{4i} \eta^2$ and $v_1^{4i}\eta^3$ \cite{Ada66} in $\pi_*(S^0)$. 
\end{thmx}

\begin{rem2}
The classes $v_1^{4i} \tau \eta^2$ and $v_1^{4i} \tau \eta^3$ in Theorem \ref{MT:Ev1} lie in nonzero coweight, so they are represented by maps between nontrivial representation spheres. One could also define $C_2$-equivariant lifts of Adams' classes $v_1^{4i} \eta^2$ and $v_1^{4i} \eta^3$ by equipping them trivial $C_2$-action, but the resulting $C_2$-equivariant stable homotopy classes would have coweight zero. 
\end{rem2}

\begin{thmx}[$C_2$-equivariant $w_1$-periodic families, Theorem \ref{cmieta}]\label{MT:Ew1}
Let $i \geq 0$ and $0 \leq j \leq 3$ with $(i,j) \neq (0,0)$. The $C_2$-equivariant Mahowald invariant of $\eta^{4i+j}$ satisfies
$$M^{C_2}(\eta^{4i+j}) \ni \begin{cases}
w_1^{4(i-1)} \eta^2 \eta_4 \quad&j = 0, \\
w_1^{4i} \nu \quad & j=1, \\
w_1^{4i}\nu^2 \quad & j=2, \\
w_1^{4i}\nu^3 \quad & j=3.
\end{cases}$$
The classes $w_1^{4i} \eta^2 \eta_4$ and $w_1^{4i} \nu^\ell$, $1 \leq \ell \leq 3$, defined in Section \ref{SectionEta}, have geometric fixed points Adams' classes $v_1^{4i} 8\sigma$ and $v_1^{4i} \eta^\ell$ in $\pi_*(S^0)$ defined in \cite{Ada66}. 
\end{thmx}

The $C_2$-equivariant stable stems have been analyzed by Araki and Iriye \cite{AI82,Iri82}, Dugger--Isaksen \cite{DI16}, and Belmont--Guillou--Isaksen \cite{BGI20}, but as in the $\r$-motivic setting, these calculations are restricted to low coweight. Theorems \ref{MT:Ev1} and \ref{MT:Ew1} identify the first $C_2$-equivariant periodic families in unbounded coweight. 

\subsection{Calculating generalized Mahowald invariants}

The heart of our work is the deep connection between $\r$-motivic and $C_2$-equivariant stable homotopy theory which has emerged over the past few decades \cite{Bac16,BS20,Cox79,DI16,HO16,HO17,MV99,Sch94}. The key idea throughout our work on generalized Mahowald invariants is that all the relevant technology, such as stunted projective spectra, Lin's Theorem, and Steenrod operations, is compatible under the comparison functors between the classical, motivic, and $C_2$-equivariant stable homotopy categories. 

We prove Theorems \ref{MT:Rv1} - \ref{MT:Ew1} through a combination of explicit computations and new methods for comparing Mahowald invariants across different categories. Before explaining further, we introduce the following convention to simplify notation:

\begin{conv}
We will sometimes refer to the \emph{generalized Mahowald invariants in a category $\cc$} instead of the classical, motivic, or equivariant Mahowald invariant. In this case, we will use the notation $M^{\cc}(-)$ with the following conventions:
\begin{itemize}
\item The classical Mahowald invariant is the Mahowald invariant in $SH$. 
\item The $k$-motivic Mahowald invariant is the Mahowald invariant in $SH_k$.
\item The $C_2$-equivariant Mahowald invariant is the Mahowald invariant in $SH_{C_2}$.
\end{itemize}
\end{conv}

For the remainder of the introduction, we let $\alpha$ denote either $2+\rho \eta$ or $\eta$. We now summarize the computations of $M^{\cc}(\alpha^{4i+j})$ which prove Theorems \ref{MT:Rv1} - \ref{MT:Ew1}. 

To compute $M^{\cc}(\alpha^{4i+j})$, we first we find an upper bound on its stem $|M^{\cc}(\alpha^{4i+j})|$. In the $\r$-motivic case, we do so by comparing $|M^\r(\alpha^{4i+j})|$ to $|M^\c(i^*\alpha^{4i+j})|$, where $i^*$ is base-change along $\r \to \c$. In the $C_2$-equivariant case, we obtain an upper bound by comparing $|M^{C_2}(\alpha^{4i+j})|$ to $|M(U \alpha^{4i+j})|$ or $|M(\Phi^{C_2}\alpha^{4i+j})|$, where $U\alpha^{4i+j}$ is the underlying nonequivariant map of $\alpha^{4i+j}$ and $\Phi^{C_2}\alpha^{4i+j}$ is its geometric fixed points. In both settings, our comparisons also yield a candidate $\beta \in \pi_{**}^{\cc}(S^{0,0})$ with the property that $\beta \in M^{\cc}(\alpha^{4i+j})$ if the upper bound is tight. 

\begin{rem2}
In the classical and $\c$-motivic settings \cite{MR93,Qui17}, the upper bound on $M(2^{4i+j})$ is obtained using approximations to the Mahowald invariant based on connective real topological K-theory or its $\c$-motivic analogs. A similar approach can be used to bound the stem of the $\r$-motivic Mahowald invariant, but the requisite calculations are much more laborious than the comparison methods we use here.
\end{rem2}

We prove that our upper bound is tight by finding a lower bound and showing it agrees. In the $\r$-motivic setting, we do so by induction on the power of $\alpha$. More precisely, we show that 
$$|M^{\r}(\alpha^{4(i+1)+j})| \geq |M^{\r}(\alpha^{4i+j})| + C,$$
where $C = 8$ if $\alpha = 2+\rho \eta$ and $C = 20$ if $\alpha = \eta$. This inequality is obtained by showing that $\alpha^4 = 0$ on a certain subcomplex of the $\r$-motivic analog of $P^\infty_{-\infty}$ comprised of cells in $C$ consecutive dimensions. The induction proceeds by calculating $M^\r(\alpha^j)$ for $0 \leq j \leq 3$ and iteratively applying the inequality. 

The relationship between $\r$-motivic and $C_2$-equivariant stable homotopy theory shines through when we need a lower bound in the $C_2$-equivariant setting. Instead of an intractable study of $C_2$-equivariant stunted projective spectra, we obtain a $C_2$-equivariant lower bound using the $\r$-motivic lower bounds and equivariant Betti realization, the functor induced by sending an $\r$-scheme to its space of $\c$-points with Galois action. In fact, we cannot progress any other way: the necessary range of $C_2$-equivariant stable stems\footnote{(all coweights less than or equal to eleven)} is far beyond our current understanding. 

In summary, we use a variety of comparison results to prove our main theorems. We list these in the following omnibus theorem, where we use $S_\cc$ and $S_\cd$ to denote the sphere spectrums in the categories $\cc$ and $\cd$, respectively. 

\begin{thmx}[Propositions \ref{compatibility}, \ref{squeeze}, and \ref{Lem:BaseChange}]\label{MT:Squeeze}
Suppose $F : \cc \to \cd$ is any of the following functors:
\begin{enumerate}
\item Equivariant Betti realization $Re_{C_2} : \Motr \to \Sptc.$
\item The forgetful functor $U : \Sptc \to \Spt.$
\item Geometric fixed points $\Phi^{C_2} : \Sptc \to \Spt$.
\item Betti realization $Re_\r : \Motr \to \Spt$. 
\item Base-change $i^*: \Motr \to \Motc$. 
\end{enumerate}
Suppose $\alpha, \beta \in \pi^{\cd}_{**}(S_\cd)$ satisfy $M^{\cd}(\alpha) \ni \beta$. Suppose further that there exist $\alpha', \beta' \in \pi^{\cc}_{**}(S_\cc)$ such that $F(\alpha') = \alpha$ and $F(\beta') = \beta$. Then we have the following inequality on stems
$$|M^{\cc}(\alpha')| \leq |\beta'|.$$
Further, if $|M^{\cc}(\alpha')| = |\beta'|$, then $\beta' \in M^{\cc}(\alpha')$. 
\end{thmx}

\begin{rem2}
For the reader's reference, we state here precisely how the upper and lower bounds needed to prove Theorems \ref{MT:Rv1} - \ref{MT:Ew1} are obtained:

\begin{itemize}

\item For Theorems \ref{MT:Rv1} and \ref{MT:Rw1}, the upper bounds are obtained using part (5) of Theorem \ref{MT:Squeeze} applied to the $\c$-motivic calculations of \cite{Qui17}. The lower bounds are obtained using Atiyah--Hirzebruch calculations.

\item For Theorems \ref{MT:Ev1} and \ref{MT:Ew1}, the upper bounds are obtained using parts (3) and (2) of Theorem \ref{MT:Squeeze}, respectively, applied to the classical calculations of \cite{MR93}. The lower bounds are obtained by applying part (1) of Theorem \ref{MT:Squeeze} to the analogous $\r$-motivic lower bounds.

\end{itemize}
\end{rem2}

\subsection{Construction of new periodic families}

Essential to the discussion above was the existence of $\r$-motivic and $C_2$-equivariant families of elements compatible with the classical and $\c$-motivic $v_1$- and $w_1$-periodic families appearing in \cite{MR93} and \cite{Qui17}.\footnote{Here ``compatible" means that the $\r$-motivic elements base-change to the corresponding $\c$-motivic elements, and the $C_2$-equivariant elements have underlying nonequivariant maps or geometric fixed points the corresponding classical elements.} The construction of these elements is somewhat involved, so we summarize it now.

We begin by constructing the desired $\r$-motivic elements, the first step of which involves constructing elements in the $E_2$-page of the $\r$-motivic Adams spectral sequence 
$$ E_2 = \Ext_{A^\r}^{***}(\m_2^\r,\m_2^\r) \Rightarrow \pi_{**}^\r(S^{0,0})$$
which will detect them. We recall several techniques for accessing the $E_2$-page in Section \ref{Sec:Ext}, including the $\c$-motivic May spectral sequence \cite{DI10}, the $\rho$-Bockstein spectral sequence \cite{DI16a,Hil11}, and vanishing and periodicity results of Guillou--Isaksen and Ang Li \cite{GI16,GI15,GI15a, Li19}. We also introduce an $\r$-motivic May spectral sequence which allows us to define the desired classes via iterated matric Massey products.

We show that the new classes are not boundaries by relating them to the analogous $\c$-motivic and classical elements of Adams \cite{Ada66}, Andrews \cite{And14}, and the author \cite{Qui17}. We then check that the elements we constructed in the $E_2$-page are permanent cycles using a myriad of technical tricks.

We use our new infinite families in $\pi_{**}^\r(S^{0,0})$ to construct analogous families in $\pi_{**}^{C_2}(S^{0,0})$. That is, we obtain classes in $\pi_{**}^{C_2}(S^{0,0})$ by taking the equivariant Betti realization of our classes in $\pi_{**}^\r(S^{0,0})$. We then show that these $C_2$-equivariant classes are nontrivial by identifying their images under the forgetful functor (in the $v_1$-periodic case) and the geometric fixed points functor (in the $w_1$-periodic case) with Adams' classical $v_1$-periodic families.

\subsection{Outline}

In Section \ref{Sec:Lin} we establish notation and recall background material, including the stunted projective spectra and variants of Lin's Theorem which are used to define generalized Mahowald invariants. 

In Section \ref{Sec:MIDef}, we define $\r$-motivic and $C_2$-equivariant Mahowald invariants using the results from Section \ref{Sec:Lin}. We make some low-dimensional calculations and prove Theorem \ref{MT:Squeeze}. 

In Section \ref{Sec:Ext}, we discuss the $E_2$-pages of the $\c$- and $\r$-motivic Adams spectral sequences. We recall the $\c$-motivic May spectral sequence, introduce the $\r$-motivic May spectral sequence, discuss the $\rho$-Bockstein spectral sequence, and summarize Guillou and Isaksen's results on vanishing and $h_1$-periodicity in motivic $\Ext$ groups. 

In Section \ref{Sec:2i}, we prove Theorems \ref{MT:Rv1} and \ref{MT:Ev1}. We lift certain $v_1$-periodic classes from the $\c$-motivic stable stems to the $\r$-motivic setting. Using Theorem \ref{MT:Squeeze}, these give an upper bound on the dimension of the generalized Mahowald invariants of $(2+\rho \eta)^i$ for $i \equiv 2,3 \mod 4$. We show this upper bound is tight by establishing identical lower bounds by analyzing the $\r$-motivic homotopy of certain stunted lens spectra. 

In Section \ref{SectionEta}, we prove Theorems \ref{MT:Rw1} and \ref{MT:Ew1}. To do so, we use the $\rho$-Bockstein spectral sequence \cite{DI16a,GHIR17,Hil11} to relate the $\r$-motivic stable stems to the $\c$-motivic and classical stable stems. This allows us to lift several $\c$-motivic $w_1^4$-periodic families of \cite{And14} to the $\r$-motivic setting. These families are shown to be contained in the generalized Mahowald invariants of $\eta^i$ using Theorem \ref{MT:Squeeze} and an intricate analysis of the homotopy of certain stunted lens spectra. 

In Appendix \ref{App:Lin}, we prove the genuine $C_2$-equivariant analog of Lin's Theorem summarized in Section \ref{Sec:Lin}. For the most part, our proof is a straightforward adaptation of Lin and Gregersen's original work \cite{Gre12,LDMA80}. We also fill in the details relating the various incarnations of Lin's Theorem and stunted projective spectra, which we use in the proof of Theorem \ref{MT:Squeeze}. 

\subsection{Notation and conventions} We will use the following notation throughout the paper:
\begin{itemize}
\item The category of spectra will be denoted $\Spt$.
\item The category of genuine $C_2$-spectra will be denoted $\Sptc$. We will always use `$C_2$' to mean the cyclic group of order two in genuine $C_2$-spectra. When an object only has an action by the cyclic group of order two (unrelated to a possible $C_2$-action), we will use `$\mu_2$' to denote the group acting on that object. 
\item The categories of motivic spectra over $Spec(\c)$ and $Spec(\r)$ will be denoted $\Motc$ and $\Motr$, respectively. 
\item $SH$ denotes the classical stable homotopy category.
\item $SH_{C_2}$ denotes the $C_2$-equivariant stable homotopy category.
\item $SH_\c$ and $SH_\r$ denote the $\c$-motivic and $\r$-motivic stable homotopy categories, respectively.
\item $A^{C_2}$ is the $C_2$-equivariant mod two Steenrod algebra.
\item $A^k$ is the $k$-motivic mod two Steenrod algebra with $k=\c$ or $k=\r$.
\item The same decorations will be used for subalgebras of the Steenrod algebra.
\item $S_{C_2}$ denotes the $C_2$-equivariant sphere spectrum and $S_k$ denotes the $k$-motivic sphere spectrum with $k=\c$ or $k = \r$. If the context is clear, we will omit these subscripts.
\item The $C_2$-equivariant, $\c$-motivic, and $\r$-motivic homology of a point will be denoted $\m^{C_2}_2$, $\m^\c$, and $\m^\r$, respectively. 
\item We will use the notation for $\Ext$-groups defined in \cite[Notation 1.2]{GHIR17}.
\item We use $\lim$ to denote a (homotopy) limit and $\colim$ to denote a (homotopy) colimit. 
\end{itemize}

We will also employ the following conventions:
\begin{itemize}
\item If the context is clear, we will not decorate motivic homotopy groups to indicate the base field. 
\item We will use $\pi^{C_2}_{**}$ to denote $RO(C_2)$-graded $C_2$-equivariant homotopy groups where the regular representation $\rho \in RO(C_2)$ has bidegree $(2,1)$. 
\item Unless otherwise stated, everything is implicitly $2$-complete. We refer the reader to \cite{HKO11, Man18} for a discussion of completions in motivic homotopy theory. 
\end{itemize}

\subsection{Acknowledgements} 
This work originally contained sections analyzing parametrized and motivic Tate constructions. In order to streamline the exposition and clarify some proofs, old work analyzing the parametrized Tate construction of the sphere spectrum now appears in Appendix \ref{App:Lin}, while most of the results on the parametrized Tate construction appear in \cite{LLQ19} and \cite{QS19}.

We thank Mark Behrens for his guidance throughout this project. We are also especially thankful to Jonas Irgens Kylling for several helpful conversations and comments, as well as for his contributions to Section \ref{Sec:2i}. We also thank Elden Elmanto, Bert Guillou, Alice Hedenlund, Jens Jakob Kjaer, Guchuan Li, Vitaly Lorman, Jay Shah, Larry Taylor, Inna Zakharevich, Mingcong Zeng, and an anonymous referee for helpful discussions and comments. Part of this project was completed during the K-Theory Trimester Program at the Hausdorff Institute for Mathematics; we thank the organizers and the Institute for providing a stimulating working environment. The author was partially supported by NSF grant DMS-1547292.

\section{Classical, equivariant, and motivic versions of Lin's Theorem}\label{Sec:Lin}

We begin by recalling Lin's Theorem \cite{LDMA80} and its motivic \cite{Gre12} and $C_2$-equivariant (Appendix \ref{App:Lin}) analogs. We will apply these to define generalized Mahowald invariants in Section \ref{Sec:MIDef}. 

The structures of Sections \ref{SS:Lin} - \ref{SS:Qui} are nearly identical, so we include a brief overview here. We start by defining generalized stunted projective spectra as Thom spectra over a generalized classifying space for a cyclic group of order two. These objects appear in the statement of Lin's Theorem and are used to define generalized Mahowald invariants. We then recall the mod two cohomology of a point, the Steenrod algebra, and the mod two cohomology of stunted projective spectra as a module over the Steenrod algebra. These are used throughout the sequel in computations, especially those involving generalized Adams spectral sequences and Atiyah--Hirzebruch spectral sequences. Finally, we state without proof the generalization of Lin's Theorem which will be used to define generalized Mahowald invariants. 

The short Section \ref{SS:Qui2} contains an alternative definition of $C_2$-equivariant stunted projective spectra using $\r$-motivic stunted projective spectra, and Section \ref{SS:Compare} contains further comparison results. These are essential in Section \ref{Sec:MIDef}, where they are used to compare $\r$-motivic and $C_2$-equivariant Mahowald invariants. 

\subsection{Classical case: Lin's Theorem}\label{SS:Lin}

We begin with the classical result, Lin's Theorem. 

\begin{defin}\label{Def:Pinftyk}
Let $\gamma \to RP^\infty$ denote the tautological line bundle over infinite real projective space. A \emph{stunted real projective spectrum},
\begin{equation}
P^\infty_k := Th(k\gamma \to RP^\infty),
\end{equation}
is the Thom spectrum of the $k$-fold Whitney sum of $\gamma$. 

The bundle map $k\gamma \to (k+1)\gamma$ induces a collapse map $P^\infty_k \to P^\infty_{k+1}$ for all $k \in \z$. We let
\begin{equation}
P^\infty_{-\infty} := \lim_{k \to \infty} P^\infty_{-k}
\end{equation}
denote the homotopy limit over these maps. 
\end{defin}

We note that if $k>0$, then $P^\infty_k \simeq RP^\infty/RP^{k-1}$ is just the quotient of $RP^\infty$ by its $(k-1)$-skeleton. The following lemma is a well-known consequence of the Thom isomorphism.

\begin{lem}\label{Lem:PCoh}
For all $k \in \z$, there is an isomorphism of graded $\f_2$-vector spaces
$$H^*(P^\infty_k) \cong \f_2[x]\{e_{k}\}$$
where $|x| = 1$ and $|e_k| = k$. The $A$-module structure of $H^*(P^\infty_k)$ is determined by
$$Sq^i(x^je_k) = {\lfloor (j+k)/2 \rfloor \choose i} x^{j+i}e_k.$$
\end{lem}

\begin{thm}\label{Thm:Lin}\cite{LDMA80}
The inclusion of the zero cell
$$S^0 \to \Sigma P^\infty_{-\infty}$$
is an equivalence of spectra.\footnote{Recall that we are implicitly $2$-complete, so this is only an equivalence after $2$-completion.}
\end{thm}

\subsection{Motivic case: Gregersen's Theorem}\label{SS:Gre}

We now recall Gregersen's Theorem \cite{Gre12}, the motivic analog of Lin's Theorem. 

\begin{defin}\label{Def:Linftyk}
Let $B_{gm}\mu_2$ denote the geometric classifying space \cite{MV99} of $\mu_2$ and let $\gamma \to B_{gm}\mu$ be the tautological bundle. A \emph{stunted motivic lens spectrum}, 
\begin{equation}
\underline{L}^\infty_k := Th(k\gamma \to B_{gm}\mu_2),
\end{equation}
is the Thom spectrum of the $k$-fold Whitney sum of $\gamma$.\footnote{See \cite{Gre12} for details of this construction.}

The bundle map $k\gamma \to (k+1)\gamma$ induces a collapse map $\underline{L}^\infty_k \to \underline{L}^\infty_{k+1}$ for all $k \in \z$. We let
\begin{equation}
\underline{L}^\infty_{-\infty} := \lim_{k \to \infty} \underline{L}^\infty_{-k}
\end{equation}
denote the homotopy limit over these maps.
\end{defin}

\begin{rem2}
Definition \ref{Def:Linftyk} directly mirrors Definition \ref{Def:Pinftyk}: $RP^\infty$ is a model for the (ordinary) classifying space $B\mu_2$. 
\end{rem2}

\begin{defin}
Fix a field $k$ of characteristic zero and let $H$ denote the mod two motivic Eilenberg--MacLane spectrum. We will write $\m_2^k := H^{**}(\Spec(k))$ for the mod two motivic cohomology of $k$, $A^k$ for the $k$-motivic Steenrod algebra, and $(A^k)^\vee$ for its $\m_2^k$-linear dual. We will occasionally suppress the superscript `$k$' when the context is clear. 
\end{defin}

\begin{lem}
The mod two motivic cohomology of the complex numbers is given by
$$\m_2^\c \cong \f_2[\tau]$$
with $|\tau| = (0,1)$. The mod two motivic cohomology of the real numbers is given by
$$\m_2^\r \cong \f_2[\tau,\rho]$$
where $|\tau|=(0,1)$ and $|\rho| = (1,1)$. The map of graded rings $\m_2^\r \to \m_2^\c$ induced by base-change $\Spec(\c) \to \Spec(\r)$ is determined by $1 \mapsto 1$, $\tau \mapsto \tau$, and $\rho \mapsto 0$. 
\end{lem}

\begin{lem}\label{Lem:LCoh}\cite[Lem. 4.1.24]{Gre12}
For all $k \in \z$, there is an isomorphism of $\m_2$-modules
$$H^{**}(\underline{L}^\infty_k) \cong \m_2[u,v]/(u^2+\rho u+\tau v)\{e_{2k}\}$$
where $|u| = (1,1)$, $|v| = (2,1)$, and $|e_{2k}| = (2k,k)$. 

The (periodic) $A$-module structure of $H^{**}(\underline{L}^\infty_0)$ is given by
$$Sq^{2i}(v^je_0) = {2j \choose 2i} v^{j+i}e_0, \quad Sq^{2i+1}(v^je_0) = 0,$$
$$Sq^{2i}(uv^je_0) = {2j \choose 2i} uv^{j+i}e_0,\quad Sq^{2i+1}(uv^je_0) = {2j \choose 2i}v^{j+i+1}e_0.$$
The $A$-module structure of $H^{**}(\underline{L}^\infty_k)$ for $k<0$ coincides with the $A$-module structure obtained by periodically extending the $A$-module structure of $H^{**}(\underline{L}^\infty_0)$ to negative dimensions.  
\end{lem}

\begin{thm}\label{Thm:Gre}\cite{Gre12}
The inclusion of the $(0,0)$-cell 
$$S^{0,0} \to \Sigma^{1,0}\underline{L}^\infty_{-\infty}$$
is an equivalence of motivic spectra. 
\end{thm}

\subsection{$C_2$-equivariant case, I}\label{SS:Qui}

We now discuss a $C_2$-equivariant version of Lin's Theorem. As in the previous two subsections, our goal is only to define the relevant objects and state the main results needed in the sequel. More thorough discussion and details appear in Appendix \ref{App:Lin}. 

\begin{defin}\label{Def:Qinftyk}
Let $B_{C_2}\mu_2$ denote the $C_2$-equivariant classifying space of $C_2$-equivariant principal $\mu_2$-bundles, or equivalently, the classifying space of the family $\{e, C_2, \Delta\}$ of subgroups of $\mu_2 \rtimes C_2$ (where $\Delta$ is the diagonal subgroup). Let $\gamma \to B_{C_2}\mu_2$ denote the tautological $C_2$-equivariant line bundle. A \emph{stunted $C_2$-equivariant projective spectrum},
\begin{equation}
\underline{Q}^\infty_k := Th(k\gamma \to B_{C_2}\mu_2),
\end{equation}
is the $C_2$-equivariant Thom spectrum of the $k$-fold Whitney sum of $k\gamma$. We let
$$\underline{Q}^\infty_{-\infty} := \lim_{k \to \infty} \underline{Q}^\infty_{-k}$$
denote the homotopy limit along the collapse maps $\underline{Q}^\infty_k \to \underline{Q}^\infty_{k+1}$ induced by the bundle maps $k\gamma \to (k+1)\gamma$. 
\end{defin}

\begin{defin}
Let $H$ denote the Eilenberg--MacLane spectrum for the constant Mackey functor $\mft$. We will write $\m_2^{C_2} := H^{**}(pt)$ for the $C_2$-equivariant Bredon cohomology of a point with coefficients in $\mft$, $A^{C_2}$ for the $C_2$-equivariant Steenrod algebra, and $(A^{C_2})^\vee$ for its $\m_2^{C_2}$-linear dual. 
\end{defin}

\begin{prop}\cite[Prop. 6.2]{HK01}\cite[Prop. 2.13]{Ric14}
The Bredon cohomology of a point with coefficients in $\mft$ is given by
$$\m_2^{C_2} \cong \f_2[\tau, \rho] \oplus \bigoplus_{s \geq 0} \dfrac{\f_2[\tau]}{\tau^\infty}\left\{ \dfrac{\gamma}{\rho^s} \right\}$$
where $|\tau| = (0,1)$, $|\rho| = (1,1)$, $|\gamma| = (0,-1)$, and following \cite[Rmk. 2.1]{GHIR17}, $\frac{\f_2[\tau]}{\tau^\infty}$ is the infinitely divisible $\f_2[\tau]$-module consisting of elements of the form $\frac{x}{\tau^k}$ for $k \geq 1$. 
\end{prop}

The cooperations $(A^{C_2})^\vee = \pi_{**}^{C_2}(H\mft \wedge H\mft)$ form an $RO(C_2)$-graded Hopf algebroid called the \emph{$C_2$-equivariant dual Steenrod algebra}.

\begin{thm}\cite[Thm. 6.41]{HK01}
The $C_2$-equivariant dual Steenrod algebra has presentation
$$A^{C_2}_{**} = \pi_{**}^{C_2}(H \mft_{**} \wedge H\mft) = \m^{C_2}_2[ \xi_{i+1}, \tau_i : i \geq 0] / I$$
where $|\xi_{i}| = (2^{i+1}-2,2^i-1)$, $|\tau_i| = (2^{i+1}-1,2^i-1)$, and $I$ is the ideal generated by the relation $\tau^2_i = \rho \xi_{i+1} + (\rho \tau_0 + \tau) \tau_{i+1}$. 
\end{thm}

\begin{prop}
As modules over $\m^{C_2}_2$, we have
$$H^{**}_{C_2} (\underline{Q}^\infty_{-k}) \cong \Sigma^{-2k,-k} \m^{C_2}_2[x,y]/(x^2 + \rho x + \tau y),$$
$$H^{**}_{C_2,c}(\underline{Q}^\infty_{-\infty}) := \colim_{k \to \infty} H^{**}_{C_2}(\underline{Q}^\infty_{-k}) \cong \m^{C_2}_2[x,y,y^{-1}]/(x^2 + \rho x + \tau y).$$

The $A^{C_2}$-module structure of $H^{**}(\underline{Q}^\infty_0)$ is given by
$$Sq^{2i}(y^je_0) = {2j \choose 2i} y^{j+i}e_0, \quad Sq^{2i+1}(y^je_0) = 0,$$
$$Sq^{2i}(xy^je_0) = {2j \choose 2i} xy^{j+i}e_0,\quad Sq^{2i+1}(xy^je_0) = {2j \choose 2i}y^{j+i+1}e_0.$$

The $A^{C_2}$-module structure of $H^{**}(\underline{Q}^\infty_k)$ for $k<0$ coincides with the $A^{C_2}$-module structure obtained by periodically extending the $A^{C_2}$-module structure of $H^{**}(\underline{Q}^\infty_0)$ to negative dimensions. 
\end{prop}

\begin{thm}[Theorem \ref{Thm:Qui}]\label{Thm:Quigley}
The inclusion of the $(0,0)$-cell 
$$S^{0,0} \to \Sigma^{1,0} \underline{Q}^\infty_{-\infty}$$
is an equivalence of $C_2$-spectra after $2$-completion. 
\end{thm}

\subsection{$C_2$-equivariant case, II}\label{SS:Qui2}

We now discuss a motivic-to-equivariant construction of the $C_2$-equivariant stunted projective spectra. Again, more details appear in Appendix \ref{App:Lin}. 

In \cite[Sec. 3.3]{MV99}, Morel and Voevodsky defined the Betti realization functor
$$Re_B : \Motc \to \Spt$$
which is induced by the assignment which sends a scheme over $\Spec(\c)$ to its $\c$-points equipped with the analytic topology. They observe that
$$Re_B(S^{1,0}) \cong Re_B(S^{1,1}) \cong S^1 \quad\text{and}\quad Re_B(BG) \cong B(G(\c))$$
for any smooth group scheme over $\c$; see also \cite[Sec. 5]{Lev14}. 

They also introduced an equivariant Betti realization functor
$$Re_{C_2} : \Motr \to \Sptc$$
which is induced by the assignment which sends a scheme over $Spec(\r)$ to its $\c$-points equipped with the analytic topology, with $C_2$-action given by the Galois group of $\c$ over $\r$. Equivariant Betti realization was studied further by Heller and Ormsby in \cite{HO16}, where they show in particular that $Re_{C_2}$ is strong symmetric monoidal \cite[Prop. 4.8]{HO16} and that it preserves Eilenberg--MacLane spectra \cite[Thm. 4.17]{HO16}. 

\begin{defin}\label{ebrdef}
Define \emph{motivic-to-equivariant stunted projective spectra} by
$$\underline{R}^\infty_{-k} := Re_{C_2}(\underline{L}^\infty_{-k})$$
and let
$$\underline{R}^\infty_{-\infty} := \lim_{k\to \infty} \underline{R}^\infty_{-k}$$
denote the homotopy limit over the equivariant Betti realization of the collapse maps $\underline{L}^\infty_k \to \underline{L}^\infty_{k+1}$.\footnote{We opt to define $\underline{R}^\infty_{-\infty}$ as the above homotopy limit, instead of as the equivariant Betti realization of $\underline{L}^\infty_{-\infty}$, since equivariant Betti realization need not preserve homotopy limits.}
\end{defin}

\begin{thm}[see Appendix \ref{SS:A4}]
For all $k \in \z$, there is an equivalence of $C_2$-spectra
$$\underline{R}^\infty_{k} \simeq \underline{Q}^\infty_{k}$$
after $2$-completion. Similarly, there is an equivalence of $C_2$-spectra
$$\underline{R}^\infty_{-\infty} \simeq \underline{Q}^\infty_{-\infty}.$$
In particular, the inclusion of the $(0,0)$-cell 
$$S^{0,0} \to \Sigma^{1,0} \underline{R}^\infty_{-\infty}$$
is an equivalence of $C_2$-spectra after $2$-completion. 
\end{thm}

\subsection{Further comparisons between stunted projective spectra}\label{SS:Compare}

We will need a few more results comparing generalized stunted projective spectra in order to prove Theorem \ref{MT:Squeeze}.

\begin{lem}\label{Lem:ReC2U}
For all $N \in \z$, the following statements hold:
\begin{enumerate}
\item $Re_{C_2}(\underline{L}^\infty_{-N}) = \underline{R}^\infty_{-N}$. 
\item $U(\underline{R}^\infty_{-N}) \simeq P^\infty_{-2N}$. 
\end{enumerate}
\end{lem}

\begin{proof}
Fix $N \in \z$. 
\begin{enumerate}
\item This holds by definition. 
\item On cohomology, the forgetful functor preserves $A$-module structure while setting $\rho = 0$ and $\tau= 1$. Therefore we have an isomorphism of $A$-modules
$$H^*(U(\underline{R}^\infty_{-N})) \cong \Sigma^{-2N} \f_2[x,y]/(x^2 = y) \cong \Sigma^{-2N} \f_2[x] \cong H^*(P^\infty_{-2N})$$
which implies an abstract weak equivalence $U(\underline{R}^\infty_{-N}) \simeq P^\infty_{-2N}$ via the Adams spectral sequence.
\end{enumerate}
\end{proof}

\begin{defin}
The \emph{geometric fixed points functor} $\Phi^{C_2}: \Sptc \to \Spt$ is defined by setting
$$\Phi^{C_2}(X) := (X \wedge \widetilde{EC_2})^{C_2}.$$
\end{defin}

\begin{defin}\cite{MV99}
The \emph{real Betti realization functor} $Re_\r : \Motr \to \Spt$ is defined by sending an $\r$-scheme to its $\r$-points equipped with the analytic topology.
\end{defin}

\begin{lem}\label{Lem:GeomFP}
There are equivalences of spectra
$$\Phi^{C_2}(\underline{R}^\infty_k) \simeq P^\infty_k \vee P^\infty_k$$
and
$$Re_\r(\underline{L}^\infty_k) \simeq P^\infty_k \vee P^\infty_k.$$
\end{lem}

\begin{proof}
To see the first equivalence, recall that we have equivalences
$$\Phi^{C_2}(\underline{R}^\infty) \simeq \Phi^{C_2}(B_{C_2}\mu_2) \simeq RP^\infty \sqcup RP^\infty.$$
Recall that the geometric fixed points of a suspension spectrum are the suspension spectrum of the fixed points. Further, recall that the geometric fixed points of an equivariant Thom spectrum can be calculated as
$$\Phi^{C_2}(Th(V \to X)) \simeq Th((V|_X)^{C_2} \to X^{C_2}).$$
We need to calculate $\Phi^{C_2}(Th(k\gamma \to B_{C_2}\mu_2))$. By the above, we have $(B_{C_2}\mu_2)^{C_2} \simeq RP^\infty \sqcup RP^\infty.$ The fixed points of the restriction $(k\gamma)|_{(B_{C_2}\mu_2)^{C_2}}$ can be identified with $k\gamma$. Therefore we have 
$$\Phi^{C_2}(\underline{R}^\infty_{k}) \simeq Th(k\gamma \to RP^\infty \sqcup RP^\infty) \simeq Th(k\gamma \to RP^\infty) \vee Th(k\gamma \to RP^\infty) \simeq P^\infty_k \vee P^\infty_k.$$

The second equivalence follows from commutativity of the diagram \cite[Prop. 31]{Bac16}
\[
\begin{tikzcd}
SH_\r[1/2,\eta^{-1}] \arrow{r}{Re_{C_2}} \arrow{dr}{Re_\r} & SH_{C_2}[1/2,\eta^{-1}] \arrow{d}{\Phi^{C_2}} \\
& SH[1/2]
\end{tikzcd}
\]
along with our calculations of $\Phi^{C_2}(\underline{R}^\infty_k)$ and $Re_{C_2}(\underline{L}^\infty_k)$. 
\end{proof}

The previous lemma gives the following example of non-commutativity of geometric fixed points and homotopy limits.

\begin{cor}\label{PhiC2}
After $2$-completion, there are equivalences of spectra
$$\Phi^{C_2} \left(\lim_{k \to \infty} \underline{R}^\infty_{-k} \right) \simeq S^{-1} \quad \text{and} \quad \lim_{k \to \infty} \left( \Phi^{C_2}(\underline{R}^\infty_{-k}) \right) \simeq S^{-1} \vee S^{-1}.$$
\end{cor}

\section{Real motivic and $C_2$-equivariant Mahowald invariants}\label{Sec:MIDef}

We now come to the main definitions of the paper. In this section, we define motivic and equivariant Mahowald invariants. We then give some low-dimensional computations, demonstrate some simple transcategorical comparison techniques, and prove the comparison results (Theorems \ref{compatibility} and \ref{squeeze}) comprising Theorem \ref{MT:Squeeze}. 

\subsection{Definitions of generalized Mahowald invariants} 

We start by recalling the definition of the classical Mahowald invariant using Lin's Theorem (Theorem \ref{Thm:Lin}). 

\begin{defin}
Let $\alpha \in \pi_t(S^0)$. The \emph{classical Mahowald invariant} of $\alpha$ is the coset of completions of the diagram
\[
\begin{tikzcd}
S^t \arrow{d}{\alpha} \arrow[rr,dashed] & & S^{-N+1} \arrow{d} \\
S^0 \arrow{r}{\simeq} & \Sigma P^\infty_{-\infty} \arrow{r} & \Sigma P^\infty_{-N}
\end{tikzcd}
\]
where $N>0$ is minimal so that the left-hand composite is nontrivial. The equivalence on the left-hand side is by Lin's Theorem \cite{LDMA80}, and the dashed arrow is the lift to the fiber of the sequence
$$S^{-N+1} \to \Sigma P^\infty_{-N} \to \Sigma P^\infty_{-N+1}$$
which is nontrivial by choice of $N$. The classical Mahowald invariant of $\alpha$ will be denoted $M^{cl}(\alpha)$. 
\end{defin}

Now fix a field $k = \c$ or $k=\r$. The $\c$-motivic Mahowald invariant was defined using Gregersen's Theorem (Theorem \ref{Thm:Gre}) in \cite{Qui17}. The same definition works over any field of characteristic zero:

\begin{defin}
Let $\alpha \in \pi_{s,t}(S^{0,0}_k)$. The \emph{$k$-motivic Mahowald invariant} of $\alpha$ is the coset of completions of the diagram
\[
\begin{tikzcd}
S^{s,t} \arrow[dashed,rr] \arrow{d}{\alpha} && S^{-2N+1,-N} \vee S^{-2N+2,-N+1} \arrow{d} \\
S^{0,0} \arrow{r}{\simeq} & \Sigma^{1,0} \underline{L}^\infty_{-\infty} \arrow{r} & \Sigma^{1,0} \underline{L}^\infty_{-N}
\end{tikzcd}
\]
where $N>0$ is minimal so that the left-hand composite is nontrivial. The dashed arrow is the lift to the fiber of the sequence
$$S^{-2N+1,-N} \vee S^{-2N+2,-N+1} \to \Sigma^{1,0} \underline{L}^\infty_{-N} \to \Sigma^{1,0} \underline{L}^\infty_{-N+1}$$
which is nontrivial by choice of $N$. The $k$-motivic Mahowald invariant of $\alpha$ will be denoted $M^k(\alpha)$. 
\end{defin}

Note that the target of the dashed arrow is a wedge of spheres. The $k$-motivic Mahowald invariant could be detected on both of these spheres, but we can (and will) regard $M^k(\alpha)$ as a coset in the homotopy of only the higher-dimensional sphere if this is the case. That is, if the composition of the dashed arrow in the diagram with projection onto $S^{-2N+2,-N+1}$ is non-trivial, we will say that $M(\alpha)$ is detected on $S^{-2N+2,-N+1}$; otherwise, we will say it is detected on $S^{-2N+1,-N}$. This convention is also used in \cite{Qui17}.

Our final variant of the Mahowald invariant is the $C_2$-equivariant Mahowald invariant, which we define using Theorem \ref{Thm:Quigley}.

\begin{defin}
Let $\alpha \in \pi^{C_2}_{**}(S^{0,0})$. The \emph{$C_2$-equivariant Mahowald invariant} of $\alpha$ is the coset of completions of the diagram
\[
\begin{tikzcd}
S^{s,t} \arrow[dashed,rr] \arrow{d}{\alpha} & & S^{-2N+1,-N} \vee S^{-2N+2,-N+1} \arrow{d} \\
S^{0,0} \arrow{r}{\simeq} & \Sigma^{1,0} \underline{R}^\infty_{-\infty} \arrow{r} & \Sigma^{1,0} \underline{R}^\infty_{-N}
\end{tikzcd}
\]
where $N>0$ is minimal so that the left-hand composite is nontrivial. The $C_2$-equivariant Mahowald invariant of $\alpha$ will be denoted $M^{C_2}(\alpha)$. 
\end{defin}

As in the motivic case, the target of the dashed arrow is a wedge of spheres. We will employ the convention above and say that the $C_2$-equivariant Mahowald invariant is only detected on one cell.

\subsection{Elementary $\r$-motivic computations}\label{Section:LowDim}
As in the classical and $\c$-motivic cases, we can compute the $\r$-motivic and $C_2$-equivariant Mahowald invariants of the Hopf invariant one maps $\eta$, $\nu$, and $\sigma$. In the motivic case, these were constructed using Cayley-Dickson algebras in \cite{DI13}. In the equivariant case, these are the usual Hopf maps where $C_2$ acts on the source and target by (a fixed) complex conjugation \cite[Section 10]{GHIR17}. 

The proof of the following is similar to the proofs of \cite[Propositions 2.18-2.21]{Qui17}. In particular, one can compute the Atiyah--Hirzebruch spectral sequence by replacing Betti realization $Re_B$ in those proofs by the base-change functor $i^* : SH_\r \to SH_\c$ induced by the morphism $Spec(\c) \to Spec(\r)$. 

\begin{prop}\label{Prop:ElementaryR}
We have
\begin{enumerate}
\item $\eta \in M^{\r}(2 +\rho \eta) \subset \pi^\r_{1,1}(S^{0,0})$, $\tau\eta^2 \in M^\r((2+\rho \eta)^2) \subseteq \pi^{\r}_{2,1}(S^{0,0})$, and $\tau \eta^3 \in M^\r((2 + \rho \eta)^3) \subseteq \pi_{3,2}^{\r}(S^{0,0})$,
\item $\nu \in M^\r(\eta) \subset \pi_{3,2}^{\r}(S^{0,0})$, and
\item $\sigma \in M^\r(\nu) \subset \pi_{7,4}^{\r}(S^{0,0})$.
\end{enumerate}
\end{prop}

A similar proof can be given in the $C_2$-equivariant setting, but we will proceed using the results in the next section instead.

\subsection{Squeeze lemmas for generalized Mahowald invariants}
The following lemmas allow us to compare Mahowald invariants in different categories. Throughout this section, we let $|\alpha|$ denote the (topological) dimension of a class $\alpha$ in the classical, motivic, or $C_2$-equivariant stable stems. Recall that the notation $S_\cc$ denotes the sphere spectrum in the category $\cc$. 

\begin{lem}\label{compatibility}
Suppose $F : \cc \to \cd$ is any of the following combinations of functors and categories:
\begin{enumerate}
\item Equivariant Betti realization $Re_{C_2} : \Motr \to \Sptc.$
\item The forgetful functor $U : \Sptc \to \Spt.$
\end{enumerate}
Suppose $\alpha, \beta \in \pi^{\cd}_{**}(S_\cd)$ satisfy $M^{\cd}(\alpha) \ni \beta$. Suppose further that there exist $\alpha', \beta' \in \pi^{\cc}_{**}(S_\cc)$ such that $F(\alpha') = \alpha$ and $F(\beta') = \beta$. Then
$$|M^{\cc}(\alpha')| \leq |\beta'|.$$
Further, if $|M^{\cc}(\alpha')| = |\beta'|$, then $\beta' \in M^{\cc}(\alpha')$. 
\end{lem}

\begin{proof}
Recall that the image of the analog of $P^\infty_{-N}$ under each functor was determined in Lemma \ref{Lem:ReC2U}.

Let $\alpha \in \pi^{\cd}_{s,t}(S_\cd)$ and let $\beta \in \pi^{\cd}_{s+2N+1 + \epsilon,t+N + \epsilon}(S_\cd)$ with $\epsilon \in \{0,1\}$. Then $M^\cd(\alpha) \ni \beta$ means the following diagram commutes,
\[
\begin{tikzcd}
S^{s,t}_\cd \arrow{d}{\alpha} \arrow[rr,dashed,"\beta"] & & S^{-2N+1+\epsilon,-N+\epsilon}_\cd \arrow{d} \\
S^{0,0}_\cd \arrow{r}{\simeq} & \Sigma_\cd \underline{R}^\infty_{-\infty} \arrow{r} & \Sigma_\cd \underline{R}^\infty_{-N},
\end{tikzcd}
\]
where $N>0$ is minimal so that the left-hand composite is nontrivial. Here, $\Sigma_\cd = \Sigma^{1,0}$ if $\cd = SH_\r, SH_\c, SH_{C_2}$ and $\Sigma_\cd = \Sigma$ if $\cd = SH$. By assumption, this diagram can be obtained by applying $F$ to the diagram
\[
\begin{tikzcd}
S^{s',t'}_\cc \arrow{d}{\alpha'} \arrow[rr,dashed,"\beta'"] && S^{-2N+1+\epsilon,-N+\epsilon}_\cc \arrow{d} \\
S^{0,0}_\cc \arrow{r}{\simeq} & \Sigma_\cc \underline{L}^\infty_{-\infty} \arrow{r} & \Sigma_\cc \underline{L}^\infty_{-N}.
\end{tikzcd}
\]
If the left-hand composite in the second diagram were trivial, then so would the left-hand composite in the first diagram, which would be a contradiction. However, we cannot conclude that $N$ is minimal so that the left-hand composite in the second diagram is nontrivial. Therefore we have $|M^\cc(\alpha')| \leq |\beta'|$. 
\end{proof}

We will also need the following analysis for computing more complicated $C_2$-equivariant Mahowald invariants. In particular, we will want a ``squeeze lemma" for geometric fixed points and non-equivariant Betti realization. 

The natural transformation
$$\Phi^{C_2} \lim(-) \to \lim \Phi^{C_2}(-)$$
applied to our $C_2$-equivariant stunted projective spectra results in the map of spectra 
$$p: S^0 \to S^0 \vee S^0$$
induced by the desuspension of the pinch map. We will need the following lemma regarding the effect of this map in homotopy. 

\begin{lem}\label{pinch}
Let $\alpha \in \pi_t(S^0)$. Then $\alpha$ can be represented by the composite
$$S^t \xrightarrow{\alpha} S^0 \xrightarrow{p} S^0 \vee S^0 \xrightarrow{\pi_1} S^0$$
where $\pi_1: S^0 \vee S^0 \to S^0$ is the projection onto the first summand. 
\end{lem}

\begin{proof}
This follows from commutativity of the diagram 
\[
\begin{tikzcd}
& S^1 \arrow{r}{=} & S^1  \\
S^1 \arrow{ur}{=} \arrow{r}{\Delta} &S^1 \times S^1 \arrow{u}{\pi_1} &  S^1 \vee S^1 \arrow{l}{\simeq} \arrow{u}{p}
\end{tikzcd}
\]
in the stable homotopy category. 
\end{proof}

We can now prove our analog of the squeeze lemma for $\Phi^{C_2} : \Sptc \to \Spt$ and $Re_\r : \Motr \to \Spt$. 

\begin{lem}\label{squeeze}
Suppose $F : \cc \to \cd$ is any of the following combinations of functors and categories:
\begin{enumerate}
\item Geometric fixed points $\Phi^{C_2} : \Sptc \to \Spt$.
\item Betti realization $Re_\r : \Motr \to \Spt$. 
\end{enumerate}
Suppose $\alpha, \beta \in \pi^{\cd}_{**}(S_\cd)$ satisfy $M^{\cd}(\alpha) \ni \beta$. Suppose further that there exist $\alpha', \beta' \in \pi^{\cc}_{**}(S_\cc)$ such that $F(\alpha') = \alpha$ and $F(\beta') = \beta$. Then
$$|M^{\cc}(\alpha')| \leq |\beta'|.$$
Further, if $|M^{\cc}(\alpha')| = |\beta'|$, then $\beta' \in M^{\cc}(\alpha')$. 
\end{lem}

\begin{proof}
\begin{enumerate}
\item We claim that there is a commutative diagram
\[
\begin{tikzcd}
S^{s,t} \arrow[dashed,rr,"\beta' "] \arrow{d}{\alpha'} &&  S^{s',t'} \arrow{d} \\
S^{0,0} \arrow{r}{\simeq} & \Sigma^{1,0} \underline{R}^\infty_{-\infty} \arrow{r} & \Sigma^{1,0} \underline{R}^\infty_{-N'}
\end{tikzcd}
\]
where the left-hand composite is nontrivial and $s', t'$ are determined by $N'$. Consider the following diagram:
\[
\begin{tikzcd}
S^{s-t} \arrow{d}{\alpha} \arrow[dashed, drrr,"\beta"] \arrow[dashed,ddrrr] \arrow[dashed,dddrrr,bend left=40,"\beta"]& & & \\
S^0 \arrow[r,crossing over,"\simeq"] \arrow{d}{p} & \Sigma P^\infty_{-\infty} \arrow[r,crossing over] \arrow[d,crossing over] &\Sigma P^\infty_{-N} \arrow[d,crossing over] & S^{-N+1} \arrow{d} \arrow{l} \\
S^0 \vee S^0 \arrow{r}{\simeq} \arrow{d}{\pi} &\Sigma P^\infty_{-\infty} \vee \Sigma P^\infty_{-\infty} \arrow{r} \arrow{d} & \Sigma P^\infty_{-N} \vee \Sigma P^\infty_{-N} \arrow{d} & S^{-N+1} \vee S^{-N+1} \arrow{l} \arrow{d} \\
S^0 \arrow{r}{\simeq} & S^0 \arrow{r} & \Sigma P^\infty_{-N} & S^{-N+1} \arrow{l}.
\end{tikzcd}
\]
The top-left vertical map $S^{s-t} \xrightarrow{\alpha} S^0$ along with the middle dashed arrow and the top full row are the diagram defining the classical Mahowald invariant $M^{cl}(\alpha)$. The left-hand vertical composite is $\alpha \in \pi_{s-t}(S^0)$ by Lemma \ref{pinch}, so the outer part of the diagram is also the diagram defining the classical Mahowald invariant $M^{cl}(\alpha)$. The middle part of the diagram consisting of the composite $S^0 \to S^0 \vee S^0$, the middle full row, and the bottom dashed arrow (unlabeled) are the diagram resulting from applying $\Phi^{C_2}$ to the claimed commutative diagram by Corollary \ref{PhiC2}. By commutativity of the entire diagram, we conclude that the composite 
$$S^0 \xrightarrow{2^i} S^0  \xrightarrow{p} S^0 \vee S^0 \simeq \Sigma P^\infty_{-\infty} \vee \Sigma P^\infty_{-\infty} \to \Sigma P^\infty_{-N} \vee \Sigma P^\infty_{-N}$$
is nontrivial. Since this composite is the image of the composite (in $\Sptc$)
$$S^{s,t} \xrightarrow{\alpha'} S^{0,0} \simeq \Sigma^{1,0} \underline{R}^\infty_{-\infty} \to \Sigma^{1,0} \underline{R}^\infty_{-N}$$
under $\Phi^{C_2}$, we conclude that the composite in $\Sptc$ is nontrivial. The result then follows from the same proof as in Lemma \ref{compatibility}.
\item This follows from the previous case and Part $(1)$ of Lemma \ref{compatibility}. 
\end{enumerate}
\end{proof}

We record one other comparison lemma. Combined with the previous results, this completes the proof of Theorem \ref{MT:Squeeze}. 

\begin{lem}\label{Lem:BaseChange}\cite[Lem. 2.6]{Qui19c}
Let $i : \r \to \c$ denote the inclusion of fields. Suppose $\alpha, \beta \in \pi_{**}^\c(S^{0,0})$ such that $M^k(\alpha) = \beta$. Suppose further that there exist $\alpha', \beta' \in \pi_{**}^\r(S^{0,0})$ such that $i^*(\alpha') = \alpha$ and $i^*(\beta') = \beta$. Then 
$$|M^\r(\alpha')| \leq |\beta'|.$$
Further, if $|M^\r(\alpha')| = |\beta'|$, then $M^\r(\alpha') \ni \beta'$. 
\end{lem}

\subsection{Elementary computations II}

We first demonstrate Lemma \ref{compatibility} by computing the $C_2$-equivariant Mahowald invariants of the classes $\eta$ and $\nu$. We will need the following result of Belmont--Guillou--Isaksen \cite{BGI20}, which supersedes previous work of Dugger--Isaksen \cite{DI16}:

\begin{thm}\label{DI16}\cite[Thm. 1.1]{BGI20}
The map $\pi^\r_{s,w}(S^{0,0}) \to \pi^{C_2}_{s,w}(S^{0,0})$ induced by equivariant Betti realization is an isomorphism if $2w - s < 5$ and $(s,w) \neq (0,2)$, and it is injective if $2w-s=5$.  
\end{thm}

\begin{prop}\label{Prop:ElementaryC2}
We have
\begin{enumerate}
\item $\eta \in M^{C_2}(2+\rho \eta) \subset \pi^{C_2}_{1,1}(S)$
\item $\nu \in M^{C_2}(\eta) \subset \pi^{C_2}_{3,2}(S)$
\item $\sigma \in M^{C_2}(\nu) \subset \pi^{C_2}_{7,4}(S)$
\end{enumerate}
\end{prop}

\begin{proof}
We need the following relations:
\begin{enumerate}
\item $U(2 + \rho \eta) = 2$, $U(\eta) = \eta$, $U(\nu) = \nu$, and $U(\sigma) = \sigma$,
\item $Re_{C_2}(2+\rho\eta) = 2$, $Re_{C_2}(\eta) = \eta$, $Re_{C_2}(\nu) = \nu$, and $Re_{C_2}(\sigma) = \sigma$,
\end{enumerate}
Then by Lemma \ref{compatibility} applied to (1), we have 
$$|M^{C_2}(2+\rho\eta)| \leq |\eta| = 1,$$
$$|M^{C_2}(\eta) | \leq |\nu| = 3,$$
$$|M^{C_2}(\nu)| \leq |\sigma| = 7.$$

Suppose that $M^{C_2}(2+\rho\eta) = \gamma$ for some $\gamma \neq \eta$ with $|\gamma| < |\eta|$. Then we have $\gamma : S^{0,0} \to S^{0,-\epsilon}$ where $\epsilon \in \{0,1\}$, i.e. $\gamma \in \pi^{C_2}_{s,w}$ where $s \geq 2w-5$. By Theorem \ref{DI16}, there exists some $\gamma' \in \pi^\r_{**}(S^{0,0})$ such that $|\gamma'| < |\eta|$ and $Re_{C_2}(\gamma') = \gamma$. By (2), this would imply that $M^{\r}(2) \leq |\gamma'| < |\eta|$ which contradicts our computation that $\eta \in M^{\r}(2)$. Therefore we must have $M^{C_2}(2+\rho\eta) = \eta$. 

The same argument applies to prove (2) and (3) in the statement of the proposition.
\end{proof}

\section{The $E_2$-pages of the $\c$- and $\r$-motivic Adams spectral sequences}\label{Sec:Ext}

We have now seen a few elementary computations of $\r$-motivic and $C_2$-equivariant Mahowald invariants. To prove our more complicated results, we will need some more elaborate machinery. We will be especially dependent upon the Adams spectral sequence, so we dedicate this section to discussing its properties we will need in the sequel.  

\subsection{Notation}

We will employ the $(s,f,w)$-trigrading convention for elements in $\Ext_{A^F}^{***}(\m_2^F,\m_2^F)$, where $s$ is stem, $f$ is Adams filtration, and $w$ is motivic weight. Suppose $x \in \Ext^{s,f,w}_{A^F}(\m_2^F,\m_2^F)$. There are five numbers associated to $x$:
\begin{enumerate}
\item Its \emph{stem} $s(x) = s$.
\item Its \emph{Adams filtration} $f(x) = f$.
\item Its \emph{motivic weight} $w(x) = w$.
\item Its \emph{Milnor--Witt degree}, \emph{Milnor--Witt stem}, or \emph{coweight} $mw(x) = s-w$.
\item Its \emph{Chow degree} $c(x) = s+f-2w$. 
\end{enumerate}

If $\alpha \in \pi_{s,w}^F(X)$, then its stem, motivic weight, and Milnor--Witt stem are defined similarly.

\subsection{The $\c$-motivic May spectral sequence}\label{SS:CMay}

The $\c$-motivic Adams spectral sequence 
\begin{equation}\label{Eqn:CASS}
E_2 = \Ext_{A^\c}^{s,f,w}(\m_2^\c,\m_2^\c) \Rightarrow \pi_{s,w}^\c(S^{0,0})
\end{equation}
was first studied extensively by Dugger and Isaksen in \cite{DI10}. Its $E_2$-term is the cohomology of the $\c$-motivic Steenrod algebra which can be calculated using the cobar complex
$$\m_2^\c \xrightarrow{\eta_R - \eta_L} (A^\c)^\vee \longrightarrow (A^\c)^\vee \otimes_{\m_2^\c} (A^\c)^\vee \longrightarrow \cdots.$$

The cobar complex can be filtered by powers of the augmentation ideal $I$ of $(A^\c)^\vee \to \m_2^\c$, and this induces a filtration on $\Ext_{A^\c}^{***}(\m_2^\c,\m_2^\c)$ called the \emph{May filtration}. This gives rise to the \emph{$\c$-motivic May spectral sequence}
\begin{equation}\label{Eqn:CMay}
E_2 = \Ext_{Gr_I(A^\c)}^{m,s,f,w}(\m_2^\c,\m_2^\c) \Rightarrow \Ext_{A^\c}^{s,f,w}(\m_2^\c,\m_2^\c).
\end{equation}

Dugger and Isaksen show that the $E_2$-term of \eqref{Eqn:CMay} can be obtained from the $E_2$-term of the classical May spectral sequence by base-change along $\f_2 \to \m_2^\c$. This implies that the $E_2$-term of \eqref{Eqn:CMay} is the cohomology of the differential graded algebra $\f_2[\tau, h_{i,j}:i>0, j \geq 0]$ with quad-gradings $(m,s,f,w)$ (where $m$ is May filtration)
\begin{equation}\label{Eqn:MayDeg}
|\tau| = (0,0,0,-1), \quad |h_{i,0}| = (i,2^i-2,1,2^{i-1}-1), \text{ and } \quad |h_{i,j}| = (i,2^j(2^i-1)-1,1,2^{j-1}(2^i-1)) \text{ if } j > 0.
\end{equation}

The $\c$-motivic May spectral sequence is calculated up to the $36$-stem in \cite{DI10}. Substantially further calculations can be found in \cite{Isa14} and \cite{IWX20}. 

\subsection{The $\r$-motivic May spectral sequence}\label{SS:RMay}

We now define the $\r$-motivic May spectral sequence following the definition of the $\c$-motivic May spectral sequence in \cite[Sec. 5]{DI10} and the classical May spectral sequence in \cite[Sec. 3.1]{Rav86}. 

Recall that $\Ext^{***}_{A^\r}(\m_2^\r,\m_2^\r)$ may be computed as the cohomology of the cobar complex
$$\m_2^\r \xrightarrow{\eta_R - \eta_L} (A^\r)^\vee \longrightarrow (A^\r)^\vee \otimes_{\m_2^\r} (A^\r)^\vee \longrightarrow \cdots$$
as described in \cite[pg. 6]{DI16a}. The $\r$-motivic May spectral sequence is obtained by filtering the cobar complex as follows. For $i \geq 1$, $j \geq 0$, define
$$h_{i,j} := \begin{cases}
\tau_{i-1} \quad & \text{ if } j=0, \\
\xi_i^{2^{n-1}} \quad & \text{ if } j>0.
\end{cases}
$$
Any monomial in $(A^\r)^\vee$ may be expressed uniquely as a product of powers of $h_{i,j}$, $\rho$, and $\tau$. Define a grading by setting $|h_{i,j}| = 2i-1$, $|\tau| = 0$, $|\rho| = -1$, and extending linearly. Let $F_i((A^\r)^\vee) \subseteq (A^\r)^\vee$ be the sub-comodule consisting of elements of degree less than or equal to $i$. This gives rise to a filtration of the cobar complex by setting
$$F_i(C^k) = \bigoplus_{i_1+\ldots+i_k = i} F_{i_1} \otimes \cdots \otimes F_{i_k}.$$
We will refer to the resulting spectral sequence as the \emph{$\r$-motivic May spectral sequence}
\begin{equation}\label{Eqn:May}
E_1^{m,s,f,w}(\r) \Rightarrow \Ext^{s,f,w}_{A^\r}(\m^\r_2,\m^\r_2).
\end{equation}

\begin{rem2}
The choice $|\rho| = -1$ is chosen to make the filtration multiplicative and to maintain compatibility with the Hopf algebroid structure maps. 

Multiplicativity implies that $|\rho| \leq -1$. Indeed, if the $\r$-motivic May filtration is multiplicative, then there is a pairing
$$\mu : F_1(C^\bullet) \otimes F_1(C^\bullet) \to F_2(C^\bullet)$$
given by $[x] \otimes [y] \mapsto [xy]$ must be compatible with the relation
$$\tau_i^2 = \tau \xi_{i+1} + \rho \tau_{i+1} + \rho \tau_0 \xi_{i+1}$$
in $A^\r$. Consider 
$$\mu([\tau_0] \otimes [\tau_0]) = [\tau_0^2] = [\tau \xi_1] + [\rho \tau_1] + [\rho \tau_0 \xi_1] = \tau h_{11} + \rho h_{20} + \rho h_{10}h_{11}.$$
The degrees of the summands on the right-hand side are given by $|\tau| +1$, $|\rho| + 3$, and $|\rho| + 2$; we are thus forced to take $|\rho| \leq -1$ so that each summand has degree at most $2$. 

In fact, compatibility with the Hopf algebroid structure maps actually implies that $|\rho| \geq -1$. In particular, we have
$$\eta_R(\tau) = \tau + \rho \tau_0.$$
We must take $|\tau| = 0$ to ensure compatibility with the $\c$-motivic May filtration, so $|\rho \tau_0| = 0$. Since $|\rho \tau_0| = |\rho| + |\tau_0| = |\rho| + 1$, we are required to take $|\rho|=-1$. 
\end{rem2}

In \cite[Sec. 5]{DI10}, Dugger and Isaksen identify the $E_1$-page of the $\c$-motivic May spectral sequence, denoted $E_1(\c)$, with the $E_1$-page of the classical May spectral sequence base-changed along the inclusion $\f_2 \hookrightarrow \f_2[\tau] \cong \m_2^\c$. In the $\r$-motivic setting, an identification in terms of the classical May spectral sequence is unclear because the $\r$-motivic analog of \cite[Lem. 5.1]{DI10} is complicated by the presence of $\rho$ in the Adem relations over $Spec(\r)$. Nevertheless, we may identify the $E_1$-page in a range via the following observation.

\begin{lem}
The generator $\rho$ is primitive and exterior. Also, for all $i \geq 1$ and $j \geq 0$, the generator $h_{ij}$ is primitive in the associated graded algebra. Moreover, the generator $h_{ij}$ squares to zero unless $i=1$ and $j=0$, in which case we have $h_{10}^2 = \rho h_{20}.$
\end{lem}

\begin{proof}
Primitivity follows from inspection of the coproduct in $(A^\r)^\vee$. For the second claim, if $j > 0$ then we have
$$h_{ij}^2 = [\xi_i^{2^{j-1}}] \cdot [\xi_i^{2^{j-1}}] = [\xi_i^{2^j}] = h_{i,j+1}.$$
The degree of the left-hand side is $2(2i-1) = 4i-2$ and the degree of the right-hand side is $2i-1$. Since $4i-2 \geq 2i-1$ for all $i \geq 1$, the claim holds in this case. If $j=0$, then we have
$$h_{i0}^2 = [\tau_{i-1}] \cdot [\tau_{i-1}] = [\tau_{i-1}^2] = [\tau \xi_{i}] + [\rho \tau_{i}] + [\rho \tau_0 \xi_{i}] = \tau h_{i1} + \rho h_{i+1,0} + \rho h_{1,0} h_{i,1}.$$
The degree of the left-hand side is $2(2i-1) = 4i-2$ and the degrees of the terms on the right-hand side (from left to right) are $2i-1$, $2(i+1)-1-1 = 2i$, and $1 + 2i-1-1 = 2i-1$. When $i \geq 2$, the left-hand degree is strictly greater than each right-hand degree, but when $i = 1$, the degree of $[\rho \tau_i]$ is $2$ and the degree of $[\tau_{i-1}]^2$ is $2$. Therefore in the associated graded we have the claimed relation. Finally, $\rho$ is exterior since the degree of $\rho^2$ is $-2 < -1$. 
\end{proof}

\begin{rem2}
We will not investigate the $\r$-motivic May spectral sequence in detail, but we include two suggestions for calculating its $E_1$-page $E_1(\r)$. 

First, it seems plausible to calculate the $E_1(\r)$ via a $\rho$-Bockstein spectral sequence \cite{Hil11} of the form
\begin{equation}
E_1 = E_1^{m,s,f,w}(\c)[\rho] \Rightarrow E_1^{m,s,f,w}(\r).
\end{equation}
This spectral sequence is obtained by filtering the cobar complex by powers of $\rho$. Since $\rho = 0$ in $E_1(\c)$, one may be able to mitigate the negative grading on $\rho$ in $E_1(\r)$ via this method. 

Alternatively, $E_1(\r)$ may be accessible via the Cartan-Eilenberg spectral sequence associated to the extension
$$\Phi \to gr_\bullet((A^\r)^\vee) \to E_{\f_2[\tau]}(h_{ij} : i \geq 1, j \geq 0, \text{ and } (i,j) \notin \{(1,0),(2,0)\})$$
where 
$$\Phi = \f_2[\rho,\tau_0,\tau_1]/(\rho^2=0,\tau_1^2=0,\tau_0^2=\rho\tau_1).$$
The $\Ext$-groups over the right-hand side are isomorphic to a polynomial algebra over $\f_2[\tau]$ by \cite{DI10}, but the $\Ext$-groups over the left-hand side $\Ext_\Phi(\f_2,\f_2)$ are more mysterious. 
\end{rem2}

Our main tool for computing in \eqref{Eqn:May} is the map between the $\r$-motivic and $\c$-motivic May spectral sequences
\begin{equation}\label{Eqn:MayMap}
\{E_n(\r)\}_{n \geq 1} \to \{E_n(\c)\}_{n \geq 1}
\end{equation}
induced by the map of Hopf algebroids $(A^\r)^\vee \to (A^\c)^\vee$ which sends $\rho \mapsto 0$, $\tau \mapsto \tau$, and $h_{i,j} \mapsto h_{i,j}$.\footnote{This map induces a map of spectral sequences because the $\r$-motivic and $\c$-motivic May filtrations of the generators $\tau$ and $h_{i,j}$ coincide.}

The following application will be used to study algebraic $v_1$-periodicity operators in the sequel. 

\begin{lem}\label{Lem:b20}
There is an $\r$-motivic May differential $d_4(b_{20}^2) = h_3 h_0^4 + c_0 h_1^2 \rho^3$.
\end{lem}

\begin{proof}
The classes $b_{20}^2$ and $h_3 h_0^4$ are not $\rho$-divisible, so their images under the quotient map $E_1(\r) \to E_1(\c)$ are the (nontrivial) classes with the same name. Since $b_{20}^2$ and $h_3 h_0^4$ both survive to the $E_4$-page of the $\c$-motivic May spectral sequence, the same holds in the $\r$-motivic May spectral sequence. Furthermore, there is a $\c$-motivic May differential $d_4(b_{20}^2) = h_3 h_0^4$ by \cite[Sec. 5.3]{DI10}. Examining \cite[Fig. 3, MW=3]{DI16a}, we see that there is a relation $h_3 h_0^4 = c_0 h_1^2 \rho^3$. The only way of achieving this relation in $\Ext_{A^\r}^{***}(\m_2^\r,\m_2^\r)$ which is compatible with the $\c$-motivic May differential is for there to be an $\r$-motivic May differential $d_4(b_{20}^2) = h_3 h_0^4 + c_0 h_1^2 \rho^3$. 
\end{proof}

\subsection{The $\rho$-Bockstein spectral sequence}\label{SS:rhoBSS}

The standard approach to calculating the $E_2$-term of the $\r$-motivic Adams spectral sequence $\Ext_{A^\r}^{***}(\m_2^\r,\m_2^\r)$ is via the \emph{$\rho$-Bockstein spectral sequence} 
\begin{equation}\label{Eqn:rhoBSS}
E_1 = \Ext_{A^\c}^{***}(\m_2^\c,\m_2^\c)[\rho] \Rightarrow \Ext_{A^\r}^{***}(\m_2^\r,\m_2^\r)
\end{equation}
introduced by Hill in \cite{Hil11}. This spectral sequence is obtained from the filtration of the cobar complex computing $\Ext_{A^\r}^{***}(\m_2^\r,\m_2^\r)$ induced by the filtration of $\m_2^\r = \m_2^\c[\rho]$ by powers of $\rho$. 

The $\rho$-Bockstein spectral sequence computing the cohomology of $A^\r$ was first considered by Dugger--Isaksen in \cite{DI16a}, where they applied it to compute the first $4$ Milnor--Witt stems of $\pi_{**}^\r(S^{0,0})$. More recently, Belmont and Isaksen \cite{BI20} have computed the first $11$ Milnor--Witt stems. In both cases, many $\rho$-Bockstein differentials are deduced using the following theorem which we will need in the sequel.

\begin{thm}\label{Thm:Double}\cite[Thm. 4.1]{DI16a}
There is an isomorphism from $\Ext_{A^{cl}}^{**}(\f_2,\f_2)[\rho^{\pm 1}]$ to $\Ext^{***}_{A^\r}(\m_2^\r,\m_2^\r)[\rho^{-1}]$ such that 
\begin{enumerate}
\item The isomorphism is highly structured, i.e., preserves products, Massey products, and algebraic squaring operations in the sense of \cite{May70}.
\item The element $h_n \in \Ext_{A^{cl}}^{**}(\f_2,\f_2)$ corresponds to the element $h_{n+1}$ of $\Ext^{***}_{A^\r}(\m_2^\r,\m_2^\r)$.
\item An element of $\Ext^{**}_{A^{cl}}(\f_2,\f_2)$ in stem $s$ and Adams filtration $f$ corresponds to an element in $\Ext^{***}_{A^\r}(\m_2^\r,\m_2^\r)$ of stem $2s+f$, Adams filtration $f$, and motivic weight $s+f$. 
\end{enumerate}
\end{thm}

\subsection{Vanishing and periodicity}\label{SS:Exth1}

Our construction of periodic families in $\Ext_{A^\r}^{***}(\m_2^\r,\m_2^\r)$ will rely on certain vanishing and periodicity results in $\Ext_{A^k}^{***}(\m_2^k,\m_2^k)$. 

The first such result is Guillou and Isaksen's $\c$-motivic vanishing line of slope $1/2$, which generalizes Adams' vanishing line of slope $1/2$ in $\Ext_{A^{cl}}^{**}(\f_2,\f_2)$ from \cite{Ada61}. 

\begin{thm}\label{Thm:GIVanishing}\cite[Thm. 1.1]{GI15a}
Let $s>0$. The map
$$h_1 : \Ext_{A^\c}^{s,f,w}(\m_2^\c,\m_2^\c) \to \Ext_{A^\c}^{s+1,f+1,w+1}(\m_2^\c,\m_2^\c)$$
is an isomorphism if $f \geq \frac{1}{2} s + 2$, and it is a surjection if $f \geq \frac{1}{2} s + \frac{1}{2}$. 
\end{thm}

In other words, all of the elements in $\Ext_{A^\c}^{***}(\m_2^\c,\m_2^\c)$ above a line of slope $1/2$ are $h_1$-periodic. The behavior of these $h_1$-periodic families was investigated further by Guillou and Isaksen in \cite{GI15}. 

\begin{thm}\label{Thm:GIetaC}\cite[Thm. 1.1]{GI15}
The $h_1$-localized algebra $\Ext_{A^\c}^{***}(\m_2^\c,\m_2^\c)[h_1^{\pm 1}]$ is a polynomial algebra over $\f_2[h_1^{\pm 1}]$ on generators $v_1^4$ and $v_n$ for $n \geq 2$, where:
\begin{enumerate}
\item $v_1^4$ is in the $8$-stem and has Adams filtration $4$.
\item $v_n$ is in the $(2^{n+1}-2)$-stem and has Adams filtration $1$. 
\end{enumerate}
\end{thm}

Building on this result, Guillou and Isaksen then calculated the $h_1$-localized cohomology of the $\r$-motivic Steenrod algebra using the $\rho$-Bockstein spectral sequence in \cite{GI16}. The final answer is not needed in the sequel, but we mention it since many of the ideas in our $\Ext$ group calculations are based on ideas implemented by Guillou and Isaksen in \cite{GI16}. 

Finally, we will need the following result of Ang Li regarding $v_1$-periodicity over $\Spec(\c)$. 

\begin{thm}\label{Thm:Ang}\cite[Thm. 1.3]{Li19}
For $r \geq 2$, the Massey product operation $P_r(-) := \langle h_{r+1}, h_0^{2^r}, - \rangle$ is uniquely defined on $\Ext^{s,f,w}_{A^\c}(\m_2^\c,\m_2^\c)$ when $s>0$ and $f> \frac{1}{2}s + 3-2^r$. 

Furthermore, for $f> \frac{1}{5}s + \frac{12}{5}$, the restriction of $P_r$ to the $h_1$-torsion
$$P_r : (Ext^{s,f,w}_{A^\c}(\m_2^\c,\m_2^\c))_{h_1\text{-torsion}} \to (\Ext_{A^\c}^{s+2^{r+1},f+2^r,w+2^r}(\m_2^\c,\m_2^\c))_{h_1\text{-torsion}}$$
is an isomorphism. 
\end{thm}

\section{$v_1$-periodicity over $\r$ and generalized Mahowald invariants of $(2+\rho\eta)^i$}\label{Sec:2i}

We are now prepared to prove our first main theorems, Theorem \ref{MT:Rv1} and \ref{MT:Ev1}. In this section, we study $v_1$-periodicity over $\r$, compute $M^{\r}((2+\rho\eta)^i)$ for $i \equiv 2,3\mod 4$, and compute $M^{C_2}((2+\rho\eta)^i)$ for $i \equiv 2,3 \mod 4$. This section benefited greatly from discussions with Jonas Irgens Kylling, whom we especially thank for his suggestion to use the cobar complex and matric Massey products in order to produce certain $v_1$-periodic families.

\subsection{Upper bounds: $v_1$-periodic classes in $\pi_{**}^{\r}(S^{0,0})$ and $\pi_{**}^{C_2}(S^{0,0})$}

We initiate our computations by constructing $v_1$-periodic classes in $\pi_{**}^\r(S^{0,0})$ and $\pi_{**}^{C_2}(S^{0,0})$ which will be contained in the Mahowald invariants $M^\r((2+\rho\eta)^i)$ and $M^{C_2}((2+\rho \eta)^i)$ for $i \equiv 2,3 \mod 4$. Our approach is to define the $\r$-motivic classes as lifts of their $\c$-motivic counterparts using the motivic Adams spectral sequence, and then define the $C_2$-equivariant classes via equivariant Betti realization. 

In \cite{Qui17}, we showed that $M^\c(2^{4i+j})$ contains $v_1^{4i}\tau \eta^j$ for $j \in \{2, 3\}$. To construct these elements, we defined an algebraic periodicity operator $P_\c(-)$  (denoted $P(-)$ in \cite{Qui17}) on the $h_0^4$-torsion in $\Ext_{A^\c}^{***}(\m_2^\c,\m_2^\c)$ by 
$$P_\c(x) := \langle h_3, h_0^4, x \rangle.$$
The operator $P(-)$ detects $v_1^4$ in the $\c$-motivic Adams spectral sequence, or more precisely, for any $v_1$-periodic class $\alpha \in \pi_{s,w}^\c(S^{0,0})$ detected by $x \in \Ext_{A^\c}^{s,f,w}(\m_2^\c,\m_2^\c)$ considered in \cite{Qui17}, the subset $P(x) \subseteq \Ext_{A^\c}^{s+8,f+4,w+4}(\m_2^\c,\m_2^\c)$ contains precisely one nontrivial element, and that element detects $v_1^4\alpha \in \pi_{s+8,w+4}^\c(S^{0,0})$. 

Our first goal in this section is to define an $\r$-motivic analog of $P_\c(-)$. This turns out to be more subtle than in the $\c$-motivic setting: examination of \cite[Fig. 3]{DI16a} reveals that $h_3 h_0^4 \neq 0$ in $\Ext^{***}_{A^\r}(\m_2^\r,\m_2^\r)$, so the operator $\langle h_3, h_0^4, - \rangle$ is not defined in $\Ext^{***}_{A^\r}(\m_2^\r,\m_2^\r)$. Nevertheless, we can define an $\r$-motivic periodicity operator as follows.

Recall that the relation $h_3h_0^4 = 0 \in \Ext_{A^\c}^{***}(\m_2^\c,\m_2^\c)$ is produced by the $\c$-motivic May differential $d_4(b_{20}^2) = h_3 h_0^4$. We observed in Lemma \ref{Lem:b20} that there is an $\r$-motivic May differential $d_4(b_{20}^2) = h_3 h_0^4 + c_0 h_1^2 \rho^3$.  Since we want our $\r$-motivic periodicity operator to be compatible with the $\c$-motivic periodicity operator under \eqref{Eqn:MayMap}, we are thus led to define an $\r$-motivic periodicity operator using a matric Massey product which relies on the relation $h_3 h_0^4 + c_0h_1^2\rho^3$. 

\begin{defin}
We define the $\r$-motivic periodicity operator $P_\r(-)$ on the elements $x \in \Ext^{***}_{A^\r}(\m_2^\r,\m_2^\r)$ which are both $h_0^4$-torsion and $c_0$-torsion as the matric Massey product
$$P_\r(x) = \left\langle \begin{bmatrix} h_3 & \rho^3 h_1^2 \end{bmatrix}, \begin{bmatrix} h_0^4 \\  c_0 \end{bmatrix}, x \right\rangle.$$
\end{defin}

\begin{lem}\label{Lem:pExt}
Let 
\begin{equation}
p : \Ext^{***}_{A^\r}(\m_2^\r,\m_2^\r) \to \Ext^{***}_{A^\c}(\m_2^\c, \m_2^\c)
\end{equation}
be the map induced by the map $A^\r \to A^\c$ which is the identity on the generators and $\tau$, and sends $\rho$ to $0$. There is an equality of cosets $P_\r(x) = P_\c(p(x))$.
\end{lem}

\begin{proof}\footnote{The proof of this lemma in a previous version contained some gaps. We thank an anonymous referee for pointing out the significantly simpler proof given here.}
By naturality of matric Massey products, the image of the matric Massey product $P_\r(x)$ under $p$ is 
$$\left\langle \begin{bmatrix} h_3 & 0 \end{bmatrix}, \begin{bmatrix} h_0^4 \\  c_0 \end{bmatrix}, p(x) \right\rangle$$
which reduces to the Massey product $\langle h_3, h_0^4, p(x)\rangle = P_\c(p(x))$. 
\end{proof}

We now iteratively apply $P_\r(-)$ to construct two infinite families of elements in $\Ext_{A^\r}^{***}(\m_2^\r,\m_2^\r)$. Let $P^i_\r(-)$ denote the $i$-fold composite of $P_\r(-)$. 

\begin{prop}\label{Cor:Pntau}
For all $i \geq 0$ and $j \in \{2,3\}$, the iterated matric Massey products $P^i_\r(\tau h_1^j)$ is defined, nontrivial, and has zero indeterminacy. Its image is under $p$ is $P^i_\c(\tau h_1^j)$. 
\end{prop}

\begin{proof}
We proceed by induction on $i$. When $i=0$, the result is equivalent to the fact that $\tau h_1^j \neq 0 \in \Ext_{A^\r}^{***}(\m_2^\r,\m_2^\r)$, which follows from \cite{DI16a}. 

Suppose now for the sake of induction that the result holds for $i < n$. Suppose further that $h_0^4 P^{n-1}(\tau h_1^j) = 0$ and $c_0 P^{n-1}(\tau h_1^j)=0$, so the Massey product $P(P^{n-1}(\tau h_1^j))$ is defined. Then by Lemma \ref{Lem:pExt}, the image of $P^n(\tau h_1^j) = P(P^{n-1}(\tau h_1^j))$ under $p$ is $\langle h_3, h_0^4, p(P^{n-1}(\tau h_1^j))\rangle$. By the induction hypothesis, $p(P^{n-1}(\tau h_1^j))$ contains $P^{n-1}_\c(\tau h_1^j)$, so we may rewrite $\langle h_3, h_0^4, p(P^{n-1}(\tau h_1^j)) \rangle$ as $\langle h_3, h_0^4, P^{n-1}_\c(\tau h_1^j)\rangle$. But this is just $P^n_\c(\tau h_1^j)$, which was shown to be nonzero in \cite{Qui17}. Therefore the image of $P^n(\tau h_1^j)$ under $p$ is nontrivial, so $P^n(\tau h_1^j)$ must contain a nonzero element whose image is $P^n_\c(\tau h_1^j)$. 

We now check that there is no indeterminacy in $P^{n-1}_\r(\tau h_1^j)$. The indeterminacy lies in
$$h_3 \Ext_{A^\r}^{8n-6+j,4n-4+j,4n-5+j}(\m_2^\r,\m_2^\r) \quad \text{and} \quad \rho^3 h_1^2 \Ext_{A^\r}^{8n+2+j, 4n-7+j,4n+j}(\m_2^\r,\m_2^\r).$$
We will show that these indeterminacy groups vanish by using the $\rho$-Bockstein spectral sequence to reduce to the $\c$-motivic setting, and then we will employ the $\c$-motivic May spectral sequence to rule out the existence of nonzero elements.

In the first case, we must consider the possible elements
$$x \in h_3 \Ext_{A^\r}^{8n-6+j,4n-4+j,4n-5+j}(\m_2^\r,\m_2^\r) \subseteq \Ext^{8n+1+j,4n-3+j,4n-1+j}_{A^\r}(\m_2^\r,\m_2^\r).$$
We observe that $x$ must satisfy three hypotheses: $mw(x) = 4n+2$, $f(x)=4n-3+j$, and $x$ is divisible by $h_3$. Suppose for simplicity that $j=3$ so $f(x) = 4n$; the case $j=2$ is similar. Using the $\rho$-Bockstein spectral sequence, we have that $x$ is detected by $\rho^i y$ for some $i \geq 0$ and $y \in \Ext_{A^\c}^{***}(\m_2^\c,\m_2^\c)$. Since $mw(\rho) = 0$ and $f(\rho)=0$, we see that $mw(y) = 4n+2$, $f(y) = 4n$, and $y$ is divisible by $h_3$. We will argue that any such $y$ must be zero by showing that no nonzero element in the $\c$-motivic May spectral sequence could detect $y$. By Theorem \ref{Thm:GIetaC}, we may assume that $y$ is $\eta$-torsion. 

Now, we will use the $\c$-motivic May spectral sequence to argue $y$ is zero. We observe that 
\begin{equation}
\begin{split}
mw(h_{i0}) \begin{cases}
=0 \quad & \text{ if } i = 1,\\
=1 \quad & \text{ if } i=2, \\
=3 \quad & \text{ if } i=3, \\
\geq 7 \quad & \text{ if } i \geq 4,
\end{cases}
\end{split}
\quad \quad \text{and} \quad \quad
\begin{split}
mw(h_{ij}) \begin{cases}
=1 \quad & \text{ if } (i,j) = (1,1), \\
=2 \quad & \text{ if } (i,j)= (1,2), \\
=4 \quad & \text{ if } (i,j)=(1,3), \\
=3 \quad & \text{ if } (i,j) = (2,1),\\
\geq 6 \quad & \text{ otherwise,}
\end{cases}
\end{split}
\end{equation}
where $j>0$ in the right-hand column. Moreover, we have $f(h_{ij}) = 1$ for all $i,j$. Crucially, we note that the only way to increase filtration without increasing Milnor--Witt degree is by $h_0$ multiplication. In particular, since $y$ is divisible by $h_3$ and $d_4(b_{20}^2) = h_3h_0^4$ in the $\c$-motivic May spectral sequence \cite{DI10}, $y$ cannot be divisible by $h_0^i$ for any $i \geq 4$. This severely restricts the possible forms of $y$, since multiplication by an arbitrary May generator could increase the Milnor--Witt stem too much relative to the filtration. Let us make this claim precise. Suppose $y$ has the form $y = h_0^3 h_3 z$ with $z$ not divisible by $h_0$; if $y = h_0^i h_3 z$ with $i <3$, the argument is similar. We have $mw(z) = 4n-1$ and $f(z) = 4n-4$, so $z$ contains a factor $w$ with $mw(w) = f(w)+3$. The only way this could occur is if $w \in \{ h_{12}^3, h_{21}h_{12}, h_{12}h_{30}, h_{13} \}$. The remaining factors of $z$ must be powers of $h_{20}$ and $h_{11}$. Inspection of the $\c$-motivic May spectral sequence reveals that there is no nonzero $y$ with the desired properties contributing to $\Ext_{A^\c}^{s,f,w}(\m_2^\c,\m_2^\c)$ for $s \leq 20$ and $f \leq 10$, so by Theorem \ref{Thm:Ang}, there are no nonzero $y$ in all of $\Ext^{***}_{A^\c}(\m_2^\c,\m_2^\c)$. This shows that the first indeterminacy group is zero.

A similar argument applies for the second indeterminacy group, where we must consider the elements
$$x \in \rho^3 h_1^2 \Ext_{A^\r}^{8n+2+j,4n-7+j,4n+j}(\m_2^\r,\m_2^\r) \subseteq \Ext_{A^\r}^{8n+1+j,4n-5+j,4n-1+j}(\m_2^\r,\m_2^\r).$$
In this case, $x$ must satisfy the hypotheses that $mw(x) = 4n+2$, $f(x) = 4n-5+j$, and $x = \rho^{3+i} h_1^2 y$ for some $i \geq 0$ and $y \in \Ext_{A^\c}^{***}(\m_2^\c,\m_2^\c)$. Since $d_1(h_{20}) = h_0h_1$ in the $\c$-motivic May spectral sequence, we see that $y$ cannot be divisible by $h_0$. As in the previous case, this allows us to explicitly list all possible forms of $y$ in terms of $\c$-motivic May generators. By Theorem \ref{Thm:GIetaC}, we may restrict to $h_1$-torsion, and an exhaustive search of $\Ext_{A^\c}^{***}(\m_2^\c,\m_2^\c)$ in low stem and filtration, along with Theorem \ref{Thm:Ang}, imply that there are no nonzero elements with this form. Therefore the second indeterminacy group is zero. 

Finally, we tie up loose ends by showing that $h_0^4 P^{n-1}(\tau h_1^j) = 0$ and $c_0 P^{n-1}(\tau h_1^j) = 0$. In both cases, we prove the result by juggling using \cite[Thm. A1.4.6]{Rav86}. Let $x$ denote $h_0^4$ or $c_0$. We have
\begin{align*}
x P^{n-1}(\tau h_1^j) &= \left\langle \begin{bmatrix} h_3 & \rho^3 h_1^2 \end{bmatrix}, \begin{bmatrix} h_0^4 \\  c_0 \end{bmatrix}, P^{n-2}(\tau h_1^j) \right\rangle x \\
& \subseteq \left\langle \begin{bmatrix} h_3 & \rho^3 h_1^2 \end{bmatrix}, \begin{bmatrix} h_0^4 \\  c_0 \end{bmatrix}, x P^{n-2}(\tau h_1^j) \right\rangle \\
& = \left\langle \begin{bmatrix} h_3 & \rho^3 h_1^2 \end{bmatrix}, \begin{bmatrix} h_0^4 \\  c_0 \end{bmatrix}, 0 \right\rangle \subseteq \{0\},
\end{align*}
where the induction hypothesis is applied in the passage from the penultimate to ultimate lines. The analysis from the preceding paragraphs can be used to show that there is no indeterminacy in these matric Massey products (roughly speaking, the indeterminacy lies in close enough tridegrees that the relevant periodicity and vanishing theorems can be applied). This complete the proof. 
\end{proof}

We will identify the coset $P^i_\r(\tau h_1^j)$ with its unique nontrivial element from now on. The previous proposition gives us two infinite families in the $E_2$-page of the $\r$-motivic Adams spectral sequence.; our goal now is to show that these families survive to $E_\infty$. We will need the following technical lemma to exclude certain differentials. 

\begin{lem}\label{Lem:rho2}
For all $i \geq 0$ and $j \in \{2,3\}$, $P^i_\r \tau h_1^j \subseteq \Ext_{A^\r}^{***}(\m_2^\r,\m_2^\r)$ is annihilated by $\rho^2$. 
\end{lem}

\begin{proof}
We begin by showing the result holds in the associated graded of the $\rho$-Bockstein filtration of $\Ext_{A^\r}^{***}(\m_2^\r,\m_2^\r)$ by induction on $i$. When $i =0$, there are $\rho$-Bockstein differentials $d_2(\tau^2h_1) = \rho^2\tau h_1^2$ and $d_1(\tau h_2 h_0) = \rho \tau h_1^3$. Assume for the sake of induction that there are $\rho$-Bockstein differentials $d_2(P^i_\r(\tau^2 h_1)) = P^i_\r \rho^2 \tau h_1^2$ and $d_1(P^i \tau h_2 h_0) = P^i_\r \rho \tau h_1^3$ for all $i < n$. Applying Moss's higher Leibniz rule \cite[Thm. 1.1]{Mos70} for the $\rho$-Bockstein spectral sequence, we find that
$$d_2(P^i_\r \tau^2 h_1) \subseteq P_\r(d_2(P^{i-1}_\r(\tau^2 h_1))),$$
$$d_1(P^i_\r \tau h_2 h_0) = P_\r(d_1(P^{i-1}_\r(\tau h_2h_0)))$$
using the fact that each element used to define $P_\r$ is a $d_1$- and $d_2$-cycle in the $\rho$-Bockstein spectral sequence. By the induction hypothesis, the right-hand sides of the inclusions above are just $P^i_\r \tau h_1^2$ and $P^i_\r \tau h_1^3$, respectively. 

We now show that the result actually holds in $\Ext$. Fix $i$ and write $x$ for $P^i_\r \tau h_1^2$. By the previous paragraph, it suffices to show that there is no element in higher $\rho$-Bockstein filtration which could detect $\rho^2 x$. Suppose there were such an element. Then it would be detected by $\rho^b z$ in the $\rho$-Bockstein spectral sequence where $z$ is not divisible by $\rho$. We claim there are no elements $z$ in the appropriate tridegrees for this to occur. Indeed, if $\rho^b z = \rho^2 x$, then for $\rho^b z$ we have
$$s = 8i, \quad f = 4i+2, \quad w = 4i-1, \quad mw = 4i+1, \quad c = 4i+4$$
and thus for $z$ we have
$$s=8i+b, \quad f=4i+2, \quad w=4i-1+b, \quad mw=4i+1, \quad c=4i+4-b.$$

Since $z$ is not divisible by $\rho$, it is detected by some polynomial in the $\c$-motivic May spectral sequence generators $h_{ij}$, which we recall have tridegrees $(2^i-2,1,2^{i-1}-1)$ if $j=0$ and $(2^j(2^i-1)-1,1,2^{j-1}(2^i-1))$ if $j>0$. In particular, we have
$$mw(h_{i0})=2^{i-1}-1, \quad c(h_{i0})=1, \quad mw(h_{ij}) = 2^{j-1}(2^i-2), \quad c(h_{ij}) = 0.$$

We see that $z$ contains precisely $4i+4-b$ terms of the form $h_{i0}$ since $c(z) = 4i+4-b$ and only those terms can increase Chow degree. It follows that $mw(z) \equiv b \mod 2$, and since $mw(z) \equiv 1 \mod 2$, we must have $b \equiv 1 \mod 2$. Then $w(z)$ must be even.

However, we have already observed that $z$ is a polynomial in $h_{ij}$'s with precisely $4i+4-b$ factors of the form $h_{i0}$. Since $b$ is odd and $h_{i0}$ has odd weight, the total weight of these factors is odd. But $w(h_{ij})$ is even for $j>0$, so the total weight of $z$ must be odd, contradicting the requirement above. Thus there can be no $z$ with $\rho^b z$ detecting $\rho^2 \tau h_1^2$. 

A similar argument works if we take $x$ to be $P^i_\r \tau h_1^3$. In this case, we show that no element in higher $\rho$-Bockstein filtration can detect $\rho x$ instead of $\rho^2 x$. The choice $\rho x$ instead of $\rho^2 x$ allows one to apply similar modular arithmetic to rule out the existence of some $z$ with $\rho^b z = \rho x$.
\end{proof}

\begin{prop}
For all $i \geq 0$ and $j \in \{2,3\}$, $P^i_\r \tau h_1^j$ survives in the $\r$-motivic Adams spectral sequence. Thus, it detects a nontrivial class in $\pi_{**}^\r(S^{0,0})$. 
\end{prop}

\begin{proof}
Fix $i$ and $j$ as above and let $x$ denote $P^i_\r \tau h_1^j$. We will show that $x$ is not the target of a differential and that it is a permanent cycle, so it survives in the $\r$-motivic Adams spectral sequence. 

We begin by showing $x$ is not the target of a differential by using base-change. By construction, $x$ base-changes to the unique nontrivial element in $P^i_\c \tau h_1^j$ in $\Ext_{A^\c}^{***}(\m_2^\c,\m_2^\c)$. Since the latter survives in the $\c$-motivic Adams spectral sequence and base-change induces a map between the $\r$-motivic and $\c$-motivic Adams spectral sequences, we conclude that $x$ cannot be the target of a nontrivial differential in the $\r$-motivic Adams spectral sequence. 

We now show that $x$ is a permanent cycle in three steps. Let $d_r(x) = y$ denote a generic differential in the $\r$-motivic Adams spectral sequence. First, we invoke Lemma \ref{Lem:rho2} to show that $x$ is $\rho^2$-torsion, which forces $y$ to be $\rho^2$-torsion. Second, we show that $y$ must be divisible by $\rho$ using the $\rho$-Bockstein spectral sequence and $\c$-motivic vanishing results of Guillou--Isaksen \cite{GI15a}. Finally, we show that any class $y$ in the relevant tridegree which is $\rho^2$-torsion and $\rho$-divisible must actually be zero using $h_1$-periodic calculations of Guillou--Isaksen \cite{GI15,GI16}.

We have $x \in \Ext_{A^\r}^{8i+j,4i+j,4i+j-1}(\m_2^\r,\m_2^\r)$ and potential targets of a nontrivial differential have the form $y \in \Ext_{A^\r}^{8i+j-1,4i+j+r,4i+j-1}(\m_2^\r,\m_2^\r)$ with $r \geq 2$. 

First, we have by Lemma \ref{Lem:rho2} that $\rho^2 x = 0$ in $\Ext_{A^\r}^{***}(\m_2^\r,\m_2^\r)$. Therefore $y$ must be $\rho^2$-torsion as well. 

Second, we claim that $y$ is divisible by $\rho$. Suppose not. Then $y$ is detected by a class $y' \in \Ext_{A^\c}^{8i+j-1,4i+j+r,4i+j-1}(\m_2^\c,\m_2^\c)$ in the $\rho$-Bockstein spectral sequence, and any such class is $h_1$-torsion free by Theorem \ref{Thm:GIVanishing}. It follows that $y$ itself must be $h_1$-torsion free. But $x$ is $h_1$-torsion in the $E_2$-page of the $\r$-motivic Adams spectral sequence, which follows from juggling and the relation $\tau h_1^{j+2} = 0$ for $j \geq 2$. Therefore we cannot have $d_r(x) = y$.

We have now shown that the potential target of a differential $y$ is both $\rho$-torsion, divisible by $\rho$, and $h_1$-torsion. We will now argue that any $y$ with these properties must be zero. Since $y$ is $\rho$-torsion, there must have been a $\rho$-Bockstein differential $d_s(w)=\rho^2 y$ with $w$ $\rho$-torsion free in the $E_s$-page of the $\rho$-Bockstein spectral sequence. Moreover, we can calculate that $w$ would have contributed to $\Ext_{A^\r}^{8i+j-2,4i+j+r-1,4i+j-1}(\m_2^\r,\m_2^\r)$ since $\rho$-Bockstein differentials decrease stem by one, increase Adams filtration by one, and preserve motivic weight. 

We claim that $w$ must be divisible by $\rho$. Before proving the claim, we explain how this implies $y$ is zero. Observing that $\rho^2 y = \rho \cdot \rho y$, we can apply $\rho$-linearity of $\rho$-Bockstein differentials to conclude that that $d_s(w/\rho) = \rho y$. Iterating the argument from the previous paragraph shows that $d_s(w/\rho^2) = y$, thus proving that $y$ is zero.

Finally, we prove the claim. The tridegree $(8i+j-2,4i+j+r-1,4i+j-1)$ to which $w$ would contribute satisfies $f \geq \frac{1}{2}s + \frac{1}{2}$, so $w$ is either an $h_1$-periodic class in the $\rho$-Bockstein spectral sequence or a $\rho$ multiple of some $h_1$-torsion class by Theorem \ref{Thm:GIVanishing}. The latter case is what we wanted to prove, so we just need to rule out the possibility that $w$ is an $h_1$-periodic class in $\Ext_{A^\c}^{***}(\m_2^\c,\m_2^\c)$. 

The fate of these $h_1$-periodic classes in the $\rho$-Bockstein spectral sequence was completely determined by Guillou and Isaksen. By \cite[Table 1]{GI16}, the generators of the $E_1$-page of the $h_1$-local $\rho$-Bockstein spectral sequence are $\rho$, $v_1^4$, and $v_n$, $n \geq 2$. The table also contains the Milnor--Witt stem, $s-w$, and Chow degree, $s+f-2w$, of each generator. For every generator except $\rho$,  the Milnor--Witt stem is greater than or equal to the Chow degree. However, the Milnor--Witt stem of $w$ is $4i-1$ and the Chow degree of $w$ is $4i+r-1$. Therefore any infinite $h_1$-tower which could include $w$ must be $\rho$-divisible. 
\end{proof}

\begin{defin}
We define $v^{4i}_1 \tau \eta^2$ and $v^{4i}_1 \tau \eta^3$ to be the classes in $\pi_{**}^\r(S^{0,0})$ detected by $P^i_\r(\tau h_1^2)$ and $P^i_\r(\tau h_1^3)$, respectively, which were constructed in Corollary \ref{Cor:Pntau}. 
\end{defin}

\begin{defin}
Let $v^{4i}_1 \tau \eta^2$ and $v^{4i}_1 \tau \eta^3$ in $\pi_{**}^{C_2}(S^{0,0})$ be the images of the classes with the same name in $\pi_{**}^{\r}(S^{0,0})$ under equivariant Betti realization.
\end{defin}

\begin{cor}
The classes $v^{4i}_1 \tau \eta^2$ and $v^{4i}_1 \tau \eta^3$ are nonzero in $\pi_{**}^{C_2}(S^{0,0})$. 
\end{cor}

\begin{proof}
The images of $v^{4i}_1 \tau \eta^2$ and $v^{4i}_1 \tau \eta^3$ under the forgetful functor $\Sptc \to \Spt$ are the classical elements $v^{4i}_1 \eta^2$ and $v^{4i}_1 \eta^3$, respectively. 
\end{proof}

In summary, we have constructed two infinite $v_1$-periodic families in the $\r$-motivic and $C_2$-equivariant stable stems as lifts of $\c$-motivic and classical $v_1$-periodic families. These $v_1$-periodic families appeared as Mahowald invariants in previous work of Mahowald--Ravenel \cite{MR93} and the author \cite{Qui17}:

\begin{thm}\label{Lem:OldMI}
For $i \geq 0$ and $j \in \{2, 3\}$, the following inclusions hold:
\begin{enumerate}
\item \cite{MR93} $v_1^{4i}\eta^j \in M(2^{4i+j})$. 
\item \cite{Qui17} $v_1^{4i}\tau \eta^j \in M^{\c}(2^{4i+j})$. 
\end{enumerate}
\end{thm}

\begin{prop}\label{Prop:UB}
For $i \geq 0$ and $j \in \{2, 3\}$, the following statements hold:
\begin{enumerate}
\item We have $|M^\r((2+\rho \eta)^{4i+j})| \leq |v_1^{4i}\tau \eta^j|$. If equality holds, then $v_1^{4i}\tau \eta^j \in M^\r((2+\rho \eta)^{4i+j})$. 
\item We have $|M^{C_2}((2+\rho \eta)^{4i+j})| \leq |v_1^{4i} \tau \eta^j|$. If equality holds, then $v_1^{4i} \tau \eta^j \in M^{C_2}((2+\rho \eta)^{4i+j})$. 
\end{enumerate}
\end{prop}

\begin{proof}
In both cases, we use Lemma \ref{Lem:OldMI} and compatibilty of generalized Mahowald invariants under certain functors.
\begin{enumerate}
\item Let $i : \r \to \c$ and let $i^* : SH_\r \to SH_\c$ denote the corresponding base-change functor. Then $i^*(2+\rho\eta)^{4i+j} = 2^{4i+j}$ and $i^*v_1^{4i}\tau \eta^j = v_1^{4i}\tau\eta^j$ in $\pi_{**}^\c(S^{0,0})$. The claim then follows from (2) in Lemma \ref{Lem:OldMI} along with Lemma \ref{Lem:BaseChange}. 
\item Recall that $U : SH_{C_2} \to SH$ is the underlying spectrum functor. We then have that $U((2+\rho \eta)^{4i+j}) = 2^{4i+j}$ and $U(v_1^{4i}\tau \eta^j) = v_1^{4i} \eta^j$ in $\pi_*(S^0)$. The claim then follows from (1) in Lemma \ref{Lem:OldMI} along with Lemma \ref{compatibility}.
\end{enumerate}
\end{proof}

\subsection{Lower bounds: $\r$-motivic homotopy of stunted lens spectra}

We now prove that the dimension inequalities of Proposition \ref{Prop:UB} are equalities. To do so, we show that for any $N \in \z$, if the composite
\begin{equation}\label{Eqn:2iComposite}
S^{0,0} \xrightarrow{(2+\rho \eta)^i} S^{0,0} \to \Sigma^{1,0} \underline{L}^\infty_{-N+1}
\end{equation}
is null, then the composite 
\begin{equation}\label{Eqn:2i4}
S^{0,0} \xrightarrow{(2+\rho \eta)^{i+4}} S^{0,0} \to \Sigma^{1,0} \underline{L}^\infty_{-N-4+1}
\end{equation}
is also null. This implies that $|M^\r((2+\rho \eta)^{i+4})| \geq |M^\r((2+\rho \eta)^i| + 8$ for all $i$. 

\begin{prop}\label{Prop:16Null}
Let $N \in \z$. If the composite \eqref{Eqn:2iComposite} is null, then the composite \eqref{Eqn:2i4} is also null. 
\end{prop}

\begin{proof}
Analogous results were obtained classically \cite{MR93} and $\c$-motivically \cite{Qui17} using a result of Toda \cite{Tod63} about the degree of the identity map between projective spaces. In particular, we see by base changing to $\Spec(\c)$ that the only way that $(2+\rho \eta)^{i+4}$ in \eqref{Eqn:2i4} could be essential is if it is detected in the Atiyah--Hirzebruch spectral sequence by some element $\alpha[m,n]$ with $\alpha \in \pi_{**}^\r(S^{0,0})$ base changing to zero. 

We begin by reducing to proving the proposition for just one $N \in \z$. First, James periodicity implies that the attaching maps of $\underline{L}^\infty_{-\infty}$ relevant to our arguments below are $4$-periodic with respect to $N$. Therefore if the proposition holds for some $N \in \z$, it also holds any $N'$ with $N' \equiv N \mod 4$. Second, we claim that if the proposition holds for one congruence class of $N \mod 4$, then it holds for all congruence classes. To see this, suppose for contradiction that the result holds for $N$ but not for $N'$ with $N'\not\equiv N \mod 4$. In view of the first reduction, the result holds for some $N'' \equiv N \mod 4$ with $0 < N'' - N' < 4$. But now if the result does not hold for $N'$, then there is some element in the Atiyah--Hirzebruch spectral sequence for $\Sigma^{1,0} \underline{L}^\infty_{-N'-4+1}$ which detects a new multiple of $2+\rho \eta$, and this element would detect a new multiple of $2+\rho \eta$ in the Atiyah--Hirzebruch spectral sequence for $\Sigma^{1,0} \underline{L}^\infty_{-N''-4+1}$. But this contradicts the assumption that there are no new multiples of $2+\rho \eta$ detected in the latter, so there could not have been such a multiple in the former. 

We have therefore reduced to proving the proposition for a single $N \in \z$; we will prove the case $N = 2$.  Let $L_m \in \Motr$ denote the subcomplex of $\underline{L}^\infty_{-\infty}$ with cells in topological dimensions $-2  \leq d \leq 5$, and let $X = L_m \wedge D^\r L_m$ where $D^\r(-)$ is the $\r$-motivic Spanier--Whitehead dual functor $D^\r(-) = F(-,S^{0,0})$. We will show that in the Atiyah--Hirzebruch spectral sequence for $X$, no classes of the form $\alpha[-m,-n]$ with $\alpha \in \pi_{m,n}^\r(S^{0,0})$ base changing to zero survive.\footnote{We will use the notation $\alpha[m,n]$ to denote a class of the form $\alpha$ on a cell in bidegree $(m,n)$. If there is only one cell in that bidegree, we will refer to $\alpha[m,n]$ as \emph{the class}, but if there are multiple classes in that bidegree, we will say $\alpha[m,n]$ is \emph{a class of the form} $\alpha[m,n]$.} In other words, there are no ``purely $\r$-motivic" classes which contribute to $\pi_{0,0}^\r(X)$, or by adjunction, which contribute an additional degree of $2+\rho \eta$-divisibility to the identity map of $L_m$. The proposition for $N=2$ (and thus all $N \in \z$) then follows from comparing the Atiyah--Hirzebruch spectral sequence for $L_m$ with the Atiyah--Hirzebruch spectral sequence for $\underline{L}^\infty_{-N+4-1}$. 

We start by reducing the possible classses $\alpha[-m,-n]$ by restricting the bidegrees of $\alpha$. The $\r$-motivic homotopy group $\pi_{0,0}^\r(X)$ may be computed via the Atiyah--Hirzebruch spectral sequence arising from the filtration of $X$ by topological dimension. Observe that $X$ is a $64$-cell complex with cells concentrated in bidegrees of the form $(2k \pm \epsilon, k \pm \epsilon)$ where $-3 \leq k \leq 3$ and $\epsilon \in \{0,1\}$. The possible contributions to $\pi_{0,0}^\r(X)$ in the Atiyah--Hirzebruch spectral sequence therefore have the form $\alpha[2k \pm  \epsilon, k \pm \epsilon]$ with $\alpha \in \pi_{-2k \mp \epsilon, -k \mp \epsilon}^\r(S^{0,0})$. In particular, we note that $mw(\alpha) \leq 3$, so the relevant $\r$-motivic stable stems were calculated by Dugger and Isaksen in \cite{DI16a}. Since $mw(\alpha) \geq 0$ for all $\alpha \in \pi_{**}^\r(S^{0,0})$, we really only need to consider the cases $0 \leq k \leq 3$ in the sequel. 

We may also exclude any $\rho$-torsion free class $\alpha \in \pi_{m,n}^\r(S^{0,0})$. Indeed, we claim that $\alpha[-m,-n]$ must die in the Atiyah--Hirzebruch spectral sequence. To see this, observe that the $\r$-points of $X$ are the classical spectrum $X_{cl} = P_{-1}^2 \wedge DP_{-1}^2$ and the $\r$-points of $(2+\rho \eta)^4$ is $2^4$, which is zero on $X_{cl}$ by \cite{Tod63}. On the other hand, real Betti realization is given by $\rho$-localization \cite[Thm. 1.2]{BS20}, so any class detecting $(2+\rho \eta)^4$ must be $\rho$-torsion. 

In summary, we must consider the classes $\alpha[-m,-n]$ in the Atiyah--Hirzebruch spectral sequence for $X$ with the following properties:
\begin{itemize}
\item $\alpha$ base-changes to zero in $\pi_{**}^\c(S^{0,0})$. 
\item $\alpha \in \pi_{2k\pm \epsilon, k \pm \epsilon}^\r(S^{0,0})$ with $0 \leq k \leq 3$ and $\epsilon \in \{0,1\}$. 
\item $\alpha$ is $\rho$-torsion in $\pi_{**}^\r(S^{0,0})$. 
\end{itemize}
Examining \cite[Fig. 3]{DI16a}, we obtain Table \ref{Table:2} containing all possible classes $\alpha$ with these properties. 

\begin{table}
\begin{tabular}{| c| c | c| c|}
\hline 
$k$ & $G_{2k-1,k-1}$ & $G_{2k,k}$ & $G_{2k+1,k+1}$ \\
\hline
$0$ & & &  \\
\hline
$1$ & $\rho \tau h_1^2$ & & \\
\hline
$2$ & & & \\
\hline
$3$ & $\rho^{1+i} h_1^i \tau h_2^2$, $0 \leq i \leq 2$ & $\rho^{i} h_1^i \tau h_2^2$, $1 \leq i \leq 2$ & $\rho^{1+i}h_1^i c_0$, $0 \leq i \leq 2$ \\
\hline
\end{tabular}
\caption{Classes in $\pi_{**}^\r(S^{0,0})$ which could detect $(2+\rho\eta)^4 \in \pi_{**}^\r(X)$. Classes are described using their Adams names from \cite{BI20}.}\label{Table:2}
\end{table}

We now argue that $\alpha[-2k\mp \epsilon, -k\mp \epsilon]$ dies in the Atiyah--Hirzebruch spectral sequence for each $\alpha \in G_{2k \pm \epsilon, k \pm \epsilon}$. 

Any class of the form $\rho \tau h_1^2[-1,0]$ is killed by a $d_1$-differential of the form $d_1(\tau h_1[0,0]) = \rho \tau h_1^2[-1,0]$. Indeed, the $d_1$-differentials in the Atiyah--Hirzebruch spectral sequence have the form $d_1(x[m,n]) = h_0x[m-1,n]$ whenever the cell in bidegree $(m-1,n)$ supports a nontrivial action of $Sq^1$. Applying the Cartan formula and Lemma \ref{Lem:LCoh}, we see that any cell of $X$ in bidegree $(-1,0)$ supports a nontrivial action of $Sq^1$, and consulting \cite{DI16a}, we see that $h_0 \tau h_1 = \rho \tau h_1^2$, so the claimed differentials occur.

Except for $\rho \tau h_2^2$, all of the classes in the case $k=3$ are killed by $d_2$-differentials of the form $d_2(x[m,n]) = h_1x[m-2,n-1]$. Inspecting \cite{DI16a}, we see that each class $\alpha \in G_{2k\pm \epsilon, k \pm \epsilon}$ has the form $\alpha = h_1 \beta$, and $\alpha$ and $\beta$ do not support $h_0$-multiplication and are not divisible by $h_0$. Since $d_1$-differentials are given by $h_0$-multiplication, we see that $\alpha$ and $\beta$ both survive to the $E_2$-term. Now, $d_2$-differentials in the Atiyah--Hirzebruch spectral sequence have the form $d_2(x[m,n]) = h_1 x[m-2,n-1]$ whenever the cell in bidegree $(m-2,n-1)$ supports a nontrivial action of $Sq^2$, so we only need to check that each cell in a bidegree of the form $(-6\mp \epsilon, -3\mp \epsilon)$ supports a nontrivial action of $Sq^2$. This follows from the Cartan formula and Lemma \ref{Lem:LCoh}. Therefore the claimed differentials occur.

The last remaining element is $\rho \tau h_2^2[-5,-2]$. The only cell of $X$ in bidegree $(-5,-2)$ is $e_{-1,0} \otimes f_{-4,-2}$, so we only need to rule out one element. We claim that $\rho \tau h_2^2[-5,-2]$ is killed by a $d_6$-differential $d_6(\rho \tau h_1[1,1]) = \rho \tau h_2^2[-5,-2]$, or possibly a shorter differential. If we can show that $\rho \tau h_1[1,1]$ survives to $E_6$, then the claimed $d_6$-differential will follow from the Toda bracket $\rho \tau \nu^2 \in \langle \eta, \nu, \rho \tau \eta \rangle$ and the relation $Sq^2Sq^4(e_{-1,0} \otimes f_{-4,-2}) = e_{-1,0} \otimes f_{2,1} + e_{5,3} \otimes f_{-4,-2} + \tau(e_{2,1} \otimes f_{-1,-1}) =: c$ derived using the Cartan formula and Lemma \ref{Lem:LCoh}. 

We now check that $\rho \tau h_1[1,1]$ survives to $E_6$. Since $\rho \tau h_1$ is not $h_0$-divisible and does not support multiplication by $h_0$, it survives to $E_2$. It is not divisible by $h_1$ but it supports $h_1$-multiplication, so it could support a $d_2$-differential. However, the sum of cells $c$ is not in the image of $Sq^2$ since $e_{-1,0} \otimes f_{2,1}$ does not appear as a summand in the image of $Sq^1$ or $Sq^2$ applied to any lower cells. Therefore $\rho \tau h_1[1,1]$ survives to $E_3$. Now, there are no further multiplicative relations or Toda brackets involving $\rho \tau h_1$ which could give rise to a $d_r$-differential for $3 \leq r \leq 5$, so $\rho \tau h_1[1,1]$ survives to $E_6$. Therefore we have the claimed differential, or $\rho \tau h_2^2[-5,-2]$ was killed prior to the $E_6$-page by a shorter differential.

We have shown that there can be no class detecting $(2+\rho \eta)^4 \in \pi_{0,0}^\r(X)$, so it must be zero.
\end{proof}

\subsection{Generalized Mahowald invariants of $(2+\rho\eta)^i$}
We conclude this section with some of our main results.

\begin{thm}\label{Thm:rm2i}
Let $ i \geq 0$ and let $j \in \{2,3\}$. The $\r$-Mahowald invariant of $(2+\rho\eta)^{4i+j}$ is given by
\[
M^\r((2+\rho\eta)^{4i+j}) \ni \begin{cases}
v^{4i}_1 \tau \eta^2 \quad & j=2, \\
v^{4i}_1 \tau \eta^3 \quad & j=3.
\end{cases}
\]
\end{thm}

\begin{proof}
Proposition \ref{Prop:UB} gives an upper bound on $|M^\r((2+\rho\eta)^{4i+j})|$, and if this upper bound is tight, then the theorem holds. Low dimensional computations show that the upper bound is tight for $i=0$. The result now follows by induction using Proposition \ref{Prop:16Null}.\footnote{Compare with \cite{MR93} and \cite{Qui17}.}
\end{proof}

\begin{thm}\label{Thm:em2i}
Let $ i \geq 0$ and let $j \in \{2,3\}$. The $C_2$-equivariant Mahowald invariant of $(2+\rho\eta)^{4i+j}$ is given by
\[
M^{C_2}((2+\rho\eta)^{4i+j}) \ni \begin{cases}
v^{4i}_1 \tau \eta^2 \quad & j=2, \\
v^{4i}_1 \tau \eta^3 \quad & j=3.
\end{cases}
\]
\end{thm}

\begin{proof}
We have that $Re_{C_2}(v_1^{4i} \tau \eta^j ) = v^{4i}_1 \tau \eta^j$  by definition and $U(v^{4i}_1 \tau \eta^j) = v^{4i}_1 \eta^j \in \pi_*(S^{0,0})$.  The theorem then follows from applying Lemma \ref{compatibility} to both functors in the composite
$$\Motr \xrightarrow{Re_{C_2}} \Sptc \xrightarrow{U} \Spt$$
along with our computation of $M^\r((2+\rho\eta)^{4i+j})$ in the previous theorem and the computation of $M(2^{4i+j})$ in \cite[Thm. 2.17]{MR93}. 
\end{proof}

\section{Real motivic and $C_2$-equivariant Mahowald invariants of $\eta^i$}\label{SectionEta}

Following the same approach as we used to prove Theorems \ref{MT:Rv1} and \ref{MT:Ev1}, we now prove our second set of main results, Theorems \ref{MT:Rw1} and \ref{MT:Ew1}. In this section, we lift Andrews' $\c$-motivic $w_1$-periodic families \cite{And14} to $\pi_{**}^\r(S^{0,0})$ and compute the $\r$-motivic and $C_2$-equivariant Mahowald invariants of $\eta^i$. We lift Andrews' $w_1$-periodic families in $\pi_{**}^\c(S^{0,0})$ by studying the motivic Adams spectral sequence and the $\rho$-Bockstein spectral sequence. In particular, we prove a vanishing result for $\Ext^{***}_{A^\r}(\m_2^\r,\m_2^\r)$ which implies that Andrews' $w_1$-periodic families lift \emph{uniquely} to $\pi_{**}^\r(S^{0,0})$. We then show that these lifts are contained in the $\r$-motivic Mahowald invariants $M^\r(\eta^i)$ which proves Theorem \ref{MT:Rw1}. 

We then proceed to the $C_2$-equivariant setting, where we define $w_1$-periodic families via equivariant Betti realization. We then utilize our comparison results along with our previous $\r$-motivic calculations and the classical calculations of Mahowald and Ravenel \cite{MR93} to prove Theorem \ref{MT:Ew1}. 

\subsection{Upper bounds: $w_1$-periodic families in the $\r$-motivic stable stems}
Our first step is to show that the exotic periodic families in the $\c$-motivic stable stems constructed by Andrews in \cite{And14} lift to the $\r$-motivic setting. We begin with a quick review of these elements in the $\c$-motivic setting. Andrews defines explicit cocycles in the motivic Adams-Novikov spectral sequence, shows that these are permanent cycles for tridegree reasons, and then shows that they detect nontrivial composites of the form
$$S^{s,t} \xrightarrow{\alpha} C\eta \xrightarrow{w^{4i}_1} \Sigma^{-20i,-12i} C\eta \to S^{-20i+2,-12i+1}.$$
These composites are denoted $w^{4i}_1\alpha$, and if they are nontrivial for all $i \geq 1$ we say that $\alpha$ is \emph{$w^4_1$-periodic}. 

\begin{thm}\cite[Thm. 3.12]{And14}
The classes $\nu, \nu^2, \nu^3, \eta^2 \eta_4, \overline{\sigma}$, and $\overline{\sigma} \nu$ in the $\c$-motivic stable stems are $w^4_1$-periodic. 
\end{thm}

To produce $\r$-motivic analogs of these classes, we begin by studying $\Ext_{A^\r}^{***}(\m_2^\r,\m_2^\r)$. 

\begin{cor}
The following statements hold:
\begin{enumerate}
\item The elements $h_2,h_2^2,h_2^3,h_4h_1^3$ in $\Ext^{***}_{A^\r}(\m_2^\r,\m_2^\r)$ are permanent cycles in the $\rho$-Bockstein spectral sequence.  
\item If $h_1^4 x = 0$, define a periodicity operator $P_w -$ by setting $P_w x := \langle h_4, h_1^4, x \rangle$. Let $P^i_w x$ be the $i$-fold iterate of this periodicity operator. Then the sets $P^i_w h_2, P^i_w h_2^2, P^i_w h_2^3, P^i_w h_4 h_1^3$ each consist of precisely one nontrivial element, and all of these elements are permanent cycles in the $\rho$-Bockstein spectral sequence for all $i \geq 0$. 
\item All of the classes above are $\rho$-torsion free. 
\end{enumerate}
\end{cor}

\begin{proof}
All of the statement follow from applying Theorem \ref{Thm:Double} to the classical periodic families in $\Ext^{**}_{A^{cl}}(\f_2,\f_2)$ constructed by Adams in \cite{Ada66b}. 
\end{proof}

\begin{rem2}
In view of the previous corollary, we will use the notation $P^i_w x$ to denote the unique element in the set $P^i_w x$ (instead of the entire set). This convention will be applied throughout the remainder of the section. 
\end{rem2}

We can say even more about these classes. Recall that the Milnor--Witt stem of an element $\alpha \in \pi_{s,w}^k(S^{0,0})$ is $mw(x) = s-w$. 

\begin{prop}\label{Prop:Unique}
Fix $\alpha \in \{P^i_w h_2, P^i_w h_2^2, P^i_w h_2^3, P^i_w h_4 h_1^3 : i \geq 0\}$ and take $(s,f,w) \in \n \times \n \times \n$ such that $\alpha \in \Ext^{s,f,w}_{A^\r}(\m_2^\r,\m_2^\r)$. There are no classes in higher $\rho$-Bockstein filtration which can detect a class in $\Ext^{s,t,w}_{A^\r}(\m_2^\r,\m_2^\r)$. 
\end{prop}

\begin{exm}
The case $i=0$ follows for $h_2, h_2^2$, and $h_2^3$ by inspection of \cite[Fig. 3]{DI16a} and for $h_4 h_1^3$ by inspection of the machine calculations of Knight Fu and Glen Wilson \cite{WilCharts}. 
\end{exm}

\begin{proof}[Proof of Proposition \ref{Prop:Unique}]
Fix $\alpha \in \Ext^{s,f,w}_{A^\r}$ to be one of the classes above, and write $f = 4k+\ell$ with $k \geq 0$ and $1 \leq \ell \leq 4$. Explicitly, we have
\begin{equation}
\begin{split}
s = \begin{cases}
20k + 3 \quad & \text{ if } \ell=1, \\
20k+6 \quad & \text{ if } \ell=2, \\
20k+9 \quad & \text{ if } \ell=3, \\
20k+18 \quad & \text{ if } \ell=4,
\end{cases}
\end{split}
\quad \quad \text{and} \quad \quad
\begin{split}
w = \begin{cases} 
12k+2 \quad & \text{ if } \ell=1, \\
12k+4 \quad & \text{ if } \ell=2, \\
12k+6 \quad & \text{ if } \ell=3, \\
12k+ 11 \quad & \text{ if } \ell=4.
\end{cases}
\end{split}
\end{equation}
The possible contributions to $\Ext^{s,f,w}_{A^\r}$ with higher $\rho$-Bockstein filtration in the $\rho$-Bockstein spectral sequence thus have the form $\rho^m x$ where $x \in \Ext^{s+m, f, w+m}_{A^\c}$ with $m \geq 1$. We will show that any $x$ in these tridegrees must necessarily be zero. To do so, we use the $\c$-motivic May spectral sequence \cite{DI10,Isa14} converging to $\Ext^{***}_{A^\c}$ to show $\Ext^{s+m, f, w+m}_{A^\c} = 0$. 

Recall that the $E_1$-page of the motivic May spectral sequence is generated over $\f_2[\tau]$ by $\{h_{i,j} : i > 0, j \geq 0\}$ with $|h_{i,0} | = (2^i-2,1,2^{i-1}-1)$ and $|h_{i,j}| = (2^j(2^i-1)-1,1,2^{j-1}(2^i-1))$. Of course we can find polynomials in the May generators which have the correct stem, weight, or Adams filtration to detect a nontrivial $x$, but we claim that there are no polynomials which simultaneously have the correct stem, weight, \emph{and} Adams filtration. 

We introduce a function which relates stem and motivic weight. Let $r(i,j)$ denote the ratio of stem to Milnor-Witt degree for $h_{i,j}$, i.e. let $r(i,j) = s(h_{i,j})/(s(h_{i,j})-w(h_{i,j}))$. Extend $r$ to monomials by the rule $r(ab) = r(a) + r(b)$ and observe that the summands in a polynomial contributing to a fixed tridegree in the May spectral sequence all give the same value of $r$. With these conventions, we observe that any polynomial $P$ which detects $x$ must have $r(P)$ greater than approximately $5/2$. On the other hand, we have $r(i,j) > r(2,1)$ for all $i>0$, $j\geq 0$, except for $r(1,1) = \infty$ and $r(1,2) = 3$. Moreover, for all pairs $(i,j)$, we have $11/6 \leq r(i,j) \leq 7/3$ unless $(i,j) \in S := \{(1,1), (1,2), (2,1)\}$. Explicit calculations then imply that any polynomial detecting $x$ necesarily contains $h_{ij}$ for some $(i,j) \in \{(1,1), (1,2)\}$, since otherwise $r(P)$ would be too small. 

So far we have argued using the relationship between stem and motivic weight that any polynomial detecting $x$ must contain $h_1$ or $h_2$. We now use Adams filtration and properties of the $\c$-motivic May spectral sequence to show that any such polynomial detects zero. Let $P$ denote a polynomial which could detect $x$ and let $y$ denote any monomial obtained from $P$ by dividing any of its summands by $h_1$ and $h_2$ as much as possible. We will show that any such $y$ is zero, thereby proving $P$ is zero and thus $x$ is zero. Note $h_1h_2 = 0$ in the May spectral sequence starting at $E_2$, so we need only consider mutliplication by $h_1$ and $h_2$ separately. 


We first consider multiplication by $h_2$ and show $y$ must be zero by low-dimensional $\Ext$ computations. We have $h_2^4 = 0$ on the $E_2$-page of the motivic May spectral sequence, so we may multiply $y$ by at most $h_2^3$ to obtain (a summand of) $P$. The class $h_2^3$ has stem $9$, Milnor-Witt stem $3$, and $r(h_2^3) = 3$. Basic arithmetic then shows that if $y$ has stem greater than $9$, we cannot obtain a class in the same stem and Milnor-Witt stem as $x$ exclusively via $h_2$-multiplication. On the other hand, explicitly searching $\Ext^{**}_{A^\c}(\m_2^\c,\m_2^\c)$ (depicted in \cite{Isa14b}) in stems less than or equal to $9$ shows that there are no nontrivial classes $y$ which could be mutliplied by $h_2^\ell$, $\ell \leq 3$, to obtain a summand of $P$. Thus there are no classes $y$ which could give rise to $x$ through $h_2$-multiplication. 

We now analyze multiplication by $h_1$ and argue that $y$ must be zero using Adams filtration. If $y$ has stem $a$ and Milnor-Witt stem $b$, then $h_1^c y$ has stem $a+c$ and Milnor-Witt stem $b$. For $h_1^c y $ to have the same Milnor-Witt stem as $x$, we must have $b = w$. In the edge case, we then have $7a = 4b = 4w$, or $a = 4/7 w$. On the other hand, for $h_1^c y $ to have the same as $x$, we must have $c = s - a$. Substituting the values of $s$ and $w$ above, we have 
$$c = 20k - \epsilon - 4/7 w = 20k - \epsilon - 4/7(12k + \delta) = 20k - \epsilon - 48/7 k - 4/7 \delta \geq 13k - \epsilon - \delta \geq 11k$$
for all $k \geq 1$, where $(\epsilon, \delta) \in \{(3,2),(6,4),(9,6),(18,11)\}$. In particular, the Adams filtration of $h_1^c y$ is at least $11k > 4k + 4 \geq f$ for all $k \geq 1$. Therefore we cannot obtain $x$ from $y$ by $h_1$-multiplication and we have ruled out all possible ways of obtaining $x$ from $y$. We conclude that $x =0$. 
\end{proof}

\begin{cor}\label{Cor:Unique}
The following statements hold:
\begin{enumerate}
\item There is no indeterminacy in the names of the classes $P^i_w h_2$, $P^i_w h_2^2$, $P^i_w h_2^3$, and $P^i_w h_4 h_1^3$ in $\Ext^{***}_{A^\r}(\m_2^\r,\m_2^\r)$. 
\item Base-change $f^* : \Motr \to \Motc$ along $f : Spec(\c) \to Spec(\r)$ sends each class above to the class of the same name in $\Ext^{***}_{A^\c}(\m_2^\c, \m_2^\c)$. 
\item If a class above is a permanent cycle in the $\r$-motivic Adams spectral sequence, then it survives the $\r$-motivic Adams spectral sequence and detects a nontrivial class in $\pi_{**}^\r(S^{0,0})$. 
\end{enumerate}
\end{cor}

We now show that all of the classes in $\r$-motivic $\Ext$ listed above are permanent cycles in the $\r$-motivic Adams spectral sequence; we thank Mark Behrens for suggesting the following argument. Recall from \cite{And14} that the $\c$-motivic class $w^{4i}_1 \nu^j$, $i \geq 0$, $1 \leq j \leq 3$, is represented by the composite
$$S^{s,t} \xrightarrow{\widetilde{\nu^j}} C\eta \xrightarrow{w^{4i}_1} \Sigma^{-20i,-12i}C\eta \to S^{-20i+2,-12i+1}$$
where $\widetilde{\nu^j} \in \pi_{3j+2,2j+1}^\c(C\eta)$ is a lift of $\nu^j$ to the top cell of $C\eta$. Similarly, the $\c$-motivic class $w^{4i}_1 \eta^2 \eta_4$, $i \geq 1$, is represented by the composite
$$S^{0,0} \xrightarrow{\iota} C\eta \xrightarrow{w^{4i}_1} \Sigma^{-20i,-12i} C\eta \to S^{-20i+2,-12i+1}$$
where $\iota :S^{0,0} \to C\eta$ is the inclusion of the bottom cell. 

Consider the diagram
$$
\begin{tikzcd}
\pi^\r_{**}(C\eta)\arrow{r}  \arrow{d} & \pi_{**}^\c(C\eta) \arrow{d} \\
\pi^\r_{**}(S^{2,1}) \arrow{r} & \pi_{**}^\c(S^{2,1}). 
\end{tikzcd}
$$
where the horizontal arrows are base-change along $f : \r \to \c$ and the vertical arrows are projection onto the top cell.\footnote{We use $f$ instead of $i$ for the inclusion $\r \to \c$ in this section to avoid confusion with the inclusion $\iota : S^{0,0} \to C\eta$.} The classes $w^{4i}_1 \alpha \in \pi_{**}^\c(S^{2,1})$ with $\alpha \in \{\nu^j, \eta^2 \eta_4 : 1 \leq j \leq 3, i \geq 0\}$ lift to the classes $\widetilde{w^{4i}_1 \alpha} \in \pi_{**}^\c(C\eta)$ which are represented by the composites above where we do not project onto the top cell. 

We will show that our new $\r$-motivic classes $P^i_w x$ survive the $\r$-motivic Adams spectral sequence in two steps. First, we show that the corresponding $\r$-motivic classes in the $\r$-motivic Adams spectral sequence for $C\eta$ survive using base-change to $\c$. We then argue that our classes (obtained from the $C\eta$ classes by projection onto the top cell) survive by naturality.

\begin{lem}
The classes $\widetilde{w^{4i}_1 \alpha} \in \pi_{**}^\c(C\eta)$ lift uniquely to classes $\widetilde{w^{4i}_1 \alpha}' \in \pi_{**}^\r(C\eta)$. Moreover, the image of $\widetilde{w^{4i}_1 \alpha}'$ under the projection onto the top cell of $C\eta$ satisfies the following:
\begin{enumerate}
\item It is a nontrivial class in $\pi_{**}^\r(S^{0,0})$. 
\item It base-changes to $w^{4i}_1 \alpha$. 
\item It is detected in the $\r$-motivic Adams spectral sequence by $P^i_w x$ where $x = h_2^j$ if $\alpha = \nu^j$ and $x = h_1^3 h_4$ if $\alpha = \eta^2 \eta_4$. 
\end{enumerate} 
\end{lem}

\begin{proof}
We begin with an outline of the proof. First, we will define classes in the $\c$-motivic Adams spectral sequence which detect $\widetilde{w^{4i}_1 \alpha} \in \pi_{**}^\c(C\eta)$. Second, we will produce unique classes in the $\r$-motivic Adams spectral sequence for $C\eta$ which base-change to these. Third, we will show that our new $\r$-motivic classes are permanent cycles in the $\r$-motivic Adams spectral sequence for $C\eta$, thereby proving the first statement in the lemma using naturality of the motivic Adams spectral sequence. Last, we will project onto the top cell of $C\eta$ and use base-change to conclude (1)-(3). 

The groups $\pi_{**}^\c(C\eta)$ may be computed using the motivic Adams spectral sequence
$$\Ext^{***}_{A^\c}(\m_2^\c,H^{**}(C\eta)) \Rightarrow \pi_{**}^\c(C\eta).$$
The $\Ext$-group above may be computed by means of the algebraic Atiyah--Hirzebruch spectral sequence arising from the filtration of $H^{**}(C\eta)$ by topological dimension; a class coming from the top cell will be denoted $x[2,1]$ and a class coming from the bottom cell will be denoted $x[0,0]$, where $x \in \Ext^{***}_{A^\c}(\m_2^\c,\m_2^\c)$. Since $\widetilde{\nu^j}$ is a lift to the top cell of $C\eta$, we see that $\widetilde{w^{4i}_1 \nu^j}$ is represented by $P^i_w h_2^j[2,1]$, and similarly $\widetilde{w^{4i}_1 \eta^2 \eta_4}$ is represented by $P^i_w h_1^3 h_4[0,0]$.\footnote{These names follow from the computations of \cite[Sec. 5]{Qui17} where analogous names were given in the $C\tau$-linear motivic Adams spectral sequence, i.e. the motivic Adams spectral sequence based on $\overline{H} = H \wedge C\tau$. The names we give here are obtained by lifting along the map of spectral sequences induced by the map $H \to H \wedge C\tau$.}

The groups $\pi_{**}^\r(C\eta)$ may also be computed using the motivic Adams spectral sequence
$$\Ext^{***}_{A^\r}(\m_2^\r,H^{**}(C\eta)) \Rightarrow \pi_{**}^\r(C\eta).$$
A slight modification of the analysis in the proof of Proposition \ref{Prop:Unique} shows that for all $i \geq 0$, the classes $P^i_w h_1^3 h_4[0,0]$ and $P^i_w h_2^j[2,1]$, $1 \leq j \leq 3$, are nonzero in $\Ext^{***}_{A^\r}(\m_2^\r,H^{**}(C\eta))$ and that there are no classes in the same tridegrees in higher $\rho$-Bockstein filtration. In other words, these classes are unique lifts of the analogous $\c$-motivic $\Ext$ classes. 
 
We now show that our new $\r$-motivic classes are permanent cycles. First, observe that Adams differentials decrease stem by one, decrease Milnor-Witt stem by one, and increase Adams filtration. In the proof of Proposition \ref{Prop:Unique}, we saw that multiplication by $h_1$ was the only means of producing classes in the same stem and Milnor-Witt stem as $P^i_w x$, $x \in \{h_2, h_2^2, h_2^3, h_1^3 h_4\}$ and $i \geq 0$, if we start with a class detected in the $\c$-motivic May spectral sequence by a polynomial in $h_{ij}$'s with $r(i,j)<13/6$. The same argument implies that any possible target $y \in \Ext^{***}_{A^\r}(\m_2^\r,\m_2^\r)$ of an Adams differential $d_r(P^i_w x) = y$ would be $h_1$-divisible. However, since $h_1$ detects $\eta$, we see that $h_1 = 0$ in $\Ext^{***}_{A^\r}(\m_2^\r, H^{**}(C\eta)).$ Therefore there are no nonzero targets for Adams differentials and $P^i_w x$ is a permanent cycle. 

Since $P^i_w x \in \Ext^{***}_{A^\r}(\m_2^\r, H^{**}(C\eta))$ base-changes to $P^i_w x \in \Ext^{***}_{A^\c}(\m_2^\c,H^{**}(C\eta))$ which survives to detect $\widetilde{w^{4i}_1 \alpha} \in \pi_{**}^{\c}(C\eta)$, we see that $P^i_w x \in \Ext^{***}_{A^\r}(\m_2^\r, H^{**}(C\eta))$ survives to detect a nontrivial class $\widetilde{w^{4i}_1 \alpha}' \in \pi_{**}^\r(C\eta)$. This proves the first sentence of the lemma. 

Now, the image of $\widetilde{w^{4i}_1 \alpha}'$ under the upper-right composite is $w^{4i}_1 \alpha \in \pi_{**}^\c(S^{2,1})$. Since its image is nonzero, we see that $\widetilde{w_1^{4i} \alpha'}$ is nonzero, proving (1). Since the classes detecting these homotopy classes in the $\r$-motivic Adams spectral sequence base-change to the analogous classes in the $\c$-motivic Adams spectral sequence, the homotopy classes base-change to each other, proving (2). The image of $\widetilde{w^{4i}_1 \alpha}'$ under the projection onto the top cell of $C\eta$ is a nontrivial class which we call $w^{4i}_1 \alpha'$, and the image of $w^{4i}_1 \alpha'$ under base-change is $w^{4i}_1 \alpha$. By Corollary \ref{Cor:Unique}, $P^i_w x$ is the unique class in the $\r$-motivic Adams spectral sequence which base-changes to the class $P^i_w x$ in the $\c$-motivic Adams spectral sequence. Therefore $P^i_w x$ detects $w^{4i}_1 \alpha'$, which proves (3). 
\end{proof}

\begin{defin}
We define the following classes in $\pi_{**}^\r(S^{0,0})$:
\begin{enumerate}
\item For all $i \geq 0$ and $1 \leq j \leq 3$, define $w^{4i}_1 \nu^j$ to be the class detected by $P^i_w h_2^j$ in the $\r$-motivic Adams spectral sequence.
\item For all $i \geq 0$, define $w^{4i}_1 \eta^2 \eta_4$ to be the class detected by $P^i_w h_4 h_1^3$ in the $\r$-motivic Adams spectral sequence. 
\end{enumerate}
\end{defin}

We have shown the following:

\begin{thm}
For all $i \geq 0$, the classes $w^{4i}_1 \nu^j$, $1 \leq j \leq 3$, and $w^{4i}_1 \eta^2 \eta_4$ are nonzero in $\pi_{**}^\r(S^{0,0})$. Moreover, their images under the base-change functor $f^* : \Motr \to \Motc$ are Andrews' $w_1$-periodic elements with the same names. 
\end{thm}

\subsection{$w_1$-periodic families in the $C_2$-equivariant stable stems}
We now construct the $C_2$-equivariant analogs of Andrews' $w_1$-periodic families. We begin with the following lemma. 

\begin{lem}
The equivariant Betti realization of $\eta^i \in \pi^\r_{i,i}(S^{0,0})$ is $\eta^i \in \pi^{C_2}_{i,i}(S^{0,0})$. 
\end{lem}

\begin{proof}
Let $C^\r$ denote the cobar complex whose homology is the $E_2$-page of the $\r$-motivic Adams spectral sequence and let $C^{C_2}$ denote the cobar complex whose homology is the $E_2$-page of the $C_2$-equivariant Adams spectral sequence. Equivariant Betti realization induces a map between these cobar complexes which sends $[\xi_1 | \cdots | \xi_1] \in C^\r$ to $[\xi_1 | \cdots | \xi_1] \in C^{C_2}$ by \cite[Section 3]{DI16}. Since these classes detect $\eta^i$ in the $\r$-motivic and $C_2$-equivariant settings, the result follows.
\end{proof}

\begin{lem}\footnote{We first learned of this lemma from Mike Hill.}
The image of $\eta^i \in \pi^{C_2}_{i,i}(S^{0,0})$ under the geometric fixed points functor $\Phi^{C_2}$ is $2^i \in \pi_0(S^0)$ for all $i \geq 1$. 
\end{lem}

\begin{proof}
Recall that an explicit model of $\eta \in \pi^{C_2}_{1,1}(S^{0,0})$ is given by
$$\eta : S^{3,2} \simeq S(\c^2) \to \c P^1 \simeq S^{2,1}.$$
Since $\Phi^{C_2} (S^{s,t}) \simeq S^{s-t}$, we see that $\Phi^{C_2}(\eta) \in \pi_0(S^0).$ Moreover, since the geometric fixed points of a map between suspension spectra may be computed by taking the fixed points of the corresponding $C_2$-spaces, we can see from the explicit model that $\Phi^{C_2}(\eta) = 2$. The general result follows. 
\end{proof}

\begin{cor}
The class $\eta \in \pi_{1,1}^{C_2}(S^{0,0})$ is not nilpotent.
\end{cor}

\begin{proof}
Suppose not. Then there exists some minimal $j>0$ such that $\eta^j = 0$. Then we would have $2^j = \Phi^{C_2}(\eta^j) = \Phi^{C_2}(0) = 0$, which is a contradiction.
\end{proof}

In fact, the previous lemma is an example of a more general relationship between the $C_2$-equivariant and classical stable stems. 

\begin{lem}\label{geomASS}
The map of Adams spectral sequences
$$\Ext_{A^{C_2}}(\m^{C_2}_2,\m^{C_2}_2) \to \Ext_{A}(\f_2,\f_2)$$
induced by the geometric fixed points functor 
$$\Phi^{C_2}: \Sptc \to \Spt$$
satisfies
$$h_{i,j} \mapsto h_{i,j-1}.$$
\end{lem}

\begin{proof}\footnote{We learned this argument from Mark Behrens, who attributed it to Dan Isaksen.}
The geometric fixed points functor $\Phi^{C_2} : \Sptc \to \Spt$ is the composite 
$$\Sptc \xrightarrow{(-)^{\Phi}} \Sptc \xrightarrow{(-)^{C_2}} \Spt$$
of the geometric localization functor $(-)^\Phi : \Sptc \to \Sptc$ defined by $X \mapsto X \wedge \widetilde{EC_2}$ and the categorical $C_2$-fixed points functor $(-)^{C_2} : \Sptc \to \Spt$. Recall the following facts from \cite[Section 2]{BW17}:
\begin{itemize}
\item There is an isomorphism $\pi^{C_2}_{**}(X^\Phi) \cong \pi^{C_2}_{**}(X)[\rho^{-1}]$.
\item There is an isomorphism $\pi^{C_2}_i X \cong \pi_i(X^{C_2}).$
\end{itemize}
Putting these together shows that there is an isomorphism
$$\pi_*(\Phi^{C_2}(X)) \cong \left( \pi_{**}(X)[\rho^{-1}] \right)|_{{**} \in \z}.$$
Explicitly, if $\alpha \in \pi^{C_2}_{s,w}(X)$, then $\Phi^{C_2}(\alpha)$ is computed by mapping $\alpha$ to the $\rho$-localization of $\pi^{C_2}_{s,w}(X)$ and then multiplying the image by $\rho^w$ to obtain a class in $\pi_{s-w}(X)[\rho^{-1}]$. 

Our goal is to understand this composite at the level of Adams spectral sequence $E_2$-pages. In this case, we must understand the composite 
$$\Ext^{s,*,*}_{(A^{C_2})^\vee}(\m^{C_2}_2,\m^{C_2}_2) \xrightarrow{(-)^{\Phi}} \Ext^{s,*,*}_{(A^{C_2})^\vee}(\m^{C_2}_2,\m^{C_2}_2)[\rho^{-1}] \xrightarrow{\cdot \rho^?} \Ext^{s,*}_{(A^{C_2})^\vee}(\m^{C_2}_2,\m^{C_2}_2)[\rho^{-1}].$$
The right-hand side is isomorphic to $\Ext^{s,*}_{A^{cl}_*}(\f_2,\f_2)$ by the above remarks. The proof of \cite[Thm. 4.1]{DI16a} can easily be adapted to show that that the middle term can be rewritten as
$$\Ext^{s,{*,*}}_{(A^{C_2})^\vee}(\m^{C_2}_2,\m^{C_2}_2)[\rho^{-1}] \cong \Ext^{s,*}_{DA^{cl}_*}(\f_2,\f_2) \otimes \f_2[\rho,\rho^{-1}]$$
where $DA^{cl}_*$ is the image of the classical dual Steenrod algebra under the dual of the doubling homomorphism defined by sending $Sq^i \mapsto Sq^{2i}$. Moreover, the proof of \cite[Thm. 4.1]{DI16a} shows that the class $h_{ij} \in \Ext^{s,{*,*}}_{(A^{C_2})^\vee}$ corresponds to the class $h_{i,j} \in \Ext^{s,*}_{DA^{cl}_*}$. Since the doubling homomorphism induces a map which sends $h_{ij} \mapsto h_{i,j+1}$ on $\Ext$-groups, the lemma follows. 
\end{proof}

\begin{rem2}
The above lemma also follows from the construction of the classical and equivariant Steenrod algebras. We saw in the proof of Lemma \ref{compatibility} that the geometric fixed points of $B_{C_2}\mu_2$ are $B \mu_2$, and that the image of the nonequivariant $2n$-skeleton of $B_{C_2}\mu_2$ is the nonequivariant $n$-skeleton of $B \mu_2$. In particular, the cell carrying the $C_2$-equivariant squaring operation $Sq^{2i}$ becomes the cell carrying the nonequivariant squaring operation $Sq^i$. Dualizing, we see that under the geometric fixed points functor we have $\xi^{2^i}_j \mapsto \xi^{2^{i-1}}_j$. The result follows by definition of the classes $h_{i,j}$. 
\end{rem2}

\begin{defin}
We define the following classes in $\pi_{**}^{C_2}(S^{0,0})$:
\begin{enumerate}
\item For all $i \geq 0$ and $1 \leq j \leq 3$, define $w_1^{4i}\nu^j$ to be the image of $w_1^{4i}\nu^j \in \pi_{**}^\r(S^{0,0})$ under equivariant Betti realization. 
\item For all $i \geq 0$, define $w_1^{4i} \eta^2 \eta_4$ to be the image of $w_1^{4i}\eta^2\eta_4 \in \pi_{**}^\r(S^{0,0})$ under equivariant Betti realization. 
\end{enumerate}
\end{defin}

The geometric fixed points functor relates these $C_2$-equivariant families to Adams' $v_1$-periodic families \cite{Ada66}. 

\begin{prop}\label{Prop:wv}
For all $i \geq 0$, the classes $w_1^{4i}\nu^j$, $1 \leq j \leq 3$, and $w_1^{4i} \eta^2 \eta_4$ in the $C_2$-equivariant stable stems are nonzero. 
\end{prop}

\begin{proof}
We provide the proof for $w_1^{4i}\nu$. The remaining classes are similar. 

Recall that $\eta \in \pi_1(S^0)$ is detected by $h_{11}$ in the classical Adams spectral sequence. By \cite{Ada66}, the classes $v^{4i}_1 \eta \in \pi_{8i+1}(S^0)$ are nontrivial for all $i >0$. By \cite{Ada66b}, the class $v^{4i}_1 \eta$ is detected by iterating the Massey product $P(-) = \langle h_3, h^4_0, - \rangle$ $i$-times. 

Since equivariant Betti realization induces a map between Adams spectral sequences, the class $w_1^{4i}\nu \in \pi_{20i+3,12i+2}^{C_2}(S^{0,0})$ is detected by $P_w^ih_{12}$ in the $C_2$-equivariant Adams spectral sequence. We thus have $\Phi^{C_2}(w^{4i}_1 \nu) = v^{4i}_1 \eta$ by Lemma \ref{geomASS}. Since $v^{4i}_1 \eta \neq 0$ for all $i \geq 0$, we conclude that $w^{4i}_1 \nu \neq 0$ for all $i \geq 0$. 
\end{proof}

We can now state our upper bound result for the $\r$-motivic and $C_2$-equivariant Mahowald invariants of $\eta^i$.

\begin{prop}\label{Prop:etaUB}
For $i \geq 0$ and $0 \leq j \leq 3$ with $(i,j) \neq (0,0)$, the following statements hold:
\begin{enumerate}
\item If $j \neq 0$, we have $|M^\r(\eta^{4i+j})| \leq |w_1^{4i} \nu^j|$, and if $j=0$, we have $|M^\r(\eta^{4i})| \leq |w_1^{4i-4}\eta^2\eta_4|$. Further, if equality holds then $w_1^{4i} \nu^j \in M^\r(\eta^{4i+j})$ and $w_1^{4i-4} \eta^2 \eta_4 \in M^\r(\eta^{4i})$.
\item If $j \neq 0$, we have $|M^{C_2}(\eta^{4i+j})| \leq |w_1^{4i} \nu^j|$, and if $j=0$, we have $|M^{C_2}(\eta^{4i})| \leq |w_1^{4i-4}\eta^2\eta_4|$. Further, if equality holds then $w_1^{4i} \nu^j \in M^{C_2}(\eta^{4i+j})$ and $w_1^{4i-4} \eta^2 \eta_4 \in M^{C_2}(\eta^{4i})$.
\end{enumerate}
\end{prop}

\begin{proof}
The proof is similar to the proof of Proposition \ref{Prop:UB}. The first item follows from base-change along $i : \r \to \c$, the analogous $\c$-motivic results from \cite{Qui17}, and Lemma \ref{Lem:BaseChange}. The second item follows from the analogous classical results from \cite{MR93}, the proof of Proposition \ref{Prop:wv}, and Lemma \ref{squeeze}.
\end{proof}

\subsection{Lower bounds: $\r$-motivic homotopy of stunted motivic lens spectra, II}

We now show the upper bounds of Proposition \ref{Prop:etaUB} are tight by proving an $\r$-motivic analog of \cite[Prop. 5.16]{Qui17}. This implies that if the composite
\begin{equation}
S^{i,i} \xrightarrow{\eta^i} S^{0,0} \to \Sigma^{1,0} \underline{L}^\infty_{-N+1}
\end{equation}
is null for some $N \in \z$, then the composite
\begin{equation}
S^{i+4,i+4} \xrightarrow{\eta^{i+4}} S^{0,0} \to \Sigma^{1,0} \underline{L}^\infty_{-N-10+1}
\end{equation}
is also null. This in turn implies that $|M^\r(\eta^{i+4})| \geq |M^\r(\eta^i)| + 20$ for all $i$.

\begin{prop}\label{Prop:eta4}
The map $\eta^4$ is null on $\underline{L}^{9+2m}_{-1+2m}$ for all $m \in \z$. In particular, the topological dimension of $M^\r(\eta^{i+4})$ is at least $20$ more than the topological dimension of $M^\r(\eta^i)$. 
\end{prop}

\begin{rem2}
This result appeared as a conjecture in a previous version of this work because the requisite range of $\pi_{**}^\r(S^{0,0})$ was unknown at the time. Recent results on $\pi_{**}^\r(S^{0,0})$ from \cite{BI20} now cover this range. 
\end{rem2}

\begin{proof}
Our proof only depends on the $A(1)$-module structure of $H^{**}(\underline{L}^{9+2m}_{-1+2m})$. Up to suspension, these are all isomorphic, so it suffices to prove the case $m=0$. Let $X$ be the $\r$-motivic spectrum defined by $X = \underline{L}^{9}_{-1} \wedge D^\r \underline{L}^9_{-1}$. We will show that $\eta^4 = 0 \in \pi^{\r}_{4,4}(X)$. 

The group $\pi_{4,4}^\r(X)$ may be computed via the Atiyah--Hirzebruch spectral sequence arising from the filtration of $X$ by topological dimension. Note that $X$ is a $400$-cell complex with cells concentrated in bidegrees of the form $(2k \pm \epsilon, k)$ with $-9 \leq k \leq 9$ and $\epsilon \in \{0,1\}$. The possible contributions to $\pi_{4,4}^\r(X)$ in the Atiyah--Hirzebruch spectral sequence have the form $\alpha[2k\pm \epsilon,k]$ with $\alpha \in \pi_{4-2k\mp \epsilon, 4-k}(S^{0,0})$. Note that if $k \geq 2$, then the Milnor--Witt stem of $\alpha$, $4-2k \mp \epsilon-(4-k)$, is negative. Since $\pi_{**}^\r(S^{0,0})$ is concentrated in nonnegative Milnor--Witt stems, we can exclude the cases $k \geq 2$. Let $G_{4-2k\mp \epsilon, 4-k} \subseteq \pi_{4-2k\mp \epsilon, 4-k}^\r(S^{0,0})$ denote the subgroup of elements. The  classes in $\pi_{**}^\r(S^{0,0})$ which could detect $\eta^4 \in \pi_{4,4}^\r(X)$ are enumerated in Table \ref{Table:eta}; the table is compiled by examining the charts of \cite{BI20}.

\begin{table}\label{Table:eta}
\begin{tabular}{| c| c | c| c|}
\hline 
$k$ & $G_{4-2k-1;4-k}$ & $G_{4-2k;4-k}$ & $G_{4-2k+1;4-k}$ \\
\hline
$-9$ & & $h_2c_1$; $c_0d_0$; $\rho h_2 g$ & $h_4c_0$; $h_1^2 \tau h_2^2 h_4$; $h_1^3 \rho \tau h_2^2 h_4$; $h_1 \tau c_0 d_0$; $\rho^3 h_2^2 g$ \\
\hline
$-8$ & $\rho^{3+i}h_1^{5+i}e_0$, $0 \leq i \leq 3$ & $h_0g$; $h_0^2g$ & $h_2^2 h_4$; $\rho h_2 c_1$; $\rho^2 h_2 g$; $h_2 f_0$; $\rho c_0 d_0$ \\
\hline
$-7$ & & $h_1^2 h_4$; $\rho^{3+i}h_1^{4+i}e_0$, $0 \leq i \leq 3$ & $c_1$ \\
\hline
$-6$ & & & $h_1^2 h_4$; $\rho h_1^3 h_4$ \\
\hline
$-5$ & & & $d_0 h_1$ \\
\hline
$-4$ & $c_0 h_1^3$; $\rho c_0 h_1^4$; $\rho^2 c_0 h_1^5$ & & \\
\hline
$-3$ & & $c_0 h_1^2$; $\rho c_0 h_1^3$; $\rho^2 c_0 h_1^4$ & \\
\hline
$-2$ & & & $h_3 h_1^2$; $\rho^i c_0 h_1^{1+i}$, $0 \leq i \leq 2$ \\
\hline
$-1$ & $\rho^i h_1^{5+i}$, $i \geq 0$ & & \\
\hline
$0$ & & $\rho^i h_1^{4+i}$, $i \geq 0$ & \\
\hline
$1$ & & &  $\rho^i h_1^{3+i}$, $i \geq 0$ \\
\hline
\end{tabular}
\caption{Classes in $\pi_{**}^\r(S^{0,0})$ which could detect $\eta^4 \in \pi_{**}^\r(X)$. Classes are described using their Adams names from \cite{BI20}.}
\end{table}

We will show that each class in Table \ref{Table:eta} dies in the Atiyah--Hirzebruch spectral sequence, and therefore $\eta^4$ must be nonzero. We begin by eliminating many classes using the following two properties of the Atiyah--Hirzebruch spectral sequence for $X$:

\begin{itemize}

\item ``$h_1$-multiplication": Suppose $\alpha \in G_{m,n}$ satisfies $\alpha[4-m,4-n] \neq 0$ on the $E_2$-page. If $\alpha = h_1 \beta$ where $h_0 \beta=0$ and $\beta \neq h_0 \gamma$ for any $\gamma$, then $d_2(\beta[4-m+2,4-n+1]) = \alpha[4-m,4-n]$. Similarly, if $h_1 \alpha = \beta$ and $\beta \neq 0$ on $E_2$, then $d_2(\alpha[4-m,4-n]) = \beta[4-m-2,4-n-1]$. This follows from the argument given for the analogous statement in the proof of \cite[Prop. 5.16]{Qui17}. 

\item ``$\rho$-torsion free": Suppose $\alpha \in G_{m,n}$ is $\rho$-torsion free. Then $\alpha[4-m,4-n]$ must die in the Atiyah--Hirzebruch spectral sequence. Indeed, observe that the $\r$-points of $X$ are the classical spectrum $X_{cl} = P^6_{-1} \wedge DP^6_{-1}$ and the $\r$-points of $\eta^4$ are $2^4$ which is zero on $X_{cl}$ by \cite{Tod63}. Real Betti realization is given by $\rho$-localization \cite[Thm. 1.2]{BS20}, so any class detecting $\eta^4$ must be $\rho$-torsion. 

\end{itemize}

These two properties may be applied to rule out the following classes:

\begin{enumerate}

\item[$k=-9$:] The class $\rho h_2 g \in G_{22,13}$ dies because it is $\rho$-torsion free. The classes $c_0d_0 \in G_{22,13}$ and $h_4c_0$, $h_1^2 \tau h_2^2 h_4$, $h_1^3 \rho \tau h_2^2 h_4$, and $h_1 \tau c_0 d_0 \in G_{23,13}$ are killed using $h_1$-multiplication. 

\item[$k=-8$:] The classes $h_2f_0$ and $\rho^{3+i}h_1^{5+i}e_0 \in G_{19,12}$, $0 \leq i \leq 3$, are killed using $h_1$-multiplication. The classes $h_2^2h_4$, $\rho h_2c_1$, and $\rho^2 h_2g \in G_{21,12}$ die because they are $\rho$-torsion free. 

\item[$k=-7$:] The classes $h_1^2 h_4$ and $\rho^{3+i}h_1^{1+i}e_0 \in G_{18,11}$, $0 \leq i \leq 3$, are killed using $h_1$-multiplication. The class $c_1 \in G_{19,11}$ dies because it is $\rho$-torsion free.

\item[$k=-6$:] The classes $h_1^2h_4$ and $\rho h_1^3 h_4 \in G_{17,10}$ are killed using $h_1$ multiplication.

\item[$k=-5$:] The class $d_0h_1 \in G_{15,9}$ is killed using $h_1$-multiplication.

\item[$k=-4$:] The classes $\rho^i c_0 h_1^{3+i} \in G_{11,8}$, $0 \leq i \leq 2$, are killed using $h_1$-multiplication. 

\item[$k=-3$:] The classes $\rho^i c_0 h_1^{2+i} \in G_{10,7}$, $0 \leq i \leq 2$, are killed using $h_1$-multiplication.

\item[$k=-2$:] The classes $h_3h_1^2$ and $\rho^i c_0 h_1^{1+i} \in G_{9,6}$, $0 \leq i \leq 2$, are killed using $h_1$-multiplication.

\item[$k=-1$:] The classes $\rho^ih_1^{5+i} \in G_{5,5}$, $i \geq 0$, are killed using $h_1$-multiplication.

\item[$k=0$:] The classes $\rho^ih_1^{4+i} \in G_{4,4}$, $i \geq 0$, are killed using $h_1$-multiplication.

\item[$k=1$:] The classes $\rho^ih_1^{3+i} \in G_{3,3}$, $i \geq 0$, are killed using $h_1$-multiplication. 

\end{enumerate}

The remaining classes are ruled out below. 

\begin{enumerate}

\item[$h_2c_1$:] The class $h_2c_1[-18,-9]$ is the target of a $d_4$-differential on $c_1[-14,-7]$. 

We first show that $c_1[-14,-7]$ survives to $E_4$. Observe that $c_1[-14,-7]$ survives to $E_2$ since $c_1 \in \pi_{19,11}^\r(S^{0,0})$ neither supports nontrivial $h_0$-multiplication nor is $h_0$-divisible. Similarly, $c_1$ neither supports nor is divisible by $h_1$, so $c_1[-14,-7]$ survives to $E_3$. Finally, $c_1$ is not involved in any Massey products or Toda brackets of the form $\langle h_0, h_1, - \rangle$, so it survives to $E_4$. 

Similar arguments show that $h_2c_1[-18,-9]$ survives to $E_4$. 

We now prove the claimed differential. A $d_4$-differential corresponds to a multiplication by $h_2$-attaching map, which is detected by a nontrivial action of $Sq^4$ on $H^{**}(X)$. Let $\{x_{-2},\ldots,x_{17}\}$ denote a basis for $H^{**}(\underline{L}_{-1}^{9})$ (where $|x_i| = i$) and let $\{y_{-17},\ldots,y_{2}\}$ denote a basis for $H^{**}(D^\r \underline{L}_{-1}^{9})$ (where $|y_i|=i$), so that $\{x_i \otimes y_{-j}\}_{-2 \leq i,j \leq 17}$ forms a basis for $H^{**}(X)$. The only cohomology basis element in bidegree $(-18,-9)$ is $x_{-2} \otimes y_{-16}$. By Lemma \ref{Lem:LCoh}, $Sq^2$ acts nontrivially on $x_{-2}$ and $y_{-16}$, so by the Cartan formula, $Sq^4$ acts nontrivially on $x_{-2} \otimes y_{-16}$. Thus we have the claimed differential.

\item[$h_0^2g$:] Any class of the form $h_0^2g[-16,-8]$ is the target of a $d_1$-differential.

Indeed, a $d_1$-differential corresponds to a multiplication by $h_0$-attaching map, which is detected by a nontrivial action of $Sq^1$ on $H^{**}(X)$. There are four cohomology basis elements in bidegree $(-16,-8)$, namely $x_{-2+i} \otimes y_{-14-i}$, $0 \leq i \leq 3$. Using Lemma \ref{Lem:LCoh}, we see that if $-2+i$ is even, then $Sq^1(x_{-2+i}) = 0$ but $Sq^1(y_{-14-i}) = y_{-14-i+1}$, and if $-2+i$ is odd, then $Sq^1(x_{-2+i}) = x_{-2+i+1}$ but $Sq^1(y_{-14-i})=0$. Therefore $Sq^1$ acts nontrivially on $x_{-2+i} \otimes y_{-14-i}$, $0 \leq i \leq 3$, and we have $d_1$-differentials $d_1(h_0g[-15,-8]) = h_0^2g[-16,-8]$. 

\item[$h_0g$:] Any class of the form $h_0g[-16,-8]$ is the target of a $d_6$-differential on $d_0h_1[-10,-5]$, or it dies on an earlier page of the Atiyah--Hirzebruch spectral sequence. 

The claimed $d_6$-differential follows from the Toda bracket $h_0g \in \langle h_2, h_1, d_0h_1 \rangle$ and a nontrivial action of $Sq^2Sq^4$ on any cell in bidegree $(-16,-8)$. The Toda bracket follows from its $\c$-motivic analog \cite[Table 6]{Isa14} and base-change. The nontrivial action of $Sq^6$ follows from a straightforward but tedious application of the Cartan formula and Lemma \ref{Lem:LCoh} which we omit. We are thus reduced to showing that $d_0h_1[-10,-5]$ survives to $E_6$.

To see $d_0h_1[-10,-5]$ survives to $E_6$, we argue as follows. 

First, $d_0h_1 \in \pi_{15,9}^\r(S^{0,0})$ is not involved in any $h_0$-multiplication, so it survives to $E_2$. 

Second, $d_0h_1$ is divisible by $h_1$ and so $d_0h_1[-10,-5]$ could be the target of a $d_2$-differential of the form $d_2(d_0[-8,-4])$. However, any class of the form $d_0[-8,-4]$ supports a nontrivial $d_1$-differential: $h_0d_0 \neq 0$ and $Sq^1$ acts nontrivially on any class in $H^{-9,-4}(X)$ by Lemma \ref{Lem:LCoh} and the Cartan formula. Therefore $d_0[-8,-4] = 0$ in $E_2$, so $d_0h_1[-10,-5]$ survives to $E_3$. 

Third, since $d_0h_1$ is not involved in any Toda brackets of the form $\langle h_1, h_0, - \rangle$ or $\langle h_0, h_1, - \rangle$, it is not involved in any nontrivial $d_3$-differentials and survives to $E_4$. 

Fourth, since $d_0h_1$ is not involved in any $h_2$-multiplication, it is not involved in any nontrivial $d_4$-differentials and survives to $E_5$. 

Finally, there are no Toda brackets involving $d_0h_1$ nor multiplicative relations which could give rise to a nontrivial $d_5$-differential by inspection of \cite[Table 6]{Isa14} and \cite{BI20}. Therefore $d_0h_1[-10,-5]$ survives to $E_6$ and the claimed differentials occur (or $h_0g[-16,-8]$ is killed in an earlier page). 

\item[$\rho c_0d_0$:] Ths class $\rho c_0 d_0[-17,-8]$ is the target of a $d_1$-differential on $h_2f_0[-16,-8]$.

We check that $Sq^1$ acts nontrivially on the cohomology basis element in bidegree $(-17,-8)$. The only cohomology basis element in bidegree $(-17,-8)$ is $x_{-1} \otimes y_{-16}$, and by Lemma \ref{Lem:LCoh}, $Sq^1$ acts nontrivially on $x_{-1}$. Therefore $Sq^1$ acts nontrivially on $x_{-1} \otimes y_{-16}$ by the Cartan formula and we have the claimed differential. 

\item[$c_1$:] The class $c_1[-15,-7]$ supports a nontrivial $d_4$-differential. The differential can be proven using the same arguments above which showed that $h_2c_1[-18,-9]$ is killed by $d_4(c_1[-14,-7])$. 

\end{enumerate}

\end{proof}

\subsection{Real motivic and $C_2$-equivariant Mahowald invariants of $\eta^i$}

We can now compute $M^\r(\eta^i)$ and $M^{C_2}(\eta^i)$.

\begin{thm}\label{rmieta}
For any $i \geq 0$ and $0 \leq j \leq 3$ with $(i,j) \neq (0,0)$, we have
$$M^\r(\eta^{4i+j}) \ni \begin{cases}
w_1^{4(i-1)} \eta^2 \eta_4 \quad&j = 0, \\
w_1^{4i} \nu \quad & j=1, \\
w_1^{4i}\nu^2 \quad & j=2, \\
w_1^{4i}\nu^3 \quad & j=3.
\end{cases}$$
\end{thm}

\begin{proof}
The proof is analogous to that of Theorem \ref{Thm:rm2i}. Proposition \ref{Prop:etaUB} provides an upper bound and implies that if the upper bound is tight, then the theorem holds. Low dimensional computations show the lower bound is tight for $i = 0$, and induction using Proposition \ref{Prop:eta4} shows the lower bound is tight for all $i \geq 0$. 
\end{proof}

We now use Lemma \ref{compatibility} to calculate $M^{C_2}(\eta^i)$ by comparison with $M^\r(\eta^i)$ and $M^{cl}(2^i)$. 

\begin{thm}\label{cmieta}
For any $i \geq 0$ and $0 \leq j \leq 3$ with $(i,j) \neq (0,0)$, we have
$$M^{C_2}(\eta^{4i+j}) \ni \begin{cases}
w_1^{4(i-1)} \eta^2 \eta_4 \quad&j = 0, \\
w_1^{4i} \nu \quad & j=1, \\
w_1^{4i}\nu^2 \quad & j=2, \\
w_1^{4i}\nu^3 \quad & j=3.
\end{cases}$$
\end{thm}

\begin{proof}
Proposition \ref{Prop:etaUB} gives an upper bound on $|M^{C_2}(\eta^i)|$ and implies that if the upper bound is tight, then the theorem holds. On the other hand, Theorem \ref{rmieta} and Lemma \ref{compatibility} imply that the lower bound is tight. The result follows. 
\end{proof}

\appendix

\section{The $C_2$-equivariant analog of Lin's Theorem}\label{App:Lin}

The purpose of this appendix is to prove the $C_2$-equivariant analog of Lin's Theorem discussed in Section \ref{SS:Qui}. The first three sections of the appendix are essentially an exercise in rewriting Gregersen's thesis \cite{Gre12} in $C_2$-equivariant language, with three notable exceptions:
\begin{enumerate}
\item Gregersen shows that the subalgebras $A(n)$ of the motivic Steenrod algebra are free over $\m_2^k$ using some technical diagram chasing. We avoid this in the $C_2$-equivariant setting using Ricka's theory of profile functions \cite{Ric14}. 
\item The $C_2$-equivariant homology of a point $\m_2^{C_2}$ is larger than $\m_2^\r$, so we sometimes must check that $\m_2^{C_2}$ vanishes in the necessary bidegrees.
\item There is some ambiguity in choosing an analog of projective spectra in the $C_2$-equivariant setting: we must work with $C_2$-equivariant vector bundles and $C_2$-equivariant classifying space, instead of classical projective spectra equipped with a trivial $C_2$-action.\footnote{A similar ambiguity also arises in the motivic setting, where one chooses to work with projective spectra obtained as motivic Thom spectra of bundles over a geometric classifying space, instead of, for example, the naive projective spectrum obtained by taking a constant presheaf of classical projective spectra.}
\end{enumerate}

We proceed as follows. In Section \ref{SS:Steenrod}, we recall the $C_2$-equivariant Steenrod algebra $A^{C_2}$ and show that its subalgebras $A(n)$ are free over the cohomology of a point. These are then used in Section \ref{SS:Singer} to define and study the $C_2$-equivariant Singer construction $R_+(M)$ of any $A^{C_2}$-module $M$. We explain how Gregersen's analysis of the motivic Singer construction carries over to the $C_2$-equivariant setting to show that $R_+(M)$ is $\Tor$-equivalent to $M$. We then show in Sections \ref{SS:A3} and \ref{SS:A4} that if $M = \m_2^{C_2}$ is the $C_2$-equivariant Bredon cohomology of the $C_2$-equivariant sphere spectrum, then $R_+(M)$ is obtained as the colimit of the cohomologies of stunted projective spectra, with these spectra defined using equivariant classifying spaces or equivariant Betti realization. Finally, we invoke the $C_2$-equivariant Adams spectral sequence and the $\Tor$-equivalence mentioned above to prove our $C_2$-equivariant analog of Lin's Theorem, which says that the $C_2$-equivariant sphere spectrum is equivalent to the homotopy limit of stunted projective $C_2$-spectra. 

\subsection{The $C_2$-equivariant Steenrod algebra and some subalgebras}\label{SS:Steenrod}

We begin by recalling some properties of the category of $C_2$-equivariant spectra and the $C_2$-equivariant Steenrod algebra from \cite{HK01} and \cite{Ric14}. 

Let $RO(C_2)$ denote the ring of real orthogonal representations of $C_2$. Then there is an isomorphism
$$RO(C_2) \cong \z \{ \epsilon \} \oplus \z\{\sigma\}$$
where $\epsilon$ is the trivial $1$-dimensional representation and $\sigma$ is the sign representation. We will use $RO(C_2)$-graded groups in the sequel with the convention that $|\epsilon| = (1,0)$ and $|\sigma| = (1,1)$. The regular representation $\rho$ has bidegree $(2,1)$. 

If $X$ and $Y$ are $C_2$-spectra, then equivariant stable homotopy classes of maps from $X$ to $Y$ will be denoted $[X,Y]$. Maps between $C_2$-spectra form an $RO(C_2)$-graded Mackey functor in which restriction is induced by the projection ${C_2}_+ \to S^0$ and transfer is induced by its Spanier--Whitehead dual $S^0 \to {C_2}_+$. 

Let $\underline{\z}$ and $\underline{\f_2}$ denote the constant $C_2$-Mackey functors, i.e. the $C_2$-Mackey functors whose values on any finite $C_2$-set are always $\z$ or $\f_2$, respectively, whose restriction maps are always the identity homomorphism, and whose transfers are always given by multiplication by $2$. There are Eilenberg-MacLane $C_2$-spectra $HM$ for any Mackey functor $M$. 

\begin{prop}\cite[Prop. 6.2]{HK01}(see also \cite[Prop. 2.13]{Ric14})
The Bredon cohomology of a point with coefficients in $\mft$ is given by
$$\m_2^{C_2} \cong \f_2[\tau, \rho] \oplus \bigoplus_{s \geq 0} \dfrac{\f_2[\tau]}{\tau^\infty}\left\{ \dfrac{\gamma}{\rho^s} \right\}$$
where $|\tau| = (0,1)$, $|\rho| = (1,1)$, $|\gamma| = (0,-1)$, and following \cite[Rmk. 2.1]{GHIR17}, $\frac{\f_2[\tau]}{\tau^\infty}$ is the infinitely divisible $\f_2[\tau]$-module consisting of elements of the form $\frac{x}{\tau^k}$ for $k \geq 1$. 
\end{prop}

The cooperations $A^{C_2}_{**} = \pi_{**}^{C_2}(H\mft \wedge H\mft)$ form an $RO(C_2)$-graded Hopf algebroid called the \emph{$C_2$-equivariant dual Steenrod algebra}.

\begin{thm}\cite[Thm. 6.41]{HK01}
The $C_2$-equivariant dual Steenrod algebra has presentation
$$A^{C_2}_{**} = \pi_{**}^{C_2}(H \mft_{**} \wedge H\mft) = \m^{C_2}_2[ \xi_{i+1}, \tau_i : i \geq 0] / I$$
where $|\xi_{i}| = (2^{i+1}-2,2^i-1)$, $|\tau_i| = (2^{i+1}-1,2^i-1)$, and $I$ is the ideal generated by the relation $\tau^2_i = a \xi_{i+1} + (a \tau_0 + u) \tau_{i+1}$. 
\end{thm}

The dual of $A^{C_2}_**$ is the \emph{$C_2$-equivariant Steenrod algebra} $A^{C_2} =  \pi_{**}^{C_2}(F(H\mft,H\mft))$ which is generated over $\m_2^{C_2}$ by the equivariant Steenrod operations $Sq^i$, $i \geq 0$, modulo equivariant Adem relations (c.f. \cite[Prop. 3.1.8]{Gre12}). 

We note that the $C_2$-equivariant dual Steenrod algebra $A^{C_2}_{**}$ is isomorphic to the $\r$-motivic dual Steenrod algebra $A^{\r}$ base-changed along $\m_2^\r \hookrightarrow \m_2^{C_2}$. 

We have now recalled enough background to discuss subalgebras of the $C_2$-equivariant Steenrod algebra and its dual. This discussion was initiated by Ricka \cite{Ric14}, who translated certain profile function techniques of Margolis \cite{Mar11} to the $C_2$-equivariant setting in order to study certain quotient Hopf algebroids of $A^{C_2}_{**}$. We closely follow \cite[Sec. 5]{Ric14} for the remainder of the subsection. 

\begin{defin}
A \emph{profile function} is a pair of maps $(h,k)$,
$$h : \n \setminus \{0\} \to \n \cup \{\infty\} \quad \text{and}\quad k : \n \to \n \cup \{ \infty \}.$$
Given a profile function, denote by $I(h,k)$ the two-sided ideal of $A^{C_2}_{**}$ generated by $\xi^{2^{h(i)}}_i$ and $\tau^{2^{k(i)}}_i$ with the convention that $x^\infty = 0$. 

A profile function is \emph{minimal} if for all $i,n \geq 0$, $\tau^{2^n}_i \in I(h,k)$ if and only if $n \geq k(i)$.
\end{defin}

The following proposition provides a condition for when $A^{C_2}_{**}/ I(h,n)$ is a quotient Hopf algebroid of $A^{C_2}_{**}$. 

\begin{prop}\cite[Proposition 5.13]{Ric14} Let $(h,k)$ be a minimal profile function. Then $I(h,k)$ is a $RO(C_2)$-graded Hopf algebroid ideal if and only if the profile functions satisfy:
$$ \forall i,j \geq 1, \quad  h(i) \leq j + h(i+j) \quad or \quad h(j) \leq h(i+j),$$
$$\forall i \geq 1, j \geq 0, \quad h(i) \leq j + k(i+j) \quad or \quad k(j) \leq k(i+j).$$
\end{prop}

Additional conditions on the profile function $(h,k)$ can be imposed to guarantee the quotient Hopf algebroid $A^{C_2}_{**}/I(h,k)$ has nice properties. 

\begin{defin}
A profile function $(h,k)$ is \emph{free} if for all $i \geq 0$, $m \geq k(i)$, and $j \leq m$, 
$$k(i+m) =0 \quad \text{and} \quad h(i+j) \leq m-j.$$
\end{defin}

\begin{prop}\cite[Proposition 5.16]{Ric14} If $(h,k)$ is a free profile function, then $A^{C_2}_{**}/I(h,k)$ is free as a left $\m^{C_2}$-module.
\end{prop}

We can now apply the theory of profile functions to define some useful quotients of the $C_2$-equivariant dual Steenrod algebra. The following two examples are $C_2$-equivariant analogs of ideals appearing in Gregersen's work on the motivic Singer construction \cite{Gre12} which we will adapt to the $C_2$-equivariant setting in the sequel. 

\begin{exm}
Let $n \geq 0$. The $C_2$-equivariant analog of the ideal $I(n)$ from \cite[Definition 3.2.1]{Gre12} is defined by the profile function $(h,k)$ where
$$h = (n,n-1,n-3,\ldots,2,1,0,0,\ldots) \quad \text{and} \quad k = (\infty,\infty,\ldots,\infty,0,0\ldots),$$
where the first zero occurs as $k(n+1)$. One can verify using the previous propositions that this profile function is minimal and free, and that the corresponding quotient is a quotient Hopf algebra. We define $I(n) := I(h,k)$, and set
$$A^\vee(n) := A^{C_2}_{**} / I(n).$$
It follows from \cite[Theorem 6.15]{Ric14} that after dualizing with respect to $\m^{C_2}$, $A^{C_2}_{**}$ is cofree as a comodule over $A^\vee(n)$.
\end{exm}

\begin{exm}
Let $n \geq 0$. The $C_2$-equivariant analog of the ideal $J(n)$ from \cite[Definition 3.2.4]{Gre12} is defined by the profile function $(h',k)$ where
$$h' = (\infty,n-1,n-3,\ldots,2,1,0,0,\ldots) \quad \text{and} \quad k = (\infty,\infty,\ldots,\infty,0,0\ldots),$$
where the first zero in $k$ is $k(n+1)$. This profile function is minimal and free, but it does not follow from the above proposition that the resulting quotient is a quotient Hopf algebra. Following Gregersen, we define $J(n) := I(h',k)$, and let
$$C^\vee(n) := A^{C_2}_{**} / J(n) \quad \text{and} \quad B^\vee(n) := C(n)^\vee[\xi^{-1}_1].$$
\end{exm}

We will need to know that $B^\vee(n)$ is a left $A^\vee(n)$-comodule and a right $A^\vee(n-1)$-comodule in the sequel. This can be proven using Gregersen's results and proofs from \cite[Def. 3.2.4]{Gre12} to \cite[Def. 3.2.10]{Gre12}, which carry over almost without change to the $C_2$-equivariant setting since the coproducts in the motivic Steenrod algebra and the $C_2$-equivariant Steenrod have identical forms. The only necessary modifications are replacing tensor products over $\m_p$ (Gregersen's notation for $\m^F_2$) with tensor products over $\m^{C_2}_2$. 

\begin{defin}
We define the following subalgebras of the $C_2$-equivariant Steenrod algebra $A^{C_2}$ as follows:
$$A(n) := Hom_{\m^{C_2}_2}(A^\vee(n),\m^{C_2}_2), \quad B(n) := Hom_{\m^{C_2}_2}(B^\vee(n),\m^{C_2}_2), \quad C(n) := Hom_{\m^{C_2}_2}(C^\vee(n),\m^{C_2}_2).$$
\end{defin}

We remarked above that $B^{\vee}(n)$ is a left $A^\vee(n)$-comodule and a right $A^\vee(n-1)$-comodule. With the above definitions, we can ask about the dual properties (i.e. module structures) for $B(n)$ which will be needed to define the $C_2$-equivariant Singer construction in the sequel. Again, Gregersen has considered this question in the motivic situation, and the discussion in \cite[Def. 3.2.10]{Gre12} to \cite[Lem. 3.2.15]{Gre12} carries over \emph{mutatis mutandis} to the $C_2$-equivariant setting. The essential points are summarized in the following proposition which combines \cite[Lem. 3.2.13]{Gre12} and \cite[Lem. 3.2.14]{Gre12}. 

\begin{prop}
The subalgebra $B(n)$ is a free left $A(n)$-module with generators
$$\{Sq^{2k} : k \in \z, \quad 2^n | k\}.$$
Moreover, $B(n)$ is a free right $A(n-1)$-module with generators
$$\{Sq^{2k}, Sq^{2k+1} : k \in \z, \quad 2^n | k\}.$$
\end{prop}

By cofreeness of $A^\vee(n)$, we see that $A^{C_2}$ is free as a module over $A(n)$ for each $n$. 

\subsection{The $C_2$-equivariant Singer construction}\label{SS:Singer}

We are now ready to define the $C_2$-equivariant Singer construction, which is the essential algebraic construction needed to prove our $C_2$-equivariant analog of Lin's Theorem. 

\begin{defin}
Let $M$ be a left $A^{C_2}$-module. The \emph{$C_2$-equivariant Singer construction of $M$} is defined by
$$R^{C_2}_+(M) := \colim_{n \to \infty} B(n) \otimes_{A(n-1)} M.$$
\end{defin}

The composites
$$B(n) \otimes_{A(n-1)} M \to A \otimes_{A(n-1)} M \to A \otimes_A M \cong M$$
are compatible with increasing $n$, so taking the colimit over $n$ defines a map
$$\epsilon : R^{C_2}_+(M) \to M$$
from the $C_2$-equivariant Singer construction of $M$ to $M$.

\begin{defin}
A map of $A^{C_2}$-modules $M \to N$ is a \emph{$\Tor$-equivalence} if the induced map
$$\Tor^{A^{C_2}}_{i,j,k}(\m^{C_2}_2, M) \to \Tor^{A^{C_2}}_{i,j,k}(\m^{C_2}_2, N)$$
is an isomorphism for all $i \in \z$ and bigrading $(j,k)$ of elements in $RO(C_2)$.
\end{defin}

Note that we work with $RO(C_2)$-graded $\Tor$, valued in abelian groups, instead of $\underline{\Tor}$, which is valued in Mackey functors. We want to show that $\epsilon : R^{C_2}_+(M) \to M$ is a $\Tor$-equivalence. To do so, we need the following lemma. 

\begin{lem}
If $M$ is a free $A^{C_2}$-module, then $R^{C_2}_+(M)$ is flat as an $A^{C_2}$-module and
$$\m^{C_2}_2 \otimes_{A^{C_2}} R^{C_2}_+(M) \xrightarrow{id \otimes \epsilon} \m^{C_2}_2 \otimes_{A^{C_2}} M$$
is an isomorphism.
\end{lem}

\begin{proof}
The proof of flatness is identical to the first part of the proof of \cite[Lemma 3.3.5]{Gre12}. 

The second claim is also follows as in the proof of \cite[Lemma 3.3.5]{Gre12}, although we no longer have vanishing in the same range for $\m^{C_2}_2$. However,  inspection of \cite[Proposition 2.13]{Ric14} reveals that $\m^{C_2}_2$ vanishes in the correct degrees for the proof to carry through.
\end{proof}

Taking a free resolution of $M$ and applying the $C_2$-equivariant Singer construction at each level allows us to compare resolutions of $M$ and resolutions of $R^{C_2}_+(M)$. Using the explicit description of $B(n)$ as an $A(n)$-module, one can study the maps in the $C_2$-equivariant Singer construction as in the proof of \cite[Lem. 3.3.5]{Gre12}. It then follows from homological algebra that if $F$ is a free $A^{C_2}$-module, then the map
$$\m_2^{C_2} \otimes_{A^{C_2}} R^{C_2}_+(F) \xrightarrow{id \otimes \epsilon} \m_2^{C_2} \otimes_{A^{C_2}} F$$
is an isomorphism. Therefore we obtain isomorphisms at each stage of the free resolutions we were comparing and we have proven the following:

\begin{thm}\label{algc2singer}
Let $M$ be an $A^{C_2}$-module. Then the map $\epsilon : R^{C_2}_+(M) \to M$ is a $\Tor$-equivalence, i.e. the induced map
$$\Tor^{A^{C_2}}_{*,*,*}(\m_2^{C_2},R_+^{C_2}(M)) \to \Tor^{A^{C_2}}_{*,*,*}(\mft^{C_2}_{2},M)$$
is an isomorphism.
\end{thm}

\subsection{Homotopical realization via equivariant Thom spectra}\label{SS:A3}

Our goal in the next two subsections is to provide a homotopical realization of the $C_2$-equivariant Singer construction. That is, our goal is to construct a $C_2$-spectrum $X$ with $H^{**}(X) \cong R_+^{C_2}(\m_2^{C_2})$. This spectrum will come with a map $S^{0,0} \to \Sigma^{1,0} X$ which induces the map in cohomology $\epsilon: R_+^{C_2}(\m_2^{C_2}) \to \m_2^{C_2}$ of Theorem \ref{algc2singer}. The $\Tor$ groups in Theorem \ref{algc2singer} are the $E_2$-pages of $C_2$-equivariant Adams spectral sequences, so we will obtain an isomorphism of spectral sequences from $E_2$-terms and thus an equivalence of $C_2$-spectra after $2$-completion. 

We now construct the $C_2$-spectrum $X$. Our disucssion closely follows Gregersen's discussion from the motivic setting \cite[Sec. 4]{Gre12}. 

\begin{defin}
Let $Q^n$ denote the (non-equivariant) $2n$-skeleton of the $C_2$-equivariant classifying space of $\mu_2$, $B_{C_2}\mu_2$.\footnote{Throughout this section, $C_2$ denotes the cyclic group of order two which our objects are genuinely equivaraint with respect to, while $\mu_2$ denotes a cyclic group of order two which acts on an object.} Let $Q^\infty := B_{C_2}\mu_2 =\colim_{n \to \infty} Q^n$.
\end{defin}

\begin{lem} \cite[Theorem 6.22]{HK01}
There are ring isomorphisms
$$H^{**}_{C_2}(Q^n) \cong \m^{C_2}_2[x,y]/(x^2 + \rho x + \tau y, y^n),$$
$$H^{**}_{C_2}(Q^\infty) \cong \m^{C_2}_2[x,y]/(x^2 + \rho x + \tau y)$$
where $|x| = (1,1)$ and $|y| = (2,1)$. 
\end{lem}

By construction, $B_{C_2}\mu_2$ is equipped with a tautological $C_2$-equivariant line bundle $\gamma^1$. Let $\gamma^1_{n-1}$ denote the restriction of the tautological line bundle over $B_{C_2}\mu_2$ to $Q^n$. 

\begin{defin}
Define
$$\underline{Q}^{n-k}_{-k} := \Sigma^{-2kn,-kn} Th(k\gamma^1_{n-1} \to Q^n).$$
\end{defin}

The proof of \cite[Proposition 4.1.24]{Gre12} translates to the $C_2$-equivariant setting without change to give the following calculation.

\begin{lem}
As modules over $\m^{C_2}_2$, there is an isomorphism
$$H^{**}_{C_2}(\underline{Q}^{n-k}_{-k}) \cong \Sigma^{-2k,-k}\m^{C_2}_2[x,y]/(x^2 + \rho x + \tau y, y^n).$$
\end{lem}

As in the classical and motivic settings, the top and bottom dimensions $n$ and $k$ can be varied compatibly in the previous definition. This allows us to define the following:

\begin{defin}
Define
$$\underline{Q}^\infty_{-k} := \colim_{n \to \infty} \underline{Q}^{n-k}_{-k} \quad \text{and} \quad \underline{Q}^\infty_{-\infty} := \lim_{k \to \infty} \underline{Q}^{\infty}_{-k}.$$
\end{defin}

The discussion from \cite[Sec. 4.2]{Gre12} also carries over without change to the $C_2$-equivariant setting to prove the following. 

\begin{prop}
As modules over $\m^{C_2}_2$, we have
$$H^{**}_{C_2} (\underline{Q}^\infty_{-k}) \cong \Sigma^{-2k,-k} \m^{C_2}_2[x,y]/(x^2 + ax + uy),$$
$$H^{**}_{C_2,c}(\underline{Q}^\infty_{-\infty}) := \colim_{k \to \infty} H^{**}_{C_2}(\underline{Q}^\infty_{-k}) \cong \m^{C_2}_2[x,y,y^{-1}]/(x^2 + ax + uy).$$

The $A^{C_2}$-module structure of $H^{**}(\underline{Q}^\infty_0)$ is given by
$$Sq^{2i}(y^je_0) = {2j \choose 2i} y^{j+i}e_0, \quad Sq^{2i+1}(y^je_0) = 0,$$
$$Sq^{2i}(xy^je_0) = {2j \choose 2i} xy^{j+i}e_0,\quad Sq^{2i+1}(xy^je_0) = {2j \choose 2i}y^{j+i+1}e_0.$$

The $A^{C_2}$-module structure of $H^{**}(\underline{Q}^\infty_k)$ for $k<0$ coincides with the $A^{C_2}$-module structure obtained by periodically extending the $A^{C_2}$-module structure of $H^{**}(\underline{Q}^\infty_0)$ to negative dimensions. The continuous $A^{C_2}$-module structure of $H^{**}_{C_2,c}(\underline{Q}^\infty_{-\infty})$ is determined similarly. 
\end{prop}

The proof of \cite[Prop. 4.1.37]{Gre12} now can be applied to show that $\underline{Q}^\infty_{-\infty}$ is a homotopical realization of the $C_2$-equivariant Singer construction of $\m_2^{C_2}$:

\begin{prop}
There is an isomorphism of $A^{C_2}$-modules
$$R_+^{C_2}(\m_2^{C_2}) \cong \Sigma^{1,0} H^{**}_{C_2,c}(\underline{Q}^\infty_{-\infty}).$$
\end{prop}

Therefore by Theorem \ref{algc2singer}, we obtain an isomorphism 
$$\Tor^{A^{C_2}}_{*,*,*}(\mft^{C_2}_{**}, \Sigma^{1,0}H^{**}_{C_2}(\underline{Q}^\infty_{-\infty})) \cong \Tor^{A^{C_2}}_{*,*,*}(\mft^{C_2}_{**}, \mft^{C_2}_{**}).$$
The left-hand side is the $E_2$-page of the inverse limit of $RO(C_2)$-graded Adams spectral sequence converging to $\pi_{**}(\Sigma^{1,0} \underline{Q}^\infty_{\infty})$ and the right-hand side is the $E_2$-page of the $RO(C_2)$-graded Adams spectral sequence converging to $\pi_{**}(S^{0,0})$. Since this isomorphism is induced by a map of $C_2$-spectra $S^{0,0} \to \Sigma^{1,0} \underline{Q}^\infty_{-\infty}$, it fits into an isomorphism of spectral seuqences. We therefore obtain the desired $C_2$-equivariant analog of Lin's Theorem. 

\begin{thm}\label{Thm:Qui}
After $2$-completion, the inclusion of the $(0,0)$-cell
$$S^{0,0} \to \underline{Q}^\infty_{-\infty}$$
is an equivalence of $C_2$-spectra. 
\end{thm}

\subsection{Homotopical realization via equivariant Betti realization}\label{SS:A4}

In the previous subsection, we proved a $C_2$-equivariant analog of Lin's Theorem where our analog of stunted projective spectra were constructed as Thom spectra of virtual bundles over a $C_2$-equivariant classifying space. We now present an alternative motivic-to-equivariant construction.

In \cite[Sec. 3.3]{MV99}, Morel and Voevodsky defined Betti realization functors
$$Re_B : \Motc \to \Spt \quad \text{and} \quad Re_{C_2} : \Motr \to \Sptc$$
which are induced by the functor which sends a scheme over $Spec(\c)$ or $Spec(\r)$ to its $\c$-points. For the first realization functor, they note that
$$Re_B(S^{1,0}) \cong Re_B(S^{1,1}) \cong S^1 \quad \text{and} \quad Re_B(BG) \cong B(G(\c))$$
for any smooth group scheme over $\c$. The non-equivariant Betti realization functor was studied further in \cite[Section 5]{Lev14}. In \cite{HO16}, Heller and Ormsby study the $C_2$-equivariant Betti realization functor. We recall some of their results below.

\begin{prop}\cite[Proposition 4.8]{HO16} There is a Quillen adjoint pair
$$Re_{C_2} : \Motr \leftrightarrows Sp^{C_2} : Sing^{C_2}_B.$$
Moreover, $Re_{C_2}$ is strong symmetric monoidal.
\end{prop}

Heller and Ormsby go on to identify the image of equivariant Betti realization for $\p^1$-suspension spectra. We restrict to the case where the ground field is $\r$ below, but their result holds for any field with a real embedding.

\begin{prop}\label{rss}\cite[Lemma 4.14]{HO16} For any $\r$-scheme $X$, the natural map
$$\l Re_{C_2}(\Sigma^\infty X_+) \to \Sigma^\infty X(\c)^{an}_+$$
is an isomorphism in $SH_{C_2}$. Here $\l(-)$ denotes the left-derived functor of the left Quillen functor $Re_{C_2}$, $\Sigma^\infty$ is the $\p^1$-suspension spectrum functor on the left-hand side and the $S^\rho$-suspension spectrum functor on the right-hand side, and $(-)^{an}$ indicates that we equip $X(\c)$ with the analytic topology. 
\end{prop}

Analyzing symmetric powers and using equivariant and motivic analogs of the Dold-Thom theorem proves the following. 

\begin{thm}\label{coh}\cite[Theorem 4.17]{HO16} Let $\Lambda$ be an abelian group. There is an isomorphism in $SH_{C_2}$
$$\l Re_{C_2}(H \Lambda) \cong H \underline{\Lambda}$$.
\end{thm}

Our goal is to combine these results with some results of Gregersen to show that the equivariant Betti realizations of the stunted $\r$-motivic lens spaces $\underline{L}^\infty_{-n}$ provide a topological realization of the $C_2$-equivariant Singer construction. We will need the following result.

\begin{lem}\cite[Lemma 4.1.2]{Gre12} The motivic spaces $L^n := (\a^n \setminus 0)/\mu_2$ are represented by smooth schemes, where $\mu_2$ acts on $\a^n \setminus 0$ by inversion. 
\end{lem}

We will denote the equivariant Betti realization of $L^n$ by $R^n$. By Proposition \ref{rss}, we have
$$R^n = \l Re_{C_2}(L^n) \simeq L^n(\c)^{an}$$
where $(-)^{an}$ indicates that we are using the analytic topology. Define $R^\infty := \colim_{n \to \infty} R^n$. 

\begin{lem}
There is an equivalence in $SH_{C_2}$
$$L^n(\c)^{an} \simeq Q^n.$$
\end{lem}

\begin{proof}
The $C_2$-equivariant spaces $Q^n$ are defined by analogy with the definition of the motivic spaces $L^n$ in \cite[Section 6]{Voe03}. Since $\a^1_\r(\c) \cong \c$ and equivariant Betti realization is compatible with taking skeleta, we see that the equivariant Betti realization of $L^n$ is $Q^n$. 
\end{proof}

We now define another model for $C_2$-equivariant stunted projective spaces. Recall that for $n>0$, Gregersen defines
$$\underline{L}^{n-k}_{-k} := \Sigma^{-kn}_{\p^1} Th(k\epsilon^n, k\gamma^1_{n-1}),$$
where if $\eta \hookrightarrow \xi$ is an inclusion of vector bundles, then 
$$Th(\xi, \eta) := \dfrac{E(\xi)}{E(\xi) \setminus E(\eta)}.$$
Define $\underline{R}^{n-k}_{-k} := \l Re_{C_2}(\underline{L}^{n-k}_{-k})$. Since left derived functors preserve homotopy colimits, we have
\begin{align*}
\underline{R}^{n-k}_{-k} &:= \l Re_{C_2}(\underline{L}^{n-k}_{-k}) \\
&= \l Re_{C_2}(\Sigma^{-kn}_{\p^1} Th(k\epsilon^n, k\gamma^1_{n-1})) \\
&= \l Re_{C_2}((\p^1)^{-kn} \wedge Th(\l Re_{C_2}(k\epsilon^n), \l Re_{C_2}(k\gamma^1_{n-1})) \\ 
&= \Sigma^{-2kn,-kn} Th(k \epsilon^n, k\gamma^1_{n-1})
\end{align*}
where in the bottom line, $\epsilon$ and $\gamma^1_{n-1}$ denote the trivial and tautological Real vector bundles studied in \cite{HK01}. We can compute the Bredon cohomology of these $C_2$-spectra from the motivic cohomology $H^{**}(\underline{L}^{n-k}_{-k})$, which can be found in \cite[Proposition 4.1.24]{Gre12}.

\begin{lem}As modules over $\m^{C_2}_2$, there is an isomorphism
$$H^{**}_{C_2}(\underline{R}^{n-k}_{-k}) \cong \Sigma^{-2k,-k}\m^{C_2}_2[x,y]/(x^2 + ax + uy, y^n).$$
\end{lem}

\begin{proof}
The proof is similar to the proof of the previous lemma.
\end{proof}

We can now give a second definition of the $C_2$-equivariant analog of $P^\infty_{-k}$.

\begin{defin}\label{ebrdef}
Define
$$\underline{R}^\infty_{-k} := Re_{C_2}(\underline{L}^\infty_{-k}) \quad \text{and} \quad \underline{R}^\infty_{-\infty} := \lim_{k\to \infty} \underline{R}^\infty_{-k}.$$
\end{defin}

We opt to define $\underline{R}^\infty_{-\infty}$ as the above homotopy limit, instead of as the equivariant Betti realization of $\underline{L}^\infty_{-\infty}$, since it is not \emph{a priori} clear that equivariant Betti realization preserves homotopy limits. However, we will see below that in this case, these choices produce equivalent spectra.

Similar arguments to the above using equivariant Betti realization, the fact that left derived functors commute with homotopy colimits, and \cite[Proposition 4.1.32]{Gre12} allow us to compute the Bredon cohomology of these spectra.

\begin{prop}
As modules over $\m^{C_2}_2$, we have
$$H^{**}_{C_2} (\underline{R}^\infty_{-k}) \cong \Sigma^{-2k,-k} \m^{C_2}_2[x,y]/(x^2 + ax + uy),$$
$$H^{**}_{C_2, c}(\underline{R}^\infty_{-\infty}) \cong \m^{C_2}_2[x,y,y^{-1}]/(x^2 + ax + uy).$$
\end{prop}

So far, we have not said anything about the $A^{C_2}$-module structure of $H^{**}_{C_2}(\underline{R}^\infty_{-k})$. We address this point in the following proposition. Intuitively, this proposition holds since equivariant Betti realization takes the cell in the motivic classifying space $B_{gm}\mu_2$ carrying $Sq^i \in A^\r$ to the cell in the $C_2$-equivariant classifying space $B_{C_2}\z/2$ carrying $Sq^i \in A^{C_2}$. 

\begin{prop}
Equivariant Betti realization preserves module structure over the Steenrod algebra for the motivic cohomology of schemes over $\r$. In particular, the action of $Sq^i \in A^{C_2}$ on $x,y \in H^{**}_{C_2}(\underline{R}^\infty_{-k})$ is identical to the action of $Sq^i \in A^{\r}$ on $u,v \in H^{**}_\r(\underline{L}^\infty_{-k})$. 
\end{prop}
\begin{proof}
We will prove the dual statement, which says that equivariant Betti realization preserves comodules structures over the relevant dual Steenrod algebras. Consider the $A^\r_*$-coaction structure map
$$\nu : H^{\r}_{**}(X) \to H^{\r}_{**}(X) \otimes (A^\r_{**})^\vee.$$
This structure map is the induced map in homotopy for the map
$$S^{0,0} \wedge  X \wedge H\f_2^\r \to H\f_2^\r \wedge (X \wedge H\f_2^\r).$$
Applying equivariant Betti realization to both sides and using the fact that left derived functors commute with smash product (a homotopy colimit) produces the diagram
$$Re_{C_2}(\nu) : H\mft \wedge X(\c) \to H\mft \wedge (H\mft \wedge X(\c)).$$
Applying $\pi^{C_2}_{**}$ gives the $(A^{C_2})^\vee$-coaction on $H\mft_{**}(X(\c))$. 
\end{proof}

This strengthens the isomorphism above to give the following.

\begin{cor}
There is an isomorphism of $A^{C_2}$-modules
$$\Sigma^{1,0} H^{**}_{C_2,c}(\underline{R}^\infty_{-\infty}) \cong R_+(\m^{C_2}_2)$$
where the left-hand side is the continuous cohomology $H^{**}_{C_2,c}(\underline{R}^\infty_{-\infty}) := \colim_k H^{**}_{C_2}(\underline{R}^\infty_{-k}).$
\end{cor}

As with $\underline{Q}^\infty_{-\infty}$, we may apply this corollary and the $C_2$-equivariant Adams spectral sequence to prove the following. 

\begin{thm}\label{C2Lin}
After $2$-completion, the inclusion of the $(0,0)$-cell
$$S^{0,0} \to \Sigma^{1,0} \underline{R}^\infty_{-\infty}$$
is an equivalence of $C_2$-spectra. 
\end{thm}

\bibliographystyle{alpha}
\bibliography{master}

\end{document}